\documentclass[oneside, reqno]{amsart}
\usepackage{amsmath, amsthm, amssymb, mathrsfs}
\usepackage{geometry}
\usepackage{enumitem}
\usepackage{hyperref}

\newtheorem{theorem}{Theorem}[subsection]
\newtheorem{lemma}[theorem]{Lemma}
\newtheorem{proposition}[theorem]{Proposition}
\newtheorem{corollary}[theorem]{Corollary}
\newtheorem{conjecture}[theorem]{Conjecture}
\newtheorem{definition}[theorem]{Definition}
\newtheorem{remark}[theorem]{Remark}
\newtheorem{example}{Example}[section]

\numberwithin{equation}{section}

\title{Stationary phase analysis for analytic newvectors and application to subconvexity problems}
\author{Liyuan Ye}
\date{\today}

\begin{document}

\maketitle

\begin{abstract}
	
	In this paper, we extend the results of \cite{MV10} and \cite{HMN22} to establish an upper bound for triple product and Rankin-Selberg L-functions of the form
	$$L(\pi_1\otimes\pi_2\otimes\pi_3,\frac{1}{2})\ll_{\pi_3,\epsilon}C(\pi_1\otimes\pi_2)^{\frac{1}{2}+\epsilon}\left(\frac{C(\pi_1\otimes\pi_2)}{C(\pi_2\otimes\pi_2)}\right)^{-\delta}$$
	in the spectral aspect, allowing conductor dropping. In particular, we obtain a subconvexity bound when $\pi_1\otimes\pi_2$  stays uniformly away from QUE-like case. The new ingredient is a stationary phase analysis of the analytic newvectors introduced by Jana and Nelson in \cite{JN19}, for both  $\mathrm{PGL}_2(\mathbb{R})$ and $\mathrm{PGL}_2(\mathbb{C})$, which is applied to a test vector conjecture for local triple product periods. 

\end{abstract}

\tableofcontents

\newpage

\section{Introduction}

\subsection{Subconvexity problems for Rankin-Selberg L-functions}

The subconvexity problem is one of the central topics in the analytic theory of $L$-functions, see \cite{IS00}. Let $\mathbb{F}$ be a number field, $\mathbb{A}$ be its adele ring, and $\mathbf{G}$ be a reductive algebraic group defined over $\mathbb{F}$. For a family $ \mathcal{F}$ of automorphic representations $\Pi$ of $\mathbf{G}(\mathbb{A})$, let $L(\Pi,s)$ and $C(\Pi,s)$ denote the corresponding L-functions and analytic conductors, respectively. The subconvexity problem asks for an absolute constant $\delta>0$ such that for any $\epsilon>0$, the uniform upper bound
$$L(\Pi,s)\ll_{\mathcal{F},\epsilon} C(\Pi,s)^{\frac{1}{4}-\delta+\epsilon}$$
holds for all $\Pi\in \mathcal{F}$. 

The subconvexity problem has a long and rich history. For $\mathbf{G} = \mathrm{GL}_1,\mathrm{GL}_2$ or $\mathbf{G} = \mathrm{GL}_2\times\mathrm{GL}_1$, the problem is now considered completely resolved, thanks to the cumulative efforts of many mathematicians over a century, culminating in the definitive work of Michel and Venkatesh \cite{MV10}. As for general higher rank cases, while significant progress has been made, the results are far from a complete resolution. For a comprehensive survey of the topic, we refer the readers to \cite{Mic22}.

In this paper, we focus on the fundamental cases of $\mathbf{G}=\mathrm{GL}_2\times\mathrm{GL}_2$ and $\mathbf{G}=\mathrm{GL}_2\times\mathrm{GL}_2\times\mathrm{GL}_2$, which correspond to Rankin-Selberg L-functions and triple product L-functions, respectively. A significant amount of works (e.g., \cite{Sar01, KMV02, LY02, Mic04, HM06, LLY06, Ve10, MV10, NPS14, Hu17, Za20, BJN23}) have established subconvex bounds of the type
$$L(\pi_1\otimes\pi_2,s)\ll_{\epsilon,\pi_2,s} C(\pi_1,s)^{\frac{1}{2}-\delta+\epsilon}$$
$$L(\pi_1\otimes\pi_2\otimes\pi_3,\tfrac{1}{2})\ll_{\epsilon,\pi_2,\pi_3} C(\pi_1)^{1-\delta+\epsilon}$$
These bounds address the so-called \emph{$\pi$-twisting aspect} of the subconvexity problem, where one representation $\pi_1$ varies while the others remain fixed. We note that in this setting, the analytic conductors have the expected generic size:
$$C(\pi_1\otimes\pi_2)\asymp_{\pi_2}C(\pi_1)^2,\quad C(\pi_1\otimes\pi_2\otimes\pi_3)\asymp_{\pi_2,\pi_3}C(\pi_1)^4$$ 

However, the picture changes dramatically in the \emph{hybrid aspect}, where two or more representations vary simultaneously. This setting remains largely unexplored, primarily due to a key difficulty that arises when the varying representations are closely related—a phenomenon known as \emph{conductor dropping}. Most existing methods fail in the setting where $\pi_1$ and $\pi_2$ exhibit such dependence. In the extreme case where $\pi_2 = \tilde{\pi}_1$, the conductor drops sharply, satisfying $C(\pi_1 \otimes \tilde{\pi}_1) \asymp C(\pi_1)$, which is significantly smaller than the generic size of $C(\pi_1)^2$. This leads to one of the most challenging open problems in the field: establishing a subconvex bound for the triple product $L$-function
$$L(\pi\otimes\tilde{\pi}\otimes\pi',\tfrac{1}{2})\ll_{\epsilon,\pi'} C(\pi)^{\frac{1}{2}-\delta+\epsilon}$$
This particular challenge is often referred to as the QUE-like case (see \cite{HS10, So10, Nel25}). 

Significant progress was made in \cite{HMN22}, which addressed a hybrid level aspect subconvexity problem of triple product $L$-functions when $\pi_1\otimes\pi_2$ stays uniformly away from the QUE-like case. Their framework is general, allowing arbitrary number fields, joint ramifications and conductor dropping. By extending the methods of \cite{MV10}, they sucessfully reduced the bounds for L-functions to a local conjecture on test vectors. 

In this paper, we extend their program by resolving the local conjecture at the archimedean place. As a consequence, we obtain a new subconvex bound in the hybrid spectral aspect. The following theorem strengthens \cite[Theorem 1.1, 1.3]{HMN22} by removing the restrictions previously imposed at the archimedean places.

\begin{theorem}\label{Intro. main theorem}
	
	Let $\mathbb{F}$ be a number field and $\mathbb{A}$ be its adele ring. Let $\pi_1,\pi_2,\pi_3$ be irreducible automorphic representations of $\mathrm{PGL}_2(\mathbb{A})$. Assume that $\pi_1,\pi_2$ are cuspidal representations, while $\pi_3$ is a cuspidal or Eisenstein series representation. Suppose that $\pi_{i,v}$ (i=1,2) are principal series representations at archimedean places $v$, and that $\pi_{3,v}=\mathcal{I}(|\cdot|_v^{s_v})$ at archimedean places and are unramified at all finite places. Then there exists an absolute constant $\delta>0$, such that 
	\begin{equation}\label{thm. bounds for triple product L-functions}
		L(\pi_1\otimes\pi_2\otimes\pi_3,\tfrac{1}{2})\ll_{\epsilon,\pi_3} C(\pi_1\otimes\pi_2)^{\frac{1}{2}+\epsilon}\left(\frac{C(\pi_1\otimes\pi_2)}{C(\pi_2\otimes\pi_2)}\right)^{-\delta}.
	\end{equation}
	In particular, when $\pi_3$ is an Eisenstein series,  there exists an absolute $\delta>0$, such that 
	\begin{equation}\label{thm. bounds for Rankin-Selberg L-functions}
		L(\pi_1\otimes\pi_2,\tfrac{1}{2}+it)\ll_{\epsilon,t} C(\pi_1\otimes\pi_2)^{\frac{1}{4}+\epsilon}\left(\frac{C(\pi_1\otimes\pi_2)}{C(\pi_2\otimes\pi_2)}\right)^{-\delta}.
	\end{equation}
	
\end{theorem}

\begin{definition}
	
	Let $\pi_1,\pi_2$ be irreducible automorphic representations of $\mathrm{PGL}_2(\mathbb{A})$. We say that a family of representations $\pi_1\otimes\pi_2$ stays uniformly away from QUE-like case with exponent $1\geq \eta>0$, if 
	\begin{equation}\label{def. uniformly away from QUE-like case}
		\frac{C(\pi_1\otimes\pi_2)}{C(\pi_2\otimes\pi_2)}\gg_\eta C(\pi_1\otimes\pi_2)^{\eta}.
	\end{equation}

\end{definition}

\begin{remark}
	
	When $C(\pi_2)$ is bounded, then condition (\ref{def. uniformly away from QUE-like case}) holds with $\eta=1$, corresponding to the situation with no conductor dropping.  In contrast, when $\pi_1=\tilde{\pi}_2$ is the QUE-like case, (\ref{def. uniformly away from QUE-like case}) holds only with $\eta = 0$. Thus, the exponent $\eta$ measures the degree of conductor dropping between $\pi_1$ and $\pi_2$.
\end{remark}

\begin{corollary}\label{Intro. main corollary}
	With the notation and assumptions of Theorem \ref{Intro. main theorem}, suppose further that $\pi_1\otimes\pi_2$ stays uniformly away from QUE-like case with exponent $0<\eta\leq1$, then there exists an absolute constant $\delta=\delta(\eta)>0$, such that
	\begin{equation}\label{thm. sunbconvex bounds for triple product L-functions}
		L(\pi_1\otimes\pi_2\otimes\pi_3,\tfrac{1}{2})\ll_{\epsilon,\eta,\pi_3} C(\pi_1\otimes\pi_2)^{\frac{1}{2}-\delta+\epsilon}.
	\end{equation}
	In particular, when $\pi_3$ is an Eisenstein series, there exists an absolute $\delta>0$, such that 
	\begin{equation}\label{thm. subconvex bounds for Rankin-Selberg L-functions}
		L(\pi_1\otimes\pi_2,\tfrac{1}{2}+it)\ll_{\epsilon,\eta,t} C(\pi_1\otimes\pi_2)^{\frac{1}{4}-\delta+\epsilon}.
	\end{equation}
\end{corollary}

Corollary \ref{Intro. main corollary} provides a new hybrid spectral subconvexity bound for Rankin-Selberg L-functions, even in the most basic case when $\mathbb{F}=\mathbb{Q}$ and both $\pi_1,\pi_2$ are unramified. 

\begin{theorem}\label{Intro. main theorem for Maass forms}
	
	Let $F_1$ and $F_2$ be cusp Hecke-Maass eigenforms for $\mathrm{SL}_2(\mathbb{Z})$, with Laplace eigenvalue $\frac{1}{4}+\lambda_i^2$, $i=1,2$. Their analytic conductor is given by 
	$$C(F_1\times F_2)=\left(1+\left|\tfrac{\lambda_1+\lambda_2}{2\pi}\right|\right)^2\left(1+\left|\tfrac{\lambda_1-\lambda_2}{2\pi}\right|\right)^2.$$ 
	Let $0<\eta\leq 1$ be a fixed exponent. Then there exists an absolute constant $\delta=\delta(\eta)>0$ such that the Rankin-Selberg L-functions satisfy
	\begin{equation}
		L(F_1\times F_2,\tfrac{1}{2}+it)\ll_{\epsilon,\eta,t}C(F_1\times F_2)^{\frac{1}{4}-\delta+\epsilon},
	\end{equation}
	whenever $|\lambda_1-\lambda_2|\gg \max\{1,|\lambda_1|, |\lambda_2|\}^\eta.$
	
\end{theorem}

	In the recent work of \cite{Xu25}, the authors use moment method to give a hybrid spectral subconvexity bound when $\pi_1\otimes\pi_2$ stays away from QUE-like case with exponent $\eta\geq \tfrac{2}{3}$ and recover a Burgess-like bound when $\eta=1$. The advantage of our work is that $\eta$ can be any positive exponent, and our treatment allows a hybrid result for both spectral and level aspects over general number fields. Another recent work \cite{Gho25} studies subconvexity for $L(f\otimes g,\frac{1}{2}+it)$ at certain special points, where $f$, $g$ are cusp Hecke–Maass forms with spectral parameters $t_f,t_g\asymp T$, and  $t=\pm(t_f+t_g)$, or $t=\pm|t_f-t_g|\asymp T^{\nu}$ for $2/3+\epsilon\leq \nu\leq 1$. In particular, they treat a conductor–dropping range using a delta–method approach. In contrast, we will study the central values of these L-functions using period integral methods.

\subsection{Test vector conjecture for local triple product periods}

Let $\varphi_{i}=\otimes_v\varphi_{i,v}\in \pi_i$ be unit automorphic forms on $\mathrm{PGL}_2(\mathbb{A})$. The triple product period formula relates the global period integral of automorphic forms to special value of L-functions and the local triple product integrals
$$\left|\int_{\mathrm{PGL}_2(\mathbb{F})\backslash \mathrm{PGL}_2(\mathbb{A})}\varphi_1\varphi_2\varphi_3(g)dg\right|^2\sim\frac{L(\pi_1\otimes\pi_2\otimes\pi_3,\frac{1}{2})}{\prod_{i=1}^3 L\left(\pi_i, \mathrm{Ad}, 1\right)}\cdot \prod_{v\in S}I_v^T(\varphi_{1,v},\varphi_{2,v},\varphi_{3,v}),$$
where $S$ contains the archimedean places and ramified places, and
$$I^T_v(\varphi_{1,v},\varphi_{2,v},\varphi_{3,v}):=\int_{\mathrm{PGL}_2(F)}\prod_{i=1}^{3}\langle \pi_i(g)\varphi_{i,v},\varphi_{i,v}\rangle_{\pi_{i,v}} dg.$$
In the period approach to the subconvexity problem, one needs to bound the local triple product integrals from below and the global period from above.

The period method is now a well-established strategy, as developed and applied in a series of works \cite{BR10, Ve10, MV10, Hu17, Nel19, Za19, Za20, HMN22, BJN23}. Within the general framework of \cite{MV10}, the recent work \cite{HMN22} reduces the global upper bound (\ref{thm. bounds for triple product L-functions}) to a local conjecture on the existence of certain test vectors.

\begin{conjecture}[Test vector problem for local trilinear forms]\label{conj. local test vector conjecture}
	Let $F=\mathbb{F}_v$, $G=\mathrm{PGL}_2(F)$ and let $\pi_1,\pi_2,\pi_3$ be unitary representations of $G$. Assume that $\pi_3$ is fixed, and that $C(\pi_2)\geq C(\pi_3)=M$. Then there exists unit vectors $\varphi_i\in \pi_i$ satisfying the following:
	\begin{enumerate}[label=\textnormal{(\roman*)}]
		\item The Sobolev norms of $\varphi_3$  are controlled by $M$, i.e. there exist an absolute constant $B$ such that $\varphi_3$ is $K(M^{B})$-invariant if $F$ is non-archimedean, and there exists some $B(d)$ depending only on $d$, such that  $\mathcal{S}_d(\varphi_3)\ll_d M^{B(d)}$ if $F$ is archimedean. 
		\item For some constant $A$,
		\begin{equation}\label{conj. lower bound}
			I^T(\varphi_{1},\varphi_{2},\varphi_{3})\gg_{\epsilon} C(\pi_1\otimes\pi_2\otimes\pi_3)^{-\frac{1}{4}-\epsilon}M^A.
		\end{equation}
		\item There exists $A'$ such that for every $\psi\in\pi'$ with Sobolev norms controlled by $M$,
		\begin{equation}\label{conj. upper bound}
			I^T(\varphi_2,\varphi_2, \psi)\ll_{\epsilon}M^{A'}C(\pi_1\otimes\pi_2\otimes\pi_3)^{-\frac{1}{4}+\epsilon}\left(\frac{C(\pi_1\otimes\pi_2\otimes\pi_3)^{\frac{1}{2}}}{C(\pi_2\otimes\pi_2)}\right)^{\vartheta}.
		\end{equation}
		
	\end{enumerate}
	Here $\vartheta<\frac{1}{2}$ is any bound towards the  Ramanujan conjecture. 
\end{conjecture}

	When two of the three representations are fixed, the existence of suitable test vectors is known from \cite{MV10}. We also note \cite{BJN23}, where the lower bound (\ref{conj. lower bound}) was obtained using analytic newvectors. When two of the three representations are varying, \cite{HMN22} proved Conjecture \ref{conj. local test vector conjecture} for non-archimedean fields when $\pi_3$ is unramified. In their construction, the vectors $\varphi_3$ are taken to be the p-adic newforms, while $\varphi_1$ and $\varphi_2$ are diagonal translations of newforms.  In this work we further develop the theory of analytic newvectors, which allows us to extend the  results of \cite{HMN22} to archimedean case. In particular, we obtain the following theorem.

\begin{theorem}\label{Intro. local test vector problem}
	Let $F=\mathbb{R}$ or $\mathbb{C}$. Let $\pi_1,\pi_2,\pi_3$ be unitary principal series representations, with $\vartheta_1+\vartheta_2+\vartheta_3<\frac{1}{2}$ and let $\pi_3=\mathcal{I}(|\cdot|_F^{s})$ be fixed.
	Set $|Q|_F = C(\pi_1 \otimes \pi_2)^{1/2}$. Then there exist unit vectors $\varphi_i \in \pi_i$ such that $\varphi_3 = a(Q)\varphi_3^0$ are diagonal translates of a fixed $\varphi_3^0 \in \pi_3$, and 
	\begin{enumerate}[label=\textnormal{(\roman*)}]
		\item the lower bound holds:
		\begin{equation}\label{Intro. lower bound}
			I^{T}(\varphi_1,\varphi_2,\varphi_3)\gg_{\epsilon,\pi_3} C(\pi_1\otimes\pi_2)^{-\frac{1}{2}-\epsilon},
		\end{equation}
		\item there exists $d_0\geq 0$ such that for any other unitary representations $\pi$ and any $\psi\in \pi$, 
		\begin{equation}\label{Intro. upper bound}
			I^T(\varphi_2,\varphi_2,a(Q)\psi)\ll_{\epsilon,\pi_3} \mathcal{S}_{d_0}^\pi(\psi)^2C(\pi_1\otimes\pi_2)^{-\frac{1}{2}}\left(\frac{C(\pi_1\otimes\pi_2)}{C(\pi_2\otimes\pi_2)}\right)^{\vartheta+\epsilon}.
		\end{equation}
		
	\end{enumerate}

\end{theorem}

\subsection{Analytic newvectors}

The theory of analytic newvectors for unitary representation for $\mathrm{PGL}_n(\mathbb{R})$ was introduced and developed by Jana and Nelson in \cite{JN19}, and found applications in the series of work \cite{Ja21,Ja22,BJN23}. A key feature of this theory is that the analytic behavior of these newvectors encodes information about the analytic conductor of the representations. One characterization is that analytic newvectors are almost invariant under the action of an open neighborhood of the identity, given by
$$
K(C(\pi),\tau)=\left\{\begin{pmatrix}
	a & b \\
	c & d
\end{pmatrix} \in \mathrm{GL}_{n+1}(\mathbb{R}): \begin{array}{cc}
	|a-1_n|<\tau, & |b|<\tau, \\
	|c|<\frac{\tau}{C(\pi)}, & |d-1|<\tau
\end{array}\right\} .
$$
In this paper, we refine their work by choosing the analytic newvectors from a slightly different construction. By means of the stationary phase method, we obtain a more precise description of their Whittaker functions for both $\mathrm{PGL}_2(\mathbb{R})$ and $\mathrm{PGL}_2(\mathbb{C})$. 

We briefly outline the main features of the Whittaker functions associated with analytic newvectors.  Let  $F=\mathbb{R}$ or $\mathbb{C}$, and let $\chi$ be a unitary character of $F^\times$. Write $T=T(\chi)$ for the spectral parameter of $\chi$ as in (\ref{def. spectral parameters}) and  $C(\chi)$  for the conductor as in (\ref{def. L-factor and conductor of GL1}), so that $C(\chi)\asymp 1+|T|_F$. Let $\mathcal{I}(\chi)$ be the tempered principal series representations of $\mathrm{PGL}_2(F)$ unitarily induced from $\chi$. In \S 3, we prove that one can pick a family of analytic newvectors $f_\chi\in \mathcal{I}(\chi)$ with Whittaker functions $W_{f_\chi}$, and a constant $0<b<1$ depends only on this family, such that
	\begin{itemize}
	\item Symmetry: The analytic newvectors are equivariant under the archimedean Atkin-Lehner operator $w_T=\tiny\begin{pmatrix}
		0 & 1 \\ -\frac{1}{T^2} & 0
	\end{pmatrix}$, in the sense that 
	\begin{equation}\label{intro. symmetry}
		W_{f_\chi}(gw_T)=\chi(-1) W_{f_\chi}(g);
	\end{equation}
	\item Asymptotic formula: If $|z|\leq b$, the lower-triangular translate of Whittaker functions exhibit the following asymptotic behavior:
	\begin{equation}\label{intro. asymptotic}
		W_{f_\chi}\begin{pmatrix}y & 0 \\ \frac{z}{T} & 1\end{pmatrix}=e^{i\Phi_\chi^{(z)}(y)}W_0^{(z)}(y)+O\left(\frac{1}{C(\chi)^{1/\deg(F)}}\right).
	\end{equation}
	 For each fixed $|z|\leq b$, the function $W_0^{(z)}(y)$ is a smooth bump function supported on $|y| \sim 1$, determined by the chosen family of analytic newvectors and independent of $\chi$. The famliy $\{W_0^{(z)}\}$ varies smoothly in $z$. The phase function $\Phi_\chi^{(z)}(y)$ satisfies, on the support of $W_0^{(z)}(y)$, the heuristic approximation $e^{i\Phi_\chi^{(z)}(y)}\approx \psi_{F}(-Tz)$. 
	\item The $L^2$-mass distribution: If $b\leq |z|\leq 1$, then 
	\begin{equation}\label{intro. L^2 mass distribution}
		\int_{y\sim Y}\left|W_{f_\chi}\begin{pmatrix}y & 0 \\ \frac{z}{T} & 1\end{pmatrix}\right|^2d_F^\times y\ll \begin{cases}
			|Y|_F & \text{if } Y\leq 2,\\
			|Y|_FC(\chi)^{-N} & \text{if } Y> 2.
		\end{cases}
	\end{equation}	
	In particular, $W_{f_\chi}\begin{pmatrix}y & 0 \\ \frac{z}{T} & 1\end{pmatrix}$ is essentially supported on $0<|y|\leq 2$. 
\end{itemize}
We make a few remarks on these statements.
\begin{enumerate}
	\item	In p-adic case, the Atkin-Lehner operator $w_\pi$ is given by $\tiny\begin{pmatrix}
		0 & 1 \\ -\varpi^{c(\pi)} & 0
	\end{pmatrix}$, and (\ref{intro. symmetry}) holds for any newform of $\pi$. This invariance provides an asymptotic formula for $|z|\geq 1$. In particular, the parameter $T$ plays the role of $\varpi^{c(\chi)}$ in the archimedean setting. 
	\item When $z=0$, (\ref{intro. asymptotic}) gives 
	\begin{equation}
		W_{f_\chi}\begin{pmatrix}
			y & 0 \\ 0 & 1
		\end{pmatrix}=W_0(y)+O\left(\frac{1}{C(\chi)^{1/\deg(F)}}\right),
	\end{equation}
	where $W_0(y)$ is a fixed compactly supported function on $F^\times$, independent of $\chi$. This is comparable to the construction of test vectors in \cite{MV10} or analytic newvector in \cite{JN19}.
	\item When $|z| \lll \frac{1}{|T|}$ (or $x = \frac{z}{T} \lll \frac{1}{|T|^2} \asymp\frac{1}{C(\pi)})$, the phase function satisfies $\psi_F(-Tz) \approx 1$, corresponding to a flat oscillator. This recovers the "invariance property" in Jana–Nelson's construction. For the larger range $\frac{1}{|T|^2} \ll x \leq \frac{b}{|T|}$, our formula shows that the Whittaker function exhibits oscillation of frequency $|T^2 x|_F \approx C(\pi) |x|_F$.
	\item When $b\leq |z|\leq  1$ (or $x=\frac{z}{T}\sim \frac{1}{T})$, the behavior of the translated Whittaker functions becomes more intricate. By analyzing their Mellin components, one can uniformly estimate the $L^2$-distribution as $\chi$ varies. 
	
\end{enumerate}

	For $p$-adic newforms, the asymptotic behavior of translated Whittaker functions and their Mellin components has already been extensively studied (see \cite{Hu17, Hu18, Ass19, Hu20, HMN22}), especially those given in \cite[\S 4.2]{HMN22}. Our results should be viewed as the archimedean analogue of these works.
	
	In Definition \ref{def. analytic new}, we provide an explicit construction of analytic newvectors, first in the induced model and then in the Whittaker model, and we apply stationary phase analysis to this construction. 	This approach makes it possible to see explicitly how the conductor of the representation enters into the oscillatory behavior of these functions, and its implementation requires certain refinements of existing analytic tools.

\subsection{Stationary phase method for $\mathbb{C}$}

The stationary phase method is a fundamental and important tool for oscillatory integrals in analytic number theory and harmonic analysis. As a protype, let $w(t)$ be a compactly supported fucntion on $\mathbb{R}^n$, $\phi(t)$ be a real-valued smooth function on support of $w(t)$, and $\lambda>0$ be a parameter. The asymptotic behavior for 
$$I(\lambda):=\int_{\mathbb{R}^n}w(t)e^{i\lambda \phi(t)}dt,\quad \lambda\rightarrow \infty$$
is well-studied, with standard treatments found in textbooks such as \cite{Stein1993, Hor2003}. 

For application in analytic number theory, however, one often encounters oscillatory integrals with a more general phase function $\Phi$, not necessarily of the special form $\Phi=\lambda\phi$. The works \cite{BKY13} and \cite{KPY19} provide a general framework for their analysis, and in \cite[\S 2.2]{Mu15} the authors establish a useful stationary phase lemma in dimension two.

In this paper, we refine these results, extending them to general phase functions $\Phi$ on $\mathbb{R}^n$, and obtain asymptotic expansions for
$$I=\int_{\mathbb{R}^n}w(t)e^{i\Phi(t)}dt.$$
Identifying $\mathbb{C}\simeq\mathbb{R}^2$, we also derive stationary phase lemmas adapted to $\mathbb{C}$. These refinements allow us to carry out the stationary phase analysis needed for $\mathrm{PGL}_2(\mathbb{C})$. 

Our main interest lies in the phases of the form $e^{i\Phi(t)}=\chi(p(t))\psi_F(q(t))$, where $\chi$ is a unitary character of $F^\times$, $\psi_F$ is an additive character, $p(t)$ and $q(t)$ are rational functions. As an example, in Lemma \ref{lem. Gauss sum. } we consider the generalized Gauss sums in form of
\begin{equation}\label{def. generalized gauss sum}
	G(\chi,t):=\int_{t\in F^\times}V\left(\frac{y}{t}\right)\chi(y)\psi_F(-y)d^\times_F y,
\end{equation}
where $V$ is a fixed compactly supported function on $F^\times$. We extract the associated phase, and analyze the behavior of $G(\chi,t)$ and its derivatives as $C(\chi)\to\infty$, extending the results of \cite[Lemma 3.1.14]{MV10} and \cite[Lemma 4.1]{Wu14}.

\subsection{Acknowledgments}

I would like to thank my PhD supervisor Yueke Hu for his patient guidance and valuable advice. I also thank Zhongkai Tao and Paul D. Nelson for their helpful discussion and feedback.

\section{Preliminaries}

\subsection{Notations and measure normalization}\label{sec. notation}

	Let $\mathbb{F}$ be a number field, $\mathbb{A}$ be its adele ring. Let $\mathbb{F}_v$ be the completion of $\mathbb{F}$ at a place $v$. If $v$ is non-archimedean, then we let $\varpi_v$, $\mathcal{O}_v$, $q_v$ be a uniformizer, the ring of integers, the cardinality of the residue field of $\mathbb{F}_v$, respectively. If $v$ is archimedean, we denote by $\mathrm{deg}(\mathbb{F}_v)$ its degree over $\mathbb{R}$. Let $|\cdot|_v$ be the absolute value on $\mathbb{F}_v$ given by
	\begin{equation}
		|x|=|x|_\mathbb{R}=(x\bar{x})^{\frac{1}{2}},\quad |x|_\mathbb{C}=x\bar{x},\quad |\varpi_v|_{\mathbb{F}_v}=q_v^{-1}.
	\end{equation} 
	Define the local zeta functions by
	\begin{equation}
		\zeta_\mathbb{C}(s)=\Gamma_\mathbb{C}(s)=2(2\pi)^{-s}\Gamma(s),\quad \zeta_\mathbb{R}(s)=\Gamma_\mathbb{R}(s)=\pi^{-s/2}\Gamma(s/2),\quad \zeta_{v}(s)=(1-q_v^{-s})^{-1}.
	\end{equation}
	The (complete) Dedekind zeta function is the given by 
	\begin{equation}
		\zeta_\mathbb{F}(s)=\prod_{v<\infty}\zeta_v(s),\quad \Lambda_\mathbb{F}(s):=\prod_{v}\zeta_v(s)
	\end{equation}
	
	We denote by $\psi=\otimes \psi_v$ the additive character of $\mathbb{A}_\mathbb{F}$ given by $\psi=\psi_{\mathbb{Q}}\circ \mathrm{Tr}_{\mathbb{F}/\mathbb{Q}}$, where $\psi_\mathbb{Q}$ is the additive character of $\mathbb{Q}\backslash \mathbb{A}_\mathbb{Q}$ taking $x\mapsto e^{2\pi i x}$ on $\mathbb{R}$. For $v<\infty$, we let $d_{\psi_v}$ be the conductor of $\psi_v$. The discriminant of $\mathbb{F}$ is  $\Delta_\mathbb{F}=\prod_{v<\infty}q_v^{d_{\psi_v}}$. When $v$ is an archimedean place, we set $d_{\psi_v}=0$. 
	\newline 
	
	In this paper, we follow the measure normalization as in \cite[\S 2.2]{BJN24}. Let $F=\mathbb{F}_v$ be a local field. The additive Haar measure on $F$ is normalized to be self-dual with respect to $\psi_F$. That is, for $F=\mathbb{R}$, it gives the Lebesgue measure $d_\mathbb{R}$; for $F=\mathbb{C}$, it gives twice of the usual Lebesgue measure $d_\mathbb{C}=idz\wedge d\bar{z}=2dx\wedge dy$; for $v<\infty$, it gives the measure $q_v^{-d_\psi/2}$ to $\mathcal{O}_v$. The additive measure on $\mathbb{A}$ is given by $dx=\prod_v d_v x_v$. 
	
	We define $d_F^\times x=\frac{d_Fx}{|x|_F}$ to be the Haar measure on the multiplicative group $F^\times$, so that $\mathcal{O}_F^\times$ has total mass $\zeta_F(1)^{-1}$ when $F$ is non-archimedean and $\psi_F$ is unramified, and $d^\times x=\prod_v d_v^\times x_v$ as the Haar measure on the idele group $\mathbb{A}^\times$. 

	We denote by $G=\mathrm{PGL}_2$ and $\mathrm{B}$, $\mathrm{N}$, $\mathrm{N}'$, $\mathrm{A}$ be the usual Borel, upper triangular unipotent, lower triangular unipotent, diagonal subgroups with lower diagonal entry equal to $1$. For $t,x,y$ in any ring $R$, we set
	\begin{equation}
		n(t)=\begin{pmatrix} 1 & t \\ 0 & 1\end{pmatrix},\quad n'(x)=\begin{pmatrix} 1 & 0 \\ x & 1\end{pmatrix},\quad a(y)=\begin{pmatrix} y & 0 \\ 0 & 1\end{pmatrix},\quad w=\begin{pmatrix} 0 & -1 \\ 1 & 0\end{pmatrix}.
	\end{equation}
	The measure on $N,N',A$ are taken as $d_Ft,d_Fx,d_F^\times y$ respectively. The right Haar measure on $B$ is given by $d_Rb=d_Ftd_F^\times y$, while the left Haar measure is given by  $d_Lb=\frac{1}{|y|_F}d_Rb=\frac{1}{|y|_F}d_Ftd_F^\times y$. The standard maximal compact subgroup $K$ of $G$ is given by
	$$K=\mathrm{SU}(2)/\{\pm I\},\quad K=\mathrm{SO}(2)/\{\pm I\},\quad K=\mathrm{PGL}_2(\mathcal{O}_v),$$
	and the Haar measure on $K$ is normalized so that its total mass is $\zeta_F(2)^{-1}$. The Haar measure on $G$ is normalized by $dg=d_Lbdk$. 
	
	With this normalization, for $f$ a smooth compact support function on $N\backslash G$, we have 
	\begin{equation}
		\int_{N\backslash G}f(g)dg=\int_{k\in K}\int_{y\in F^\times}f(a(y)k)\frac{d_F^\times y}{|y|_F}dk=\int_{x\in F}\int_{y\in F^\times}f(a(y)n'(x))\frac{d_F^\times y}{|y|_F}d_Fx.
	\end{equation}

	The measure on the adelic points of each subgroup is given by the product of the corresponding local measures. We denote $dg$ the quotient measure on 
	$$\mathbf{X}_{\mathrm{PGL}_2}:=\mathrm{PGL}_2(\mathbb{F})\backslash \mathrm{PGL}_2(\mathbb{A}),$$
	whose total volume is $V_\mathbb{F}:=\mathrm{vol}(\mathbf{X}_{\mathrm{PGL}_2})=2\Delta_\mathbb{F}^{1/2}<\infty$. 

	By $X\ll Y$ or $X=O(Y)$, we mean that there exists a constant $C>0$ such that $|X|\leq C|Y|$.
	By $X\asymp Y$, we mean that both $X\ll Y$ and $Y\ll X$ hold.
	We write $X\ll_{\alpha,\beta,\ldots} Y$ or $X=O_{\alpha,\beta,\ldots}(Y)$ if the implicit constant $C$ is allowed to depend on the parameters $\alpha,\beta,\ldots$.	We say $y\sim Y$ if $y$ lies in some dyadic intervals of $Y$, e.g. $Y/2 \leq y \leq 2Y$.
	
	We write $X\lll Y$ to mean $X\ll Y$  with an implicit constant that is understood to be sufficiently small. We say $X\approx Y$ to indicate that $X$ behaves like $Y$ in the relevant range. These notations will appear only in informal remarks.

\subsection{Representations for $\mathrm{PGL}_2(F)$}

	In this section, we review some fundamental results about generic irreducible representations of $\mathrm{PGL}_2(F)$, where $F=\mathbb{R}$ or $\mathbb{C}$ are archimedean local fields. In general, we  refer to \cite{JL70, Go70, Bump97, GTM300}. 

\subsubsection{Character for $F^\times$} 

	Let $\chi$ be a character of $F^\times$, then $\chi$ can be parametrized by 
	\begin{equation}
		\chi(z)=\chi_m(z)|z|_F^s:=\left(\frac{z}{|z|}\right)^{m}|z|_F^{s},\quad z\in F^\times,
	\end{equation}
	where $m\in \{0,1\}$ if $F=\mathbb{R}$, $m\in\mathbb{Z}$ if $F=\mathbb{C}$, and $s\in \mathbb{C}$. The character $\chi$ is unitary if and only if $s=i\lambda$ for some $\lambda\in \mathbb{R}$. For such $\chi$, the associated $L$-factors and conductors are given by
	\begin{equation}\label{def. L-factor and conductor of GL1}
		L(\chi,s')=\begin{cases}
			\Gamma_\mathbb{R}(s'+s+m), & F=\mathbb{R},\\ \Gamma_\mathbb{C}(s'+s+\frac{m}{2}), & F=\mathbb{C};
		\end{cases}\quad C(\chi)=\begin{cases}
			1+\frac{|s+m|}{2\pi}, & F=\mathbb{R},\\
			\left(1+\frac{|s+\frac{m}{2}|}{2\pi}\right)^{2}, & F=\mathbb{C}. 
		\end{cases}
	\end{equation}
	If $\chi=\chi_m|\cdot|_F^{i\lambda}$ is a unitary character, we define its spectral parameter by
	\begin{equation}\label{def. spectral parameters}
		T(\chi):=\left\{\begin{aligned}
			\ &\frac{\lambda}{2\pi}, &\quad F=\mathbb{R};\\
			&\frac{\lambda+\frac{m}{2i}}{2\pi}, &\quad F=\mathbb{C}.
		\end{aligned}\right.
	\end{equation}
	With this notation, one has $C(\chi)\asymp 1+|T(\chi)|_F$. If $\mu$ is another unitary character, then we have $T(\chi\mu)=T(\chi)+T(\mu)$, $T(\chi\mu^{-1})=T(\chi)-T(\mu)$.

\subsubsection{Smooth representations for $\mathrm{PGL}_2(F)$}

	We refer to \cite[Chapter 2]{Bump97},\cite[Chapter 4]{GTM300} for the notion of irreducible admissible representations, unitary representations, and smooth representations for Lie groups. 

	We say $(\pi,V)$ is an admissible representation of $G$ if $\pi$ is an $(\mathfrak{g},K)$-module. Say $(\pi,V)$ is a unitary representation if $V$ is a Hilbert space and the action of $\pi$ is unitary. Let $(\pi,V)$ be a unitary representation and $v\in V$. If $K.v$ has finite dimension, then we say $v$ is $K$-finite. If $K.v$ spans an irreducible representation of $K$, then we say $v$ is $K$-isotypic. The $K$-finite vectors in an irreducible unitary representation consist of an irreducible $(\mathfrak{g},K)$-module of $G$. 

	Let $(\pi,V)$ be a unitary representation of $G$, the subspace of smooth vectors $V^\infty$ is naturally a Fr\'{e}chet space, w.r.t. the semi-norms defined by $\|X.v\|$, $X\in \mathfrak{U}(\mathfrak{g})$, $v\in V^\infty$. In general, assume $G$ acts on $M$ a homogeneous action, and $\pi\subset L^2(M)$ is a representation. Then the smooth vectors in $\pi$ are exactly those in $C^\infty(M)$. The subspace of $K$-finite vectors are dense in $V$ (also in $V^\infty$).
	
\subsubsection{Sobolev norms on smooth representations}

	We recall the notion of Sobolev norms in \cite[\S 2.3]{MV10}, \cite[\S6]{Nel19}. Assume that $(\pi,V)$ is a unitary representation, $V^\infty$ is the subspace of smooth vectors. The generalized Laplace operator  is given by
	$$\Delta_G=1-\sum_{X_i\in \mathcal{B}(\mathfrak{g})} X_i^2=1-\mathcal{C}_G+2\mathcal{C}_K$$
	where $\mathcal{C}_G$, $\mathcal{C}_K$ denote the Casimir element of $\mathfrak{g}$ and $\mathfrak{k}$. If $v\in \pi$ is a $K$-isotypic vector in $\pi$, with $K$-type $\nu$, then we have 
	$$\Delta_G(v)=(1-\lambda_\pi+2\lambda_\nu)v$$
	where $\lambda_\pi$, $\lambda_\nu$ are the eigenvalue of $\mathcal{C}_G$ and $\mathcal{C}_K$ action on $\pi$ and $\nu$ respectively. Let $N_{\pi\nu}:=(1-\lambda_\pi+2\lambda_\nu)^{1/2}$. For any $d\in \mathbb{R}$, the Sobolev norm of $v\in \pi$ is defined by 
	$$\mathcal{S}_d^\pi(v)^2:=\sum_{\nu}N_{\pi\nu}^{d}\|v_\nu\|^2,$$
	where $v=\sum_{\nu} v_\nu$ is the decomposition of the (smooth) vector $v\in \pi$ into $K$-isotypic components. The sum converges in the Fr\'{e}chet topology. For any $v\in V^\infty$ and any $d\in \mathbb{R}$, $\mathcal{S}_d^\pi(v)$ is finite. 
	
	\begin{lemma}[Uniform trace property]\label{lem. Sobolev norm. trace property}
		
		There exists a fixed $d_0\geq 0$ so that for all irreducible representation $\pi$ of $G$, $\mathcal{B}(\pi)$ be an orthogonal basis of $\pi$ and $d\geq d_0$, 
		$$\sum_{\psi\in \mathcal{B}(\pi)}\mathcal{S}^\pi_{-d}(\psi)\ll 1$$
		
	\end{lemma}
	
	\begin{proof}
		See \cite[\S 2.6.3]{MV10} or \cite[Lemma 6.5]{Nel19}. 
	\end{proof}

	As in \cite[\S 2.6.5]{MV10}, \cite[\S 6.7]{Nel19a}, we recall the notion of Sobolev conductor for an irreducible representation $\pi$ of $G$:
	$$C_{\mathrm{Sob}}(\pi):=\min\{N_{\pi\nu}:\nu\in K^{\wedge} \text{ with }\pi^\nu\neq \{0\}\}.$$
	When $G=\mathrm{GL}_n(F)$, we have $C_{\mathrm{Sob}}(\pi)\asymp C(\pi)$. 
		
	\begin{lemma}[Integration by parts]\label{lem. Sobolev norm. Integration by parts}
		
		Let $\mathbf{H}$ be an algebraic subgroup of $\mathbf{G}$. Let $G=\mathbf{G}(F)$, $H=\mathbf{H}(F)$, and $\pi$, $\sigma$ are unitary representation of $G$, $H$ respectively. Suppose $\ell
		:\pi \otimes\sigma\rightarrow \mathbb{C}$ is a linear functional satisfy that it is invariant under $H$-action and 
		$$\ell(u\otimes v)\ll \mathcal{S}_{d_0}^{\pi}(u)\mathcal{S}_{d_0}^\sigma(v),\quad u\in \pi,\ v\in \sigma. $$
		for some fixed $d_0$. Then for fixed $d$, $d'$ with $d'$ large in terms of $d$,
		$$\ell(u\otimes v)\ll \mathcal{S}_{d'}^{\pi}(u)\mathcal{S}_{-d}^{\sigma}(v)C_{\mathrm{Sob}}(\sigma)^{-d}.$$
		
	\end{lemma}
	
	\begin{proof}
		See \cite[Lemma 6.9]{Nel19}.
	\end{proof}
	
	\begin{lemma}[Weyl law]\label{lem. Sobolev norm. Weyl law}
		
		Let $G=\mathbf{G}(\mathbb{A})$ and $\pi$ be generic irreducible automorphic representations of $G$. Then there is a $N_0\geq 0$ such that for any $N\geq N_0$, 
		$$\int_{\pi\textnormal{ gen}}C_{\mathrm{Sob}}(\pi)^{-N}d\mu_{\mathrm{pl}}\ll 1.$$
	\end{lemma}
	
	\begin{proof}
		See \cite[\S 2.6.5]{MV10}. 
	\end{proof}

\subsubsection{Induced representations and classification of unitary representations for $\mathrm{PGL}_2(F)$.}

	Let $\chi$ be a character of $F^\times$. We denote by $\mathcal{I}(\chi)$ to be the induced representation of $G$, consisting of smooth functions on $G$ satisfying that 
	\begin{equation}
		f(a(y)n(x)g)=\chi(y)|y|^{\frac{1}{2}}_Ff(g).
	\end{equation}
	The $K$-isotypic vectors spans a subspace of $K$-finite vectors in $\mathcal{I}(\chi)$.

	There is a classification theorem for irreducible unitary generic representations for $G$ (see \cite{JL70} for example). When $F=\mathbb{R}$, $\pi$ is unitary if it is
	\begin{itemize}
		\item tempered principal series: $\pi=\mathcal{I}(\chi)$ where $\chi$ is a unitary character.
		\item discrete series: $\pi\subseteq \mathcal{I}(\chi)$ as a subrepresentation, where $\chi=|\cdot|_\mathbb{R}^{\frac{k-1}{2}}$ for positive even integer $k$. It has lowest weight $k$. 
		\item complementary seires: $\pi=\mathcal{I}(\chi)$ where $\chi=|\cdot|_\mathbb{R}^{\theta}$ for some real number $0<|\theta|<\frac{1}{2}$. 
	\end{itemize}
	When $F=\mathbb{C}$, $\pi$ is unitary if it is
	\begin{itemize}
		\item tempered principal series: $\pi=\mathcal{I}(\chi)$ where $\chi$ is a unitary character.
		\item complementary seires: $\pi=\mathcal{I}(\chi)$ where $\chi=|\cdot|_\mathbb{C}^{\theta}$ for some real number $0<|\theta|<\frac{1}{2}$.
	\end{itemize}
	As remark, there is no discrete seires representations for $F=\mathbb{C}$. We can always embed $\pi$ into an induced model $\mathcal{I}(\chi)$, as in \cite[\S 3.1]{BJN23}. If $\pi$ is a unitary representation and realized in an induced model, then the smooth vectors are those smooth functions in $\mathcal{I}(\chi)$. 
	\newline

	We define (as in \cite[\S 2]{Ich08})
	\begin{equation}
		\vartheta(\pi)=\begin{cases}
			0 & \textnormal{if }\pi \textnormal{ is tempered principal series or discrete series,}\\
			|\theta| & \textnormal{if }\pi \textnormal{ is complementary series with }\chi=|\cdot|_F^\theta
		\end{cases}
	\end{equation}
	We say $\pi$ is $\vartheta$-tempered if $\vartheta(\pi)\leq \vartheta$. The best bounds towards Selberg eigenvalue conjecture or Ramanujan conjecture are due to \cite{Kim03}, \cite{BB11}, which says the local constitutes occur in an cuspidal representation satisfy $\vartheta(\pi)\leq\frac{7}{64}$.

\subsubsection{Intertwining operators and inner products on $\mathcal{I}(\chi)$}

	Let $\chi$ be a character of $F^\times$, and $\mathcal{I}(\chi)$ be the principal series representation of $\mathrm{PGL}_2$ unitarily induced from $\chi$. We introduce the standard intertwining map and normalized intertwining map from $\mathcal{I}(\chi)$ to $\mathcal{I}(\chi^{-1})$ by
	\begin{equation}\label{def. Intertwining map}
		M(\chi)f:=\int_{t\in F}f(wn(t))dt,\quad 	M^*(\chi)=\gamma(\chi^2,\psi_F,0)M(\chi).
	\end{equation}
	The intertwining integrals are well-defined for $\Re(\chi)>0$ and have meromorphic continuation to all characters $\chi\in \mathfrak{X}(F)$. When $\mathcal{I}(\chi)$ is irreducible, these operators $M(\chi)$, $M^*(\chi)$ have no pole and are bijecton between $\mathcal{I}(\chi)$ and $\mathcal{I}(\chi^{-1})$, where the inverse is given by
	$$M^*(\chi^{-1})\circ M^*(\chi)=1.$$
	When $F=\mathbb{R}$ and $\pi\subset \mathcal{I}(\chi)$ is a discrete series representation,  $M(\chi)v=0$ for every $v\in \pi$, while $0\neq M^*(\chi)v\in \mathcal{I}(\chi^{-1})$ by cancelling the zeros by the pole of $\gamma$-factor. 
\newline

	Now we assume that $\pi\subseteq \mathcal{I}(\chi)$ is an irreducible unitary representation for $G$. For $f\in \mathcal{I}(\chi),\tilde{f}\in \mathcal{I}(\chi^{-1})$, the natural pairing $(f,\tilde{f}):=\int_{K}f(k)\tilde{f}(k)dk$ is a $G$-invariant pairing. For $f\in \mathcal{I}(\chi)$, we introduce
	\begin{equation}\label{def. nature dual in induced model}
		\tilde{f}=\begin{cases}
			\bar{f}, & \pi \textnormal{ is tempered principal series}, \\
			\overline{M^*(\chi)f}, & \textnormal{otherwise},
		\end{cases}
	\end{equation}
	belongs to $\mathcal{I}(\chi^{-1})$. Then we can introduce the inner product on $\pi\subseteq \mathcal{I}(\chi)$ by
	\begin{equation}\label{def. inner product on induced model}
		\langle f_1,f_2\rangle=\int_{t\in F}f_1(wn(t))\tilde{f}_2(wn(t))d_Ft= (f_1,\tilde{f}_2).
	\end{equation}
	We refer to \cite[\S 3.1]{BJN23}, \cite[\S 3.4]{BJN24}.

\subsubsection{Whittaker models}

	Let $\mathcal{W}(\psi)$ be the space of smooth functions on $G$ satisfy that
	$$W(n(x)g)=\psi_F(x)W(g),\quad x\in F,g\in G$$
	We say $\pi$ is a generic represenation if there exists a non-zero intertwining map $\pi\rightarrow \mathcal{W}(\psi)$. It is well-known that an irreducible representation $\pi$ for $\mathrm{PGL}_2$ is generic if and only if $\pi$ is infinite dimensional, and has multiplcity one in $\mathcal{W}(\psi)$. The image of $\pi$ in $\mathcal{W}(\psi)$ is called the Whittaker model of $\pi$, denoted by $\mathcal{W}(\pi,\psi)$. 

	For $\pi\subseteq \mathcal{I}(\chi)$, the Jacquet integral gives an intertwining map from $\pi$ to $\mathcal{W}(\pi,\psi)$, by
	\begin{equation}\label{def. intertwining map to Whittaker model, 1}
		f\mapsto W_f(g)= \int_{t\in F}^{\mathrm{reg}}f(wn(t)g)\psi_F(-t)d_Ft.
	\end{equation}
	In particular, for $g=a(y)$, we have 
	\begin{equation}\label{def. intertwining map to Whittaker model, 2}
		f\mapsto W_f(a(y))=\chi^{-1}(y)|y|^{1/2}_F\int_{t\in F}^{\mathrm{reg}}f(wn(t))\psi_F(-ty)d_Ft. 
	\end{equation}
	These integrals are absolutely convergent for $\Re(\chi)>0$ and have analytic continuation to all $\chi$. The inversion formula holds for $\Re(\chi)>-1$, see \cite[\S 1.6, \S 2.5]{Go70}
	\begin{equation}\label{def. intertwining map to Whittaker model, inverse}
		f(wn(t))=\int_{y\in F^\times}^{\mathrm{reg}}\chi(y)|y|_F^{-\frac{1}{2}}W_f(a(y))\psi_F(ty)d_Fy
	\end{equation}
	In general, we can explain them as regularized integrals, as in \cite[\S 7.2]{Nel19}. 

	When $\pi$ is unitary, we define an invariant inner product on $\pi$ by 
	\begin{equation}\label{def. Whittaker inner product}
		\langle W_1,W_2\rangle=\int_{y\in F^\times} W_1(a(y))\overline{W_2(a(y))}d_F^\times y 
	\end{equation}
	This inner product is refered to the standard inner product on $\pi$. 
	
	\begin{lemma}
		
		Let $\chi$ be a character of $F^\times$, and $\pi= \mathcal{I}(\chi)$. For $f\in \mathcal{I}(\chi)$, let $f\mapsto W_f$ be the Jacquet integral in (\ref{def. intertwining map to Whittaker model, 1}). Then we have
		\begin{equation}
			W_{M*(\chi)f}=W_f,
		\end{equation}
		where $M^*(\chi)$ is the normalized intertwining operator in (\ref{def. Intertwining map}). Furthermore, the intertwiner $f\mapsto W_f$ is an isometry under our normalization, i.e.
		$$\langle f_1,f_2\rangle =\langle W_{f_1},W_{f_2}\rangle$$
		for any $f_1,f_2\in \pi\subseteq\mathcal{I}(\chi)$. 
	\end{lemma}
	
	\begin{proof}
		See \cite[(3.11)]{BJN24}. 
	\end{proof}

\subsubsection{The local trilinear forms}

	Let $G=\mathrm{PGL}_2(F)$, $\pi_1,\pi_2,\pi_3$ be three unitary generic irreducible representations of $G$. It is well know that (see \cite{Pra90} for example) the dimension of space of trilinear invariant functionals
	$$\ell: \pi_1\otimes\pi_2\otimes\pi_3\rightarrow \mathbb{C}$$
	is at most one. It is non-zero if and only if the local $\varepsilon$-factor is $1$
	$$\varepsilon(\pi_1\otimes\pi_2\otimes\pi_3,1/2)=1. $$
	It is of interest to give such a space a basis. For $\varphi_i\in \pi_i$, we define 
	\begin{equation}\label{def. local triple product integral}
		I^T(\varphi_1,\varphi_2,\varphi_3):=\int_{g\in\mathrm{PGL}_2(F)}\prod_{i=1}^3\langle \pi_i(g)\varphi_i,\varphi_i\rangle dg
	\end{equation}
	to be the integral of product of matrix coefficients. When $\pi_1,\pi_2,\pi_3$ are $\vartheta_i$-tempered with $\vartheta_1+\vartheta_2+\vartheta_3<\frac{1}{2}$, the integral is absolutely convergent, as in \cite{Ich08}. When $\pi_1,\pi_2,\pi_3$ are unramified and  $\varphi_1,\varphi_2,\varphi_3$ are the spherical elements, it is well-known that 
	\begin{equation}\label{prop. special value for local triple product integral}
		\frac{I^T(\varphi_1,\varphi_2,\varphi_3)}{\|\varphi_1\|^2\|\varphi_2\|^2\|\varphi_3\|^2}=\zeta_F(2)\frac{L(\pi_1\otimes\pi_2\otimes\pi_3,\frac{1}{2})}{\prod_{i=1}^3L(\pi_i,\mathrm{Ad},1)}.
	\end{equation}
	(we note that our measure noramlization is slightly different from \cite{Ich08}). We could normalize the integral by
	\begin{equation}\label{def. normalized local triple product integral}
		I^0(\varphi_1,\varphi_2,\varphi_3):=\zeta_F^{-1}(2)\frac{\prod_{i=1}^3L(\pi_i,\mathrm{Ad},1)}{L(\pi_1\otimes\pi_2\otimes\pi_3,\frac{1}{2})}I^T(\varphi_1,\varphi_2,\varphi_3)
	\end{equation}
	then $I^0(\varphi_1,\varphi_2,\varphi_3)=1$ for normalized spherical vectors. (Recall the differnce in our measure normalization as in \cite{BJN24} with those in \cite{Ich08}.) 
\newline

	From now on, let  $\pi_3=\mathcal{I}(\chi)$ be an irreducible $\vartheta_3$-tempered unitary principal sereis representation, $\pi_1,\pi_2$ be irreducible $\vartheta_1$, $\vartheta_2$-tempered representations. For $W_1\in \mathcal{W}(\pi_1,\psi)$, $\bar{W}_2\in \mathcal{W}(\pi_2,\bar{\psi})$, $f_3\in \mathcal{I}(\chi)=\pi_3$, we define the local Rankin-Selberg integrals are $\pi_1\otimes\pi_2\otimes\pi_3\rightarrow \mathbb{C}$: 
	\begin{equation}
		\begin{aligned}
			\Psi(W_1,W_2,f_3)&:=\int_{N\backslash \mathrm{PGL}_2}W_1(g)W_2(g)f_3(g)dg
		\end{aligned}
	\end{equation}
They are absolutely convergent whenever $\vartheta_1+\vartheta_2+\vartheta_3<\frac{1}{2}$, and have meromorphic continuation to the whole plane. 

\begin{lemma}
	With the notation and assumptions as above, we have
	\begin{equation}
		\Psi(W_1,\bar{W}_2,f_3)\gamma(\frac{1}{2},\pi_1\otimes\pi_2\otimes\chi)=\Psi(W_1,\bar{W}_2,M^*(\chi)f_3)
	\end{equation} 
\end{lemma}

\begin{proof}
	See \cite[Lemma 5]{BJN23} or \cite[Lemma 3.2]{BJN24}. 
\end{proof}

\begin{lemma}\label{trilinear form}
	Let $\pi_1,\pi_2,\pi_3$ be unitary irreducible representations for $\mathrm{PGL}_2(F)$, and $\pi_3=\mathcal{I}(\chi)$ be a principal series representation.  Assume that $\vartheta_1+\vartheta_2+\vartheta_3<\frac{1}{2}$. 
	
	Let $W_1\in \mathcal{W}(\pi_1,\psi)$, $\bar{W}_2\in \mathcal{W}(\pi_2,\bar{\psi})$, $f_3\in\mathcal{I}(\chi)$, we define
	$$f_3^{\natural}=\begin{cases} f_3 & \textnormal{ if }\pi_3 \textnormal{ is tempered princial series}\\
		M^*(\chi)f_3 &\textnormal{ otherwise}\end{cases}$$ 
	Then we have 
	\begin{equation}
		I^T(W_1,\bar{W}_2,f_3)=\Psi(W_1,\bar{W}_2,f_3)\overline{\Psi(W_1,\bar{W}_2,f_3^\natural)}.
	\end{equation}
	Furthermore, combine with the above Lemma, we get 
	\begin{equation}
		I^T(W_1,\bar{W}_2,f_3)=\begin{cases}
			|\Psi(W_1,\bar{W}_2,f_3)|^2,  & \pi_3 \textnormal{ tempered principal seires}\\
			\overline{\gamma(\tfrac{1}{2},\pi_1\otimes\pi_2\otimes\chi)}|\Psi(W_1,\bar{W}_2,f_3)|^2, &\pi_3 \textnormal{ otherwise}.
		\end{cases}
	\end{equation}
\end{lemma}

\begin{proof}
	This is essentially in \cite[Lemma 6]{BJN23} and \cite[Lemma 3.3]{BJN24}.
\end{proof}

\begin{remark}
	The local trilinear form in this context has been extensively studied, as documented in the literature (e.g., \cite{MV10, Nel19a, Hsi21, Chen21, BJN23, BJN24}).
\end{remark}

\subsection{Stationary phase analysis for oscillatory integrals}

\subsubsection{Stationary phase lemmas for $\mathbb{R}^n$}

Let $F=\mathbb{R}^n$. For $t=(t_1,...,t_n)\in\mathbb{R}^n$, we denote the coordinate differential operators by $\partial_{1}=\frac{\partial}{\partial t_1}$, ..., $\partial_n=\frac{\partial}{\partial t_n}$.  Let $\Phi$ be a smooth function on $\mathbb{R}^n$. The gradient $\nabla\Phi$ and the Hessian matrix $H_\Phi$ of $\Phi$ are defined by 
$$\nabla\Phi(t)=\left(\partial_{1}\Phi(t),...,\partial_{n}\Phi(t)\right),\quad H_\Phi(t)=\left(\partial_{j}\partial_{k}\Phi(t)\right)_{1\leq j,k\leq n}.$$ 
The Euclidean norm of $\nabla\Phi$ and Laplacian operators are given by
$$\|\nabla\Phi(t)\|^2:=\sum_{j=1}^n\left|\partial_{j}\Phi(t)\right|^2,\quad \Delta \Phi(t)=\partial_1^2\Phi(t)+...+\partial_n^2\Phi(t).$$ 
For a multi-index $\alpha=(\alpha_1,...,\alpha_n)\in \mathbb{Z}^n_{\geq 0}$,  we set $\mathcal{D}^{\alpha}=\mathcal{D}_{1}^{\alpha_1}...\mathcal{D}_n^{\alpha_n}$, and $|\alpha|:=\alpha_1+...+\alpha_n$. 
\vspace{11pt}

Let $w(t)$ be compactly supported smooth functions on $\mathbb{R}^n$, with $\mathrm{Supp}(w)\subset S$. Let $\Phi(t)$ be real-valued smooth functions on $S$. In this section, by providing a uniform control on derivatives, we aim to provide uniform estimates for a family of the oscillatory integrals in form of 
\begin{equation}\label{def. the oscillatory integral}
	I=\int_{\mathbb{R}^n}  w(t)e^{i\Phi(t)}dt.
\end{equation}
The following theorem is refered to the case where no stationary point occurs.

\begin{theorem}\label{lem. stationary phase lemma, no stationary point}
	 
	 In the integrals (\ref{def. the oscillatory integral}), we assume that there exist parameters $X, Y, U, Q>0$ such that the derivatives of $w$ and $\Phi$ satisfy
	\begin{equation}\label{step. derivative bounds}
		\sup_{t\in S}|\mathcal{D}^\alpha w| \ll_\alpha X U^{-|\alpha|},\quad \sup_{t\in S}|\mathcal{D}^\beta \Phi|\ll_\beta YQ^{-|\beta|},
	\end{equation}
	for all multi-indices $\alpha,\beta\in \mathbb{Z}_{\geq 0}^n$ with $|\alpha|\geq 0$ and $|\beta|\geq 2$. Furthermore, we suppose that there exists $R>0$ such that for any $t\in S$, we have 
	\begin{equation}\label{step. lower bound for derivative}
		\|\nabla \Phi(t)\|\gg YR.
	\end{equation}
	Then, for any integer $m\geq 0$, the integral $I$ satisfies the uniform bound 
	\begin{equation}\label{step. stationary phase,  non-stationary point}
		I\ll_m \frac{\mathrm{Area}(S)\cdot X}{Y^m}\left(\frac{1}{U^mR^m}+\frac{1}{Q^mR^m}+\frac{1}{Q^{2m}R^{2m}}\right).
	\end{equation}
	
\end{theorem}

\begin{proof}
	
	Similar to \cite[Lemma 8.1]{BKY13}, we define the differential operator
	$$\mathcal{L}(w):=-\nabla\cdot \left(\frac{w\nabla \Phi}{i\|\nabla \Phi\|^2}\right)=-\frac{1}{i}\left(\frac{\nabla w\cdot \nabla \Phi}{\|\nabla\Phi\|^2}+w\frac{\Delta\Phi}{\|\nabla\Phi\|^2}-2w\frac{\langle H_\Phi\nabla\Phi, \nabla \Phi\rangle}{\|\nabla\Phi\|^4}\right),$$
	for smooth functions $w:\mathbb{R}^n\rightarrow \mathbb{C}$ with compact support. By integration by parts, we have 
	$$\int_{\mathbb{R}^n}w(t)e^{i\Phi(t)}dt=\int_{\mathbb{R}^n}\mathcal{L}^m(w)e^{i\Phi(t)}dt$$
	for any $m \in \mathbb{Z}_{\geq 0}$.
	We now expand $\mathcal{L}^m(w)$ inductively. We observe that with each application of the differential operator, the degree of $\|\nabla\Phi\|$ in the denominator increases by $2$ or $4$, while the numerator expands into a sum of terms involving partial derivatives acting on $w$ or $\Phi$. Hence, the expansion can be organized into the following form:
	$$\mathcal{L}^m(w)=\sum_{|\alpha|=0}^{m} \sum_{\nu=0}^{ m-|\alpha|}\frac{\mathcal{D}^\alpha w}{\|\nabla \Phi\|^{2m+2\nu}} \cdot\sum_{\substack{|\beta_1|+...+|\beta_k|+|\alpha|=2m+2\nu\\ k=m+2\nu}} c_{ \alpha, \beta} \prod_{i=1}^k\mathcal{D}^{\beta_i} \Phi,$$
	where $c_{\alpha,\beta}$ are absolute coefficients. Let $k_0$ denote the number of $|\beta_i|\geq 2$, then $k_0\leq m-|\alpha|$. After reindexing so that these correspond to the first $k_0$ term, and the expansion can be rewritten as
	\begin{equation}\label{step. expansion of L operator}
		\mathcal{L}^m(w)= \sum_{|\alpha| =0} ^m\sum_{k_0=0}^{m-|\alpha|} \frac{\mathcal{D}^\alpha w}{\|\nabla \Phi\|^{m+k_0}} \cdot\sum_{\substack{|\beta_1|+...+|\beta_{k_0}|=m-|\alpha|+k_0\\|\beta_i|\geq 2,\text{ for }0\leq i\leq k_0\\ |\beta_j|=1,\text{ for }k_0+1\leq j\leq k}} c_{\alpha, \beta} \prod_{i=1}^{k_0}\mathcal{D}^{\beta_i} \Phi\prod_{j=k_0+1}^{k}\frac{\mathcal{D}^{\beta_j}\Phi}{\|\nabla \Phi\|}.
	\end{equation}
	For $|\beta_j|=1$, we have $\frac{|\mathcal{D}^{\beta_j}\Phi|}{\|\nabla\Phi\|}\leq 1$. 
	Combining this with the assumption for $\Phi$ (\ref{step. derivative bounds}), we obtain
	\begin{equation}\label{step. estimate for higher derivatives of Phi}
		\sum_{\substack{|\beta_1|+...+|\beta_{k_0}|=m-|\alpha|+k_0\\|\beta_i|\geq 2,\text{ for }0\leq i\leq k_0\\ |\beta_j|=1,\text{ for }k_0+1\leq j\leq k}} c_{\alpha, \beta} \prod_{i=1}^{k_0}\mathcal{D}^{\beta_i} \Phi\prod_{j=k_0+1}^{k}\frac{\mathcal{D}^{\beta_j}\Phi}{\|\nabla \Phi\|}\ll \frac{Y^{k_0}}{Q^{m-|\alpha|+k_0}}.
	\end{equation}
	Combining (\ref{step. expansion of L operator}) and (\ref{step. estimate for higher derivatives of Phi}), and integrating over $S$,  we obtain
	\begin{equation}\label{step. main estimate for I}
		I\ll  |S|\cdot \sum_{|\alpha| =0} ^{m}\sum_{k_0=0}^{m-|\alpha|} \sup_{t\in S}\frac{|\mathcal{D}^\alpha w|}{\|\nabla \Phi\|^{m+k_0}}\cdot\frac{Y^{k_0}}{Q^{m-|\alpha|+k_0}}.
	\end{equation}
	Now, inserting the assumptions for $w$ in (\ref{step. derivative bounds}) and the lower bound for $\Phi$ in (\ref{step. lower bound for derivative}), we have
	$$\sup_{t\in S}\frac{|\mathcal{D}^\alpha w|}{\|\nabla \Phi\|^{m+k_0}}\ll \frac{\frac{X}{U^{|\alpha|}}}{Y^{m+k_0}R^{m+k_0}}.$$
	The claimed estimate then follows.
	
\end{proof}

\begin{lemma}\label{lem. stationary phase, with higher order zero}
	
	In the oscillatory integral \eqref{def. the oscillatory integral}, assume that the derivative bounds hold
	$$\sup_{t\in S}|\mathcal{D}^\alpha w| \ll_\alpha X U^{-|\alpha|},\quad \sup_{t\in S}|\mathcal{D}^\beta \Phi|\ll_\beta YQ^{-|\beta|}.$$
	Suppose that $\nabla\Phi(t)=0$ has a unique solution $t_0\in S$, and that for any $t\in S$ we have 
	\begin{equation}\label{step. derivative bounds 2}
		\|\nabla\Phi(t)\|\gg \frac{Y}{Q^2}\|t-t_0\|.
	\end{equation} 
	If $w(t)$ varnishes to order $2m$ at $t_0$, then 
	\begin{equation}
		I\ll_{m}  \frac{\mathrm{Area}(S)\cdot X}{Y^m}\left(\frac{Q^m}{U^m}+\frac{Q^{2m}}{U^{2m}}\right).
	\end{equation}
	
\end{lemma}

\begin{proof}
	
	For any $0\leq m'\leq 2m$, the Taylor expansion of order $m'$ around $t_0$ implies that for all multi-indices $0\leq |\alpha|\leq m'-1$, 
	$$\sup_{t\in S} |\mathcal{D}^\alpha w|\ll \|t-t_0\|^{m'-|\alpha|}\sum_{|\beta|=m'}\sup_{t\in S}|\mathcal{D}^\beta w|\ll \|t-t_0\|^{m'-|\alpha|}\cdot \frac{X}{U^{m'}}.$$
	Combining this with \ref{step. derivative bounds 2}, we obtain that for every index combination satisfying $m+|\alpha|+k_0\leq 2m$,
	$$\sup_{t\in S}\frac{|\mathcal{D}^{\alpha}w|}{\|\nabla\Phi\|^{m+k_0}}\ll \frac{\frac{X}{U^{m+|\alpha|+k_0}}}{\frac{Y^{m+k_0}}{Q^{2m+2k_0}}}.$$
	Inserting it into (\ref{step. main estimate for I}), we conclude. 
	
\end{proof}

Let $H$ be a symmetric real-valued matrix of rank $n$, and $\|H\|$ be its operator norm on $\mathbb{R}^n$: $\|H\|:=\sup_{\|v\|=1}\|Hv\|.$ We adopt the convention that
$$\det(H/2\pi i)^{-1/2}=\det(H/2\pi)^{-1/2}e^{i\pi\mathrm{sgn}(H)/4},$$
where $\mathrm{sgn}(H)$ is the inertia index of $H$. We denote by 
$$\langle H\nabla,\nabla \rangle =\sum_{1\leq i,j\leq n}H_{ij} \partial_i\partial_j,\quad\  e^{i\langle H\nabla,\nabla\rangle}w(t)=\sum_{j=0}^{\infty}\frac{i^j}{j!}\langle H\nabla,\nabla\rangle^jw(t),$$ 
the second-order differential operator and the exponent power of differential operator.

\begin{lemma}\label{lem. asymptotic formula}
	
	Let $w(t)$ be a compactly supported smooth function on $\mathbb{R}^n$, with $\mathrm{Supp}(w)\subset S$. Assume that there exists $X,U>0$, such that  
	$$\mathcal{D}^\alpha w\ll_\alpha \frac{X}{U^{|\alpha|}},$$
	for any $|\alpha|\geq 0$. Let $H$ be a symmetric non-degenerate real-valued matrix of rank $n$. Then for any integer $m\geq 0$, we have
	\begin{equation}
		\left|\int_{\mathbb{R}^n}w(t)e^{i\langle Ht,t\rangle/2}dt-\det(H/2\pi i)^{-1/2}T_m\right|\ll_m|\det(H)|^{-1/2} \|H^{-1}\|^{m+1}\frac{\mathrm{Area}(S)\cdot X}{U^{n+2m+2}}
	\end{equation}
	where the asymptotic expansion is given by
	\begin{equation}
		T_{m}=\sum_{j=0}^{m}(i\langle H^{-1}\nabla,\nabla\rangle/2)^j w(0)/j!.
	\end{equation}
	
\end{lemma}

\begin{proof}
	
	This is essentially given in \cite[Lemma 7.7.3]{Hor2003}. To exhibit the error terms in $L^\infty$-norms, we briefly review the proof. By the Fourier transform for $H$ in \cite[Lemma 7.6.1]{Hor2003}, one has
	$$\int_{\mathbb{R}^n}w(t)e^{i\langle Ht,t\rangle /2}dt=(\det(H/2\pi i))^{-1/2}e^{i\langle H^{-1}\nabla,\nabla\rangle/2}w(0).$$
	Using Fourier inversion formula, the $m$-th error term could be given by 
	$$\small\begin{aligned}
		\mathrm{Err}_m&=e^{i\langle H^{-1}\nabla,\nabla\rangle/2}w(0)-\sum_{j=0}^{m}(i\langle H^{-1}\nabla,\nabla\rangle/2)^j w(0)/j!\\
		&=(2\pi)^{-n}\int_{\mathbb{R}^n}\left(e^{i\langle H^{-1}\xi,\xi\rangle/2}-\sum_{j=0}^{m}\frac{1}{j!}(i\langle H^{-1}\xi,\xi\rangle/2)^j\right) \hat{w}(\xi)d\xi
	\end{aligned}$$
	The inside term could be estimated by Taylor formula
	$$\left|e^{i\langle H^{-1}\xi,\xi\rangle/2}-\sum_{j=0}^{m}\frac{1}{j!}(i\langle H^{-1}\xi,\xi\rangle/2)^j\right| \ll_m \|H^{-1}\|^{m+1}\|\xi\|^{2m+2}.$$				
	For any multi-index $|\alpha|\geq 0$, 
	$$\xi^{\alpha}\hat{w}(\xi)=i^{-|\alpha|}\int_{\mathbb{R}^n}\mathcal{D}^\alpha w(t)e^{i\langle \xi,t\rangle }dt\ll |S|\cdot\sup |\mathcal{D}^\alpha w|\ll_m \frac{\mathrm{Area}(S)\cdot X}{U^{|\alpha|}},$$
	Then we have 
	$$\mathrm{Err}_m\leq \|H^{-1}\|^{m+1}\int_{\mathbb{R}^n}\frac{\|\xi\|^{2m+2}(\frac{1}{U}+\|\xi\|)^{n+1}|\hat{w}(\xi)|}{(\frac{1}{U}+\|\xi\|)^{n+1}}d\xi\ll\|H^{-1}\|^{m+1} \frac{\mathrm{Area}(S)\cdot  X}{U^{n+2m+2}} $$
	and then we conclude. 
	
\end{proof}

\begin{theorem}\label{lem. stationary phase general, stationary point}
	
	In the oscillatory integral \eqref{def. the oscillatory integral}, assume that the derivative bounds hold
	$$\sup_{t\in S}|\mathcal{D}^\alpha w| \ll_\alpha X U^{-|\alpha|},\quad \sup_{t\in S}|\mathcal{D}^\beta \Phi|\ll_\beta YQ^{-|\beta|}.$$
	Suppose that $\nabla\Phi(t)=0$ has a unique solution $t_0\in S$, and that for any $t\in S$ we have 
	\begin{equation}\label{step. derivative bounds 3}
		\|\nabla\Phi(t)\|\gg \frac{Y}{Q^2}\|t-t_0\|,\quad \|H_\Phi(t_0)^{-1}\|^{-1}\gg \frac{Y}{Q^2}.
	\end{equation} 
	Furthermore, we suppose that 
	\begin{equation}\label{step. assumption on parameters}
		U=Q, \quad \mathrm{Area}(S)\asymp Q^n.
	\end{equation}
	Then for any given $m\geq 0$, we have an asymptotic expansion 
	\begin{equation}\label{step. general asymptotic formula}
		I=\int_{\mathbb{R}^n}w(t)e^{i\Phi(t)}dt=\frac{e^{i\Phi(t_0)}}{\det(H_\Phi(t_0)/2\pi i)^{1/2}}\sum_{l=0}^mP_l(t_0)+O\left(\frac{\mathrm{Area}(S)\cdot X}{Y^{m+1}}\right).
	\end{equation}
	Here by writing $\Phi(t)=\Phi(t_0)+\langle H_\Phi(t_0)(t-t_0),t-t_0\rangle+g_\Phi(t)$, $P_l$ are given by
	$$P_0(t_0)=w(t_0),\quad P_l(t_0)=\sum_{j-k=l}\sum_{2j\geq 3k}(i\langle H_\Phi(t_0)^{-1}\nabla,\nabla\rangle/2)^j w(ig_\Phi)^{k}(t_0)/j!k!,$$
	and satisfy that 
	$$\mathcal{D}^{\alpha}P_l(t_0)\ll_{\alpha,l} \frac{X}{Y^l}\cdot\frac{1}{Q^{|\alpha|}}.$$
	
\end{theorem}

\begin{proof}
	
	The proof follows the strategy in \cite[Theorem 7.7.5]{Hor2003}. For potential application and to exhibit the structure, we remain the use of parameters $Q,U,\mathrm{Area}(S)$ until the final result. In the following, we fix an integer $m\geq 0$ to give the $m$-th expansion in the asymptotic formula.

	Firstly, let $\rho$ be a fixed smooth bump function supported on $[-1,1]$ and equal to $1$ on $[-\frac{1}{2},\frac{1}{2}]$. Let $T$ be a parameter to be chosen later. Let $w_0$ be the product of $\rho(\frac{t-t_0}{T})$ and the Taylor expansion of order $2m+2$ at $t_0$ of $w$. Let $\tilde{w}_0=w-w_0$, then $\tilde{w}_0$ has a $2m+2$-order zero at $t_0$. For any $|\alpha|\geq 0$, we have 
	$$\sup_{t\in S}|\mathcal{D}^{\alpha}\tilde{w}_0|\ll_\alpha X\cdot \max\{U^{-|\alpha|},T^{-|\alpha|}\}.$$
	Apply Lemma \ref{lem. stationary phase, with higher order zero}, we obtain
	\begin{equation}\label{step. error term 1}
		\int_{\mathbb{R}^n}\tilde{w}_0(t)e^{i\Phi(t)}dt\ll_m \frac{\mathrm{Area}(S)\cdot  X}{Y^{m+1}}\max\left\{\,\left(\frac{Q}{U}\right)^{2m+2},\left(\frac{Q}{T}\right)^{2m+2},\left(\frac{Q}{U}\right)^{m+1},\left(\frac{Q}{T}\right)^{m+1}\right\}.
	\end{equation}
	\vspace{11pt}
	
	Secondly, we write $\Phi(t)=\Phi(t_0)+\langle H_\Phi(t_0)(t-t_0),t-t_0\rangle/2+g_\Phi(t)$. 
	This $g_\Phi$ has at least 3-order zero at $t_0$, hence $|\mathcal{D}^{\beta}g_\Phi(t)|\ll \|t-t_0\|^{3-|\beta|}\frac{Y}{Q^3}$ for $|\beta|\leq 2$ by the Taylor expansion of $g_\Phi$. In particular, it follows that for any $|\beta|\geq 0$, if $\|t-t_0\|\ll Q$, then 
	$$|\mathcal{D}^\beta g_\Phi(t)|\ll_\beta \frac{Y}{Q^{|\beta|}}.$$
	For $0\leq s\leq 1$, we define
	$$\Phi_s(t)=\Phi(t_0)+\frac{1}{2}\langle H_\Phi(t_0)(t-t_0),t-t_0\rangle +s g_\Phi(t).$$
	We turn to the integral
	$$I(s)=\int_{\mathbb{R}^n}w_0(t)e^{i\Phi_s(t)}dt,$$
	where $w_0(t)$ is supported on $\|t-t_0\|\leq T$, and for any $|\alpha|\geq 0$, the derivatives of $w_0$ satisfy
	$$\sup_{t\in S}|\mathcal{D}^{\alpha}w_0|\ll_\alpha X\cdot \max\{U^{-|\alpha|},T^{-|\alpha|}\}.$$
	Since $\|H_\Phi(t_0)(t-t_0)\|\geq \|H_\Phi(t_0)^{-1}\|^{-1}\|t-t_0\|$, we obtain that when $T\lll Q$,
	\begin{equation}\label{step. dominant by stationary point}
		\|\nabla\Phi_s(t)\|=\|H_\Phi(t_0)(t-t_0)+s\nabla g_\Phi(t)\|\gg \frac{Y}{Q^2}\|t-t_0\|
	\end{equation}
	uniformly for all $0\leq s\leq 1$. Now we differentiate $I(s)$ with respect to $s$, and use the Taylor expansion for $I(s)$ at $s_0=0$, we find
	$$\left|I(1)-\sum_{k=0}^{2m+1}\frac{I^{(k)}(0)}{k!}\right|\leq\frac{1}{(2m+2)!}\sup_{0<s<1}|I^{(2m+2)}(s)|.$$
	The term $I^{(2m+2)}$ is given by
	$$I^{(2m+2)}(s)=\int_{\mathbb{R}^n}w_0(t)(ig_{\Phi}(t))^{2m+2}e^{i\Phi_s(t)}dt.$$
	We note that $w_0g_\Phi^{2m+2}$ has a zero at $t_0$ of order $6m+6$, and the support of $w_0$ requires $\|t-t_0\|\leq T\ll Q$. Therefore for any $|\alpha|\geq 0$, we have 
	$$\mathcal{D}^{\alpha}(w_0g_\Phi^{2m+2})\ll XY^{2m+2}\max\{U^{-|\alpha|},T^{-|\alpha|}, Q^{-|\alpha|}\}.$$
	Apply Lemma \ref{lem. stationary phase, with higher order zero} again, we obtain for any $0<s<1$, 
	\begin{equation}\label{step. error term 2}
		\begin{aligned}
			I^{(2m+2)}(s)\ll_m\frac{\mathrm{Area}(S)\cdot X}{Y^{m+1}}\max\left\{\left(\frac{Q}{U}\right)^{3m+3},\left(\frac{Q}{T}\right)^{3m+3},1\right\}.
		\end{aligned}
	\end{equation}
	\vspace{11pt}
	
	Thirdly, let $G_\Phi(t)$ be the Taylor expansion of $g_\Phi(t)$ at $t_0$ of order $3m+3$, and $\tilde{w}_k=w_0((ig_\Phi)^k-(iG_\Phi)^k)$ for $1\leq k\leq 2m+1$. Then $\tilde{w}_k$ has a $3m+3k+1$-order zero at $t_0$, in particular, at least $2m+2k+2$-order. Furthermore, for any $|\alpha|\geq 0$, we have
	$$\mathcal{D}^\alpha(\tilde{w}_k)\ll_{\alpha,m}XY^k\cdot\max\{U^{-|\alpha|},T^{-|\alpha|},Q^{-|\alpha|}\}.$$
	Apply Lemma \ref{lem. stationary phase, with higher order zero} again, we obtain that 
	\begin{equation}\label{step. error term 3}
		\left|\int_{\mathbb{R}^n}\tilde{w}_k(t)e^{i\Phi_0(t)}dt\right|\ll_{m}\frac{\mathrm{Area}(S)\cdot X}{Y^{m+1}}\max\left\{\left(\frac{Q}{U}\right)^{m+k+1},\left(\frac{Q}{T}\right)^{m+k+1},1\right\}.
	\end{equation}
	\vspace{11pt}
	
	Fourthly, we define $w_k(t)=w_0(t)(iG_\Phi(t))^{k}$, which has $3k$-order zero at $t_0$. For any $|\alpha|\geq 0$ and $\|t-t_0\|\ll Q$, we have
	$$\mathcal{D}^{\alpha}w_k\ll XY^k\cdot\max\{U^{-|\alpha|},T^{-|\alpha|},Q^{-|\alpha|}\}.$$
	Recall that $\Phi_0(t)=\Phi(t_0)+\langle H_\Phi(t_0)(t-t_0),t-t_0\rangle/2$, and we denote by $D=i\langle H_\Phi(t_0)^{-1}\nabla,\nabla\rangle/2$. Then we apply Lemma \ref{lem. asymptotic formula} together with (\ref{step. lower bound for derivative}) to 
	$$I_k=\int_{\mathbb{R}^n}w_0(t)(iG_\Phi(t))^{k}e^{i\Phi_0(t)}dt.$$
	By expanding $m+k+1$ terms, we obtain that
	\begin{equation}\label{step. error term 4}
		\begin{aligned}
			\left|I_k-\frac{e^{i\Phi(t_0)}}{\det(H_\Phi(t_0)/2\pi i)^{1/2}}\sum_{j=0}^{m+k+1}\frac{D^jw_k(t)}{j!}\right|
			\ll |\det(H_{\Phi}(t_0))|^{-1/2} \frac{\mathrm{Area}(S)\cdot X}{Y^{m+1}}\frac{Q^{2m+2k+2}}{U^{n+2m+2k+2}}
		\end{aligned}
	\end{equation}
	\vspace{11pt}
	
	Finally, we collect the main terms by 
	$$L_m(t_0)=\sum_{k=0}^{2m+1}\sum_{j=0}^{m+k+1}\frac{D^jw_k(t_0)}{k!j!}$$
	Then we combine the error terms (\ref{step. error term 1}) (\ref{step. error term 2}) (\ref{step. error term 3}) (\ref{step. error term 4}), and put the condition $U=Q$, $\mathrm{Area}(S)\asymp Q^n$, and we set $T\asymp Q$, then we conclude that 
	$$\left|I-\frac{e^{i\Phi(t_0)}}{\det(H_\Phi(t_0)/2\pi i)^{1/2}}L_m(t_0)\right|\ll \frac{\mathrm{Area}(S)\cdot X}{Y^{m+1}}.$$
	To obtain (\ref{step. general asymptotic formula}), we need to reorganize $L_m(t_0)$. Since $w_k$ has a $3k$-order zero at $t_0$ and $D^j$ is $2j$-order differential operator, it follows that $D^jw_k(t_0)\neq 0$ implies $2j\geq 3k$. Therefore, we can reorganize $L_m$ by 
	$$L_m(t_0)=\sum_{l=0}^{m}P_l(t_0),\quad P_l(t_0)=\sum_{j-k=l}\sum_{2j\geq 3k}(i\langle H_\Phi(t_0)^{-1}\nabla,\nabla\rangle/2)^j w(ig_\Phi)^{k}(t_0)/j!k!.$$
	In particular, consider $P_l(t_0)$ as a function on $t_0$, then for any $l\geq 0$ and $\mathcal{D}^\alpha$ acting on $t_0$, we have
	$$\mathcal{D}^\alpha P_l(t_0)\ll \frac{X}{Y^{l}}\cdot \frac{1}{Q^{|\alpha|}}.$$
	We note that this estimate does not require (\ref{step. assumption on parameters}). Furthermore, when $l=0$, then $j=k$ and $2j\geq 3k$ imlies that $j=k=0$. Hence the leading term is given by $P_0(t_0)=w(t_0)$. 
	
\end{proof}

\begin{remark}
	
	We offer a few remarks on the  results.
	\begin{itemize}
		\item To obtain an asymptotic formula with an explicit leading term, one expands 
		(\ref{step. general asymptotic formula}) to order $m>n/2+1$ and discards all but the first term.  
		This gives
		$$\int_{t\in \mathbb{R}^n} w(t)e^{i\Phi(t)}dt=\frac{e^{i\Phi(t_0)}}{\det(H_\Phi(t_0)/2\pi i)^{1/2}} \left(w(t_0)+O\left(\frac{X}{Y}\right)\right).$$
		
		\item In \cite{Hor2003}, Theorems \ref{lem. stationary phase lemma, no stationary point} and 
		\ref{lem. stationary phase general, stationary point} are established on $\mathbb{R}^n$ for phases of the form $\Phi=\omega\phi$. In this framework the authors derive asymptotic formulas as $\omega\to\infty$.
		The parameter conditions appearing there correspond to the special case $U=Q=R=1$, $\mathrm{Area}(S)\asymp 1$, and $Y=\omega$ in our notation, and (\ref{step. assumption on parameters}) holds whenever $t_0$ is a non-degenerate stationary point.
		
		\item 	In \cite{BKY13} and \cite{KPY19}, the authors treat Theorems~\ref{lem. stationary phase lemma, no stationary point} 
		and \ref{lem. stationary phase general, stationary point} for general phase functions $\Phi$ on $\mathbb{R}$.  
		We extend their results to $\mathbb{R}^n$, in particular to $\mathbb{C}\cong\mathbb{R}^2$; see the next section.  
		Moreover, our condition (\ref{step. lower bound for derivative}) is slightly weaker than their assumption 
		$\Phi''(t)\gg Y/Q^2$, which in particular excludes the possibility that $\Phi''(t)$ vanishes on the support of $w$.
		
		\item Relation (\ref{step. dominant by stationary point}) shows that, throughout the range 
		$T\lll Q$, the lower bound (\ref{step. derivative bounds 3}) continues to hold. In this sense, the region $T\lll Q$ may be regarded as the stationary–phase–dominant range.

		\item The assumptions (\ref{step. derivative bounds}), (\ref{step. derivative bounds 2}), and 
		(\ref{step. derivative bounds 3}) occur frequently in analytic number theory; see \S 2.3.2, \S 2.3.3. The parameter condition (\ref{step. assumption on parameters}) also appears naturally in smooth dyadic partitions; see \S 2.3.4. 
		
	\end{itemize}
	
\end{remark}

\subsubsection{Oscillatory integral for $\mathbb{C}$}

Suppose that $\Phi(z)$ is a compactly supported smooth function on $\mathbb{C}$. By $\mathbb{C}\cong \mathbb{R}^2$, $z=x+iy$, we introduce the usual $\partial_x,\partial_y$ on $\mathbb{C}$. In this case, the gradient and determinant of Hessian matrix are given by 
$$\|\nabla{\Phi}(z)\|^2=(\partial_x\Phi)^2+(\partial_y\Phi)^2,\quad \det (H_\Phi(z))=(\partial_{x}^2\Phi)(\partial_{y}^2\Phi)-(\partial_x\partial_y\Phi)^2.$$
By writing $x=\frac{z+\bar{z}}{2}$, $y=\frac{z-\bar{z}}{2i}$, we introduce $\partial_z=\frac{1}{2}(\partial_x-i\partial_y)$, $\partial_{\bar{z}}=\frac{1}{2}(\partial_x+i\partial_y)$. We have 
\begin{equation}
	\|\nabla \Phi(z)\|^2=4\partial_z\Phi\cdot \partial_{\bar{z}}\Phi,\quad \det(H_\Phi(z))=4\left((\partial_z\partial_{\bar{z}}\Phi)^2-(\partial_z^2\Phi)(\partial_{\bar{z}}^2\Phi)\right).
\end{equation}
Furthermore, if we assume that $\Phi$ is real-valued and harmonic, then we have 
\begin{equation}\label{step. complex differential operator}
	\|\nabla\Phi(z)\|^2=4|\partial_z\Phi(z)|^2,\quad \det(H_{\Phi}(z))=-4|\partial_z^2\Phi(z)|^2
\end{equation}

\vspace{11pt}
The result in Theorem \ref{lem. stationary phase lemma, no stationary point} and Theorem \ref{lem. stationary phase general, stationary point} can be extended to $\mathbb{C}^n$ by identify $\mathbb{C}^n\cong \mathbb{R}^{2n}$, and replace $\{\partial_{t_1},...,\partial_{t_n}\}$ by $\{\partial_{t_1},\partial_{\bar{t}_1},...,\partial_{t_n},\partial_{\bar{t}_n}\}$. We formulate the result for $F=\mathbb{C}$.

\begin{corollary}\label{lem. stationary phase, C}
	Let $w(t)$ be a compactly supported smooth function on $\mathbb{C}$, with $\mathrm{Supp}(w)\subset S$, and $\Phi(t)$ be a real-valued harmonic function on $S$. Let $\mathcal{D}^{\alpha}=\partial_t^{\alpha_1}\partial_{\bar{t}}^{\alpha_2}$. Assume that there exists $X,Y,Z>0$, such that $|S|\asymp |Z|^2$, and for any $|\alpha|\geq 0$ and $j\geq 2$, 
	\begin{equation}\label{step. derivative bounds C}
		\sup_{t\in S}\left|\partial_t^{\alpha_1}\partial_{\bar{t}}^{\alpha_2}w(t)\right|\ll_\alpha\frac{X}{Z^{|\alpha|}},\quad \sup_{t\in S}\left|\partial_t^j\Phi(t)\right|\ll_j\frac{Y}{Z^j}.
	\end{equation}
	Let 
	$$I=\int_{\mathbb{C}}w(t)e^{i\Phi(t)}d_\mathbb{C}t.$$
	\begin{enumerate}[label=\textnormal{(\arabic*)}]
		\item If for any $t\in S$, 
		\begin{equation}\label{step. derivative bounds 1 C}
			\left|\partial_t\Phi(t)\right|\gg \frac{Y}{Z},
		\end{equation}
		then for any $m\geq 0$, we have
		$$I\ll_m \frac{\mathrm{Area}(S)\cdot X}{Y^m}.$$
		\item If $\partial_t\Phi(t)=0$ has a unique solution $t_0\in S$, and 
		\begin{equation}\label{step. derivative bounds 2 C}
			\left|\partial_t\Phi(t)\right|\gg \frac{Y}{Z^2}|t-t_0|,\quad \left|\partial_{tt}^2\Phi(t_0)\right|\gg \frac{Y}{Z^2},
		\end{equation}
		then we have 
		$$I=\frac{e^{ i\Phi(t_0)}}{\left|\partial_{tt}^2\Phi(t_0)/2\pi\right|}\left(w(t_0)+O\left(\frac{X}{Y}\right)\right).$$
		
	\end{enumerate}
\end{corollary}

\begin{proof}
	
	We identify $\mathbb{C}\cong \mathbb{R}^2$. We note that (\ref{step. derivative bounds}) and (\ref{step. derivative bounds C}) are equivalent since $\Phi$ is real-valued and harmonic, and we have linear relation between the differential differential operators:
	$$\partial_z=\frac{1}{2}(\partial_x-i\partial_y),\quad \partial_{\bar{z}}=\frac{1}{2}(\partial_x+i\partial_y),\quad \partial_x=\partial_z+\partial_{\bar{z}},\quad \partial_y=i(\partial_z-\partial_{\bar{z}}).$$
	By (\ref{step. complex differential operator}), we see $\|\nabla\Phi(t)\|=2|\partial_t\Phi(t)|$ and $\det(H_{\Phi}(t))=-4|\partial_{tt}^2\Phi(t)|^2$. We note that $d_\mathbb{C}t$ is twice the Lebesgue measure on $\mathbb{R}^2$. Then we apply Thm \ref{lem. stationary phase lemma, no stationary point} and Thm \ref{lem. stationary phase general, stationary point} and conclude. 
	
\end{proof}

\begin{example}
	
	Let $\chi(z)=\left(\frac{z}{|z|}\right)^m|z|_\mathbb{C}^{i\lambda}$ be a unitary character of $\mathbb{C}^\times$. After fixing a branch of the logarithm, we can write $\chi(z)=e^{i\Phi(z)}$ where the phase function is given by
	$$\Phi(z)=\frac{m}{2i}(\ln (z)-\ln(\bar{z}))+\lambda(\ln(z)+\ln(\bar{z})).$$
	Recall that $T(\chi)=\frac{\frac{m}{2i}+\lambda}{2\pi}$ as in (\ref{def. spectral parameters}). A straightforward calculation gives
	$$\partial_z\Phi(z)=\frac{\frac{m}{2i}+\lambda}{ z}=\frac{2\pi T(\chi)}{z},\quad \partial_{\bar{z}}\Phi(z)=\frac{-\frac{m}{2i}+\lambda}{ \bar{z}}=\frac{2\pi \overline{T(\chi)}}{\bar{z}}.$$
	Furthermore, we have $z\partial_z\chi(z)=2\pi iT(\chi)\chi(z)$, $\bar{z}\partial_{\bar{z}}\chi(z)=2\pi i\overline{T(\chi)}\chi(z)$. 
	
\end{example}

\begin{example}
	Let $\chi(z)=\left(\frac{z}{|z|}\right)^m|z|_\mathbb{C}^{i\lambda}$ be a unitary character of $\mathbb{C}^\times$. Let $S$ be a compact region satisfying that $|z|\asymp Z$. Assume that $p(z)$ is a rational function and has no zero on $S$. We are interested in the following phase function: 
	$$e^{i \Phi(z)}=\chi(p(z))$$
	where $\Phi(t)$ is real-valued and harmonic. Then we have	
	$$\|\nabla \Phi(z)\|=2\pi\left|T(\chi)\cdot \left(\frac{p'(z)}{p(z)}\right)\right|,\quad \det(H_\Phi(z))=-16\pi^2\left|T(\chi)\cdot \left(\frac{p''(z)p(z)-(p'(z))^2}{p(z)^2}\right)\right|^2.$$
	In our application, $p(z)$ usually is fixed while $|T(\chi)|\rightarrow\infty$. If $p'(z)$ has a unique single root $z_0$ on $S$, we see that   $\left|\frac{p^{(j)}(z)}{p(z)}\right|\ll \frac{1}{Z^j}$, $|\frac{p''(z)}{p(z)}|\gg \frac{1}{Z^2}$, and 
	$$\|\nabla \Phi(z)-\nabla\Phi(z_0)\|=4\pi |T(\chi)|\left|\frac{p'(z)-p'(z_0)}{p(z)}\right|\gg 2|T(\chi)||z-z_0|\left|\frac{p''(z_0)}{p(z)}\right|.$$
	In particular, this type of phase function satisfy our conditions in Corollary \ref{lem. stationary phase, C}. 
	
\end{example}

\subsubsection{Generalized Gauss sum}

The generalized Gauss sum (\ref{def. generalized gauss sum}) was studied in \cite[Lemma 3.1.14]{MV10} and \cite[Lemma 4.1]{Wu14}, where estimates for its size were obtained. In this paper, we focus on its oscillatory phase and provide estimates for its derivatives.  

Let $F=\mathbb{R}$ or $\mathbb{C}$, $\deg(F)=[F:\mathbb{R}]$. For a unitary character $\chi$, the conductor of $\chi$ and spectral parameter of $\chi$ are denoted by $C(\chi)$ and $T=T(\chi)$ as in (\ref{def. L-factor and conductor of GL1}) and (\ref{def. spectral parameters}). We define a phase constant by
\begin{equation}\label{step. Phase constant c_chi}
	c_\chi=e^{-\frac{\pi i}{4}(2-\deg(F))}\chi(T)\psi_F(-T).
\end{equation}
Let $\mathcal{D}^{\alpha}=\frac{d^{\alpha}}{dt^{\alpha}}$ when $F=\mathbb{R}$, or $\mathcal{D}^{\alpha}=\frac{\partial^{\alpha_1}}{\partial t^{\alpha_1}}	\frac{\partial^{\alpha_2}}{\partial \overline{t}^{\alpha_2}}$ when $F=\mathbb{C}$ be a degree $|\alpha|$ diffirential operator. The following lemma describe the behavior of these generalized Gauss sums and their derivatives.

\begin{lemma}\label{lem. Gauss sum. }
	
	Let $0<A_0<B_0$ be two fixed positive number, and $V(y)$ be a fixed compactly supported function on $F^\times$, whose support is contained in $A_0\leq |y|\leq B_0$. Let $\chi$ be a unitary character, $\psi_{F}$ be the additive character. A generalized Gauss sum is defined by
	\begin{equation}\label{lem. Gauss sum. definition}
		G(\chi,t)=G_V(\chi,t):=\int_{F^\times}V\left(\frac{y}{t}\right)\chi(y)\psi_F(-y)d^\times_F y.
	\end{equation}
	We have an aysmptotic formula, uniform for all $t\in F^\times$
	\begin{equation}\label{lem. Gauss sum. asymptotic formula}
		G(\chi,t)= \frac{c_\chi}{C(\chi)^{1/2}}V\left(\frac{T(\chi)}{ t}\right)+O\left(\frac{1}{C(\chi)^{1/2+1/\deg(F)}}\right).
	\end{equation}
	Furthermore, there exists fixed $0<A_1<A_0$ and $B_1>B_0$ such that, for any $|\alpha|\geq 0$, 
	\begin{itemize}
		\item When $\left|\frac{T(\chi)}{t}\right|\geq B_1$ or $\left|\frac{T(\chi)}{t}\right|\leq A_1$, for any $N\geq 0$, 
		\begin{equation}\label{lem. Gauss sum. rapidly decay}
			\mathcal{D}^\alpha G(\chi,t)=O\left(\frac{|t|^{-|\alpha|}}{C(\chi)^N(1+|t|)^N}\right).
		\end{equation}
		\item When $A_1\leq \left|\frac{T(\chi)}{t}\right|\leq B_1$, 
		\begin{equation}\label{lem. Gauss sum. bounds for derivatives}
			\mathcal{D}^\alpha G(\chi,t)=O\left(\frac{|t|^{-|\alpha|}}{C(\chi)^{1/2}}\right).
		\end{equation}
	\end{itemize}
	In particular, we have $\mathcal{D}^{\alpha}G(\chi,t)\ll_\alpha\frac{1}{|t|^{|\alpha|}}\frac{1}{C(\chi)^{1/2}}$ for any $t\in F^\times$. 
\end{lemma}

\begin{proof}
	
	We deal with the real and complex case together, except for neccessary discussion.   Let $\chi=\chi_m|\cdot|_F^{i\lambda}$. Recall that $T=T(\chi)$ is $\frac{\lambda}{2\pi}$ when $F=\mathbb{R}$ and $\frac{\lambda+\frac{m}{2 i}}{2\pi}$ when $F=\mathbb{C}$. 
	
	If $F=\mathbb{R}$, we write $\Phi(y)=\frac{\lambda}{2\pi}\ln|y|-y+m\cdot \mathrm{sgn}(y)$. If $F=\mathbb{C}$, we write $\Phi(y)=\frac{m}{2\pi i}(\ln y-\ln\bar{y})+\frac{\lambda}{2\pi}(\ln(y)+\ln(\bar{y}))-(y+\bar{y})$. Then we have $e^{2\pi i\Phi(y)}=\chi(y)\psi_F(-y)$. We have
	$$\frac{\partial\Phi}{\partial y}(y)=\frac{T}{y}-1$$ 
	in both cases. Hence $\|\nabla\Phi(y)\|=0$ has a unique solution $y_0=T$, and $y_0$ belongs to the support of $V(\frac{y}{t})$ if and only if $A_0\leq \left|\frac{T}{t}\right|\leq B_0$. 
	
	Moreover, $\frac{\partial^2\Phi}{\partial y^2}(y)=-\frac{T}{y^2}$, $\frac{\partial^2\Phi}{\partial y^2}(y_0)=-\frac{1}{T}$. Furthermore, we have $\frac{\partial^j\Phi}{\partial y^j}(y)=\frac{(-1)^{j-1}(j-1)!T}{y^j}$. When $y$ locates in the support of $V(\frac{y}{t})$, we have $\frac{\partial^j\Phi}{\partial y^j}(y)\ll_j\frac{|T|}{|t|^j}$, $\frac{\partial^j\Phi}{\partial \bar{y}^j}(y)\ll_j\frac{|T|}{|t|^j}$, for any $j\geq 2$. 
	
	Let $\mathcal{D}^\alpha$ be some differential operators on $t$. The derivatives on $G(\chi,t)$ are given by 
	\begin{equation}\label{step. derivatives of Gauss sum}
		\mathcal{D}^{\alpha} G(\chi,t)=\int_{F^\times} \mathcal{D}^{\alpha} V\left(\frac{y}{t}\right)\chi(y)\psi_F(-y)d^\times_Fy.
	\end{equation}
	Here
	\begin{equation}\label{step. differential operator}
		\mathcal{D}^\alpha V\left(\frac{y}{t}\right)=\frac{1}{|t|^{|\alpha|}}\sum_{\nu=0}^{|\alpha|}p_\nu\left(\frac{y}{t}\right)V^{(\nu)}\left(\frac{y}{t}\right):=\frac{1}{|t|^{|\alpha|}}V_\alpha\left(\frac{y}{t}\right),
	\end{equation}
	where these $V^{(\nu)}$ are $\nu$-derivative of $V$, and they are smooth fuctions with compact support $A_0\leq |\frac{y}{t}|\leq B_0$ and $p_\nu$ are polynomials obtained by induction. Furthermore,   $\frac{\partial^{j}}{\partial y^{j}}\frac{1}{|y|_F}\mathcal{D}^\alpha V(\frac{y}{t})\ll \frac{1}{|t|^{|\alpha|+\deg(F)}}\frac{\partial^j}{\partial  y^j}\frac{|t|_F}{|y|_F}V_\alpha\left(\frac{y}{t}\right)\ll \frac{1}{|t|^{|\alpha|+j+\deg(F)}}$ for any $j\geq 0$. In particular, the flat condition (\ref{step. derivative bounds}) holds for $\Phi$ and $\frac{1}{|y|_F}\mathcal{D}^{\alpha}V\left(\frac{y}{t}\right)$, with $X'=\frac{1}{|t|^{|\alpha|+\deg(F)}}$, $U'=|t|$, $Y'=T$, $Q'=|t|$, $\mathrm{Area}(S)\asymp |t|_F$. 
	\vspace{11pt}
	
	Now we fix $0<A_1<A_0$, and $B_1>B_0$. Firstly, we consider $\left|\frac{T}{t}\right|\geq B_1$ or $\left|\frac{T}{t}\right|\leq A_1$. The support of $V(\frac{y}{t})$ requires that $A_0\leq \left|\frac{y}{t}\right|\leq B_0$. Then there exists an $R_1>0$ such that for any $\left|\frac{T}{t}\right|\geq B_1$, $\|\nabla\Phi(y)\|\geq \frac{|T|}{|t|B_0}-1\geq R_1\left(1+\frac{|T|}{|t|}\right)$, and for any $|\frac{T}{t}|\leq A_1$, $\|\nabla\Phi(y)\|\geq 1-\frac{|T|}{|t|A_0}\geq R\left(1+\frac{|T|}{|t|}\right)$. In fact, we can pick $R_1=\min\left\{\frac{\frac{B_1}{B_0}-1}{1+B_1},\frac{1-\frac{A_1}{A_0}}{1+A_1}\right\}$. Then by setting $R'=R_1\left(\frac{1}{|t|}+\frac{1}{|T|}\right)$ and applying our parameters to Theorem \ref{lem. stationary phase lemma, no stationary point}, we get 
	$$\frac{1}{|t|^{|\alpha|}}\int_{F^\times}V_\alpha\left(\frac{y}{t}\right)\chi(y)\psi_F(-y)d^\times_Fy\ll_N\frac{\frac{1}{|t|^{|\alpha|}}}{(|t|+|T|)^N}$$ 
	Combine with the trivial bound $\mathcal{D}^\alpha G(\chi,t)\ll \frac{1}{|t|^{|\alpha|}}$, we conclude (\ref{lem. Gauss sum. rapidly decay}).
	
	As for $A_1\leq \left|\frac{T}{t}\right|\leq B_1$, the solution $y_0=T$ might locate in the support of $V(\frac{y}{t})$. We check
	$$\|\nabla\Phi(y)\|=\deg(F)\left|\frac{T(y-y_0)}{yy_0}\right|\gg \frac{|T|}{|t|^2}|y-y_0|,\quad |H_\Phi(y_0)|\gg \frac{T}{|t|^2},$$
	hold for all $A_0\leq \left|\frac{y}{t}\right|\leq B_0$. Hence the condition in (\ref{step. derivative bounds 3}), and (\ref{step. assumption on parameters}) holds, and we can apply the asymptotic formula in lemma \ref{lem. stationary phase general, stationary point}. Since the phase function is $2\pi \Phi$, we have the Hessal is given by $\det(H_{2\pi\Phi}(y_0)/(2\pi i))=i^{-2+\deg(F)}\cdot \deg(F)^2\frac{1}{|T|_F^2}$, and the phase is absorbed to $c_\chi$ in (\ref{step. Phase constant c_chi}). We also note that $d_\mathbb{C}$ is twice the Lebesgue measure on $\mathbb{C}$, which cancel the $\deg(F)$. Finally we conclude that 
	$$\mathcal{D}^\alpha G(\chi,t)=\frac{1}{|t|^{|\alpha|}}\left(\frac{c_\chi}{C(\chi)^{1/2}}V_\alpha\left(\frac{T}{t}\right)+O\left(\frac{1}{C(\chi)^{1/2+1/\deg(F)}}\right)\right).$$
	In particular, when $\alpha=0$, $V_0(\frac{y_0}{t})=V(\frac{T}{y})$, which gives (\ref{lem. Gauss sum. asymptotic formula}). For general $|\alpha|\geq 0$, we note that $V_\alpha$ is a compactly supported function, hence has finite $L^\infty$-norm. We conclude (\ref{lem. Gauss sum. bounds for derivatives}) holds.  
	
\end{proof}

\subsubsection{Non-compact oscillatory integrals}

This section is prepared for \S 3.3. For conveniece, we forcus on the case $F=\mathbb{R}$ or $\mathbb{C}$ and we write $n=\deg(F)$. We are interested in the oscillatory integrals in form of  
$$I=\int_{F^\times}w(t)e^{i\Phi(t)}d^\times_F t,$$
where $\mathrm{Supp}(w)$ is not necessarily compact. Again, by providing a uniform estimate for a family of $w(t)$ and $\Phi(t)$, we aim to obtain a uniform estimate for $I$.

\begin{proposition}\label{prop. non-compact support stationary phase, rapidly decay}
	
	Let $w(t)$ be smooth functions on $F^\times$, and $\Phi(t)$ be real-valued smooth functions or harmonic functions on $\mathrm{Supp}(w)$. Assume that there exists $X(t)>0$, $Y>0$, such that for any $|\alpha|\geq 0$,
	\begin{equation}\label{step. non-cpt function condition 1}
		\mathcal{D}^{\alpha}w(t)\ll_{\alpha}\frac{X(t)}{|t|^{|\alpha|}},\quad \textnormal{with }\|X(t)\|_{L^1(F^\times)}:=\int_{F^\times }X(t)d^\times_F t<\infty,
	\end{equation}
	and for any $|\beta|\geq 2$, $t\in \mathrm{Supp}(w)$,
	\begin{equation}\label{step. non-cpt derivative condition 1}
		|\partial_t\Phi(t)|\gg \frac{Y}{|t|},\quad \mathcal{D}^\beta\Phi(t)\ll_\beta \frac{Y}{|t|^{|\beta|}}. 
	\end{equation} 
	Then for any integer $m\geq 0$, we have 
	\begin{equation}\label{step. non-compact, bound}
		I=\int_{F^\times} w(t)e^{i\Phi(t)}d_F^\times t\ll_m \frac{\|X(t)\|_{L^1(F^\times)}}{Y^m}.
	\end{equation}
	
\end{proposition}

\begin{proof}
	
	Let $Z$ be a parameter, and $\rho$ a non-negative compactly supported smooth function on $F^\times$, supported on $[1/Z,Z]$, such that the smooth dyadic partition holds
	\begin{equation}\label{step. dyadic partition 1}
		\sum_{k\in \mathbb{Z}}\rho\left(\frac{t}{Z^k}\right)=1.
	\end{equation}
	Then we have 
	\begin{equation}\label{step. dyadic partition 2}
		\int_{F^\times}w(t)e^{i\Phi(t)}dt=\sum_{k\in \mathbb{Z}}\int_{F^\times}w(t)\rho\left(\frac{t}{Z^k}\right)e^{i\Phi(t)}dt.
	\end{equation}
	For $t\sim Z^k$, we have $\mathrm{Area}(\mathrm{Supp}\left(\rho\left(\frac{t}{Z^k}\right)\right))\asymp Z^{nk}$, 
	$$\mathcal{D}^\alpha \left(w(t)\rho\left(\frac{t}{Z^k}\right)\frac{1}{|t|_F}\right)\ll_\alpha \frac{1}{Z^{k|\alpha|}}\frac{X(t)}{Z^{nk}}.$$
	for any $|\alpha|\geq 0$, and 
	$$\mathcal{D}^\beta \Phi(t)\ll_\beta \frac{Y}{Z^{k|\beta|}},\quad |\partial_t \Phi(t)|\gg \frac{Y}{Z^k},$$
	for $|\beta|\geq 2$. Apply them to (\ref{step. expansion of L operator}) and (\ref{step. estimate for higher derivatives of Phi}) as in proof of theorem \ref{lem. stationary phase lemma, no stationary point}, we obtain 
	$$\int_{F^\times} w(t)\rho\left(\frac{t}{Z^k}\right)e^{i\Phi(t)}dt\ll_m \int_{t\sim Z^k}\frac{X(t)}{Y^m}d^\times_F t.$$
	Here $t\sim Z^k$ means that $t\in \mathrm{Supp}(\rho(\frac{t}{Z^k}))$. Take summation over the dyaidc partition, the bound (\ref{step. non-compact, bound}) follows. 
	
\end{proof}


\begin{proposition}\label{prop. non-compact support stationary phase, stationary point}
	Let $w(t)$ be smooth functions on $F^\times$, and $\Phi(t)$ be real-valued smooth functions or harmonic functions on $\mathrm{Supp}(w)$. Assume that there exists $X(t)>0$, such that for any $|\alpha|\geq 0$, 
	\begin{equation}\label{step. non-cpt function condtion 2}
		\mathcal{D}^{\alpha}w(t)\ll_{\alpha}\frac{X(t)}{|t|^{|\alpha|}},\quad \textnormal{with }  \|X(t)\|_{L^1(F^\times)}:=\int_{F^\times}X(t)d_F^\times t<\infty.
	\end{equation}
	Furthermore, assume that $\nabla\Phi(t)=0$ has a unique solution $t_0$ on $\mathrm{Supp}(w)$, and for any $|\beta|\geq 2$, 
	\begin{equation}\label{step. non-cpt derivative condition 2}
		|\partial_t\Phi(t)|\gg \frac{Y}{|t|^2|t_0|}|t-t_0|,\quad \mathcal{D}^\beta\Phi(t)\ll_\beta\frac{Y}{|t|^{|\beta|}}\max\left\{\frac{1}{|t|},\frac{1}{|t_0|}\right\}.
	\end{equation}
	Then for any integer $m\geq 0$, we have
	\begin{equation}\label{step. non-compact support stationary phase, stationary point}
		I=\int_{F^\times}w(t)e^{i\Phi(t)}d^\times_F t\ll_m \frac{\sup_{t\sim t_0} X(t)}{(1+|\frac{Y}{t_0}|_F)^{1/2}}+\frac{\|X(t)\|_{L^1(F^\times)}}{(1+|\frac{Y}{t_0}|_F)^m}.
	\end{equation}
	
\end{proposition}

\begin{proof}
	
	As in Prop \ref{prop. non-compact support stationary phase, rapidly decay}, let $Z$ be a parameter, and $\rho$ a non-negative compactly supported smooth function on $F^\times$, supported on $[1/Z,Z]$,  such that the smooth dyadic partition (\ref{step. dyadic partition 1}) and (\ref{step. dyadic partition 2}) hold. 
	
	For $k>k_0$, and $t\sim Z^k$, since $|t-t_0|\gg |t|$, we have 
	$$\|\nabla\Phi(t)\|\gg \frac{\frac{Y}{Z^{k_0}}}{Z^{k}},\quad \mathcal{D}^\beta\Phi(t)\ll_\beta \frac{\frac{Y}{Z^{k_0}}}{Z^{|\beta|k}}.$$
	As in Theorem \ref{lem. stationary phase lemma, no stationary point} and Theorem \ref{prop. non-compact support stationary phase, rapidly decay}, we have
	$$I_k\ll_m \frac{1}{(1+Y/Z^{k_0})^m}\int_{t\sim Z^k}X(t)dt.$$
	
	For $k<k_0$, and $t\sim Z^k$, since $|t_0-t|\gg |t_0|$, we have 
	$$|\partial_t\Phi(t)|\gg \frac{\frac{Y}{Z^{k}}}{Z^{k}},\quad \mathcal{D}^\beta\Phi(t)\ll \frac{\frac{Y}{Z^{k}}}{Z^{|\beta|k}}.$$
	Again as in Theorem \ref{lem. stationary phase lemma, no stationary point} and Theorem \ref{prop. non-compact support stationary phase, rapidly decay}, we have
	$$I_k\ll_m \frac{1}{(1+Y/Z^{k})^m}\int_{t\sim Z^k}X(t)dt\ll_m \frac{1}{(1+Y/Z^{k_0})^m}\int_{t\sim Z^k}X(t)dt.$$
	
	For $k=k_0$ and $t\sim Z^{k_0}$, we have 
	$$|\partial_t\Phi(t)|\gg \frac{\frac{Y}{Z^{k_0}}}{(Z^{k_0})^2}|t-t_0|,\quad |\partial_{tt}\Phi(t)|\gg \frac{\frac{Y}{Z^{k_0}}}{(Z^{k_0})^2},\quad \mathcal{D}^\beta\Phi(t)\ll \frac{\frac{Y}{Z^{k_0}}}{(Z^{k_0})^{|\beta|}}.$$
	Applying Theorem \ref{lem. stationary phase general, stationary point} or Corollary \ref{lem. stationary phase, C}, and we note that 
	$$I_{k_0}\ll \frac{\sup_{t\sim Z^{k_0}} \frac{X(t)}{|t|_F}}{|\partial_{tt}\Phi(t_0)|_F^{1/2}}\ll \frac{\sup_{t\sim Z^{k_0}}X(t)}{(1+Y/Z^{k_0})^{\deg(F)/2}}.$$
	Collect these results and apply $|t_0|\sim Z^{k_0}$, and then we conclude. 
	
\end{proof}

\section{Stationary phase analysis for analytic newvectors}

In this section, we provide an alternative explicit construction for the analytic newvectors of the tempered principal series, drawing parallels with $p$-adic newforms. We study their translated Whittaker function via the Jacquet integral in $\S 3.1$, and analyze their Mellin components in $\S 3.3$ by applying stationary phase methods. Subsequently, in $\S 3.2$ and $\S 3.4$, we apply these constructions to the local Rankin-Selberg integrals and prove that they serve as suitable test vectors for the problem stated in ($\ref{Intro. local test vector problem}$).

\subsection{Analytic newvectors for tempered unitary principal seires}

	Let $F=\mathbb{R}$ or $\mathbb{C}$. For a unitary character $\chi$, we denote its conductor by $C(\chi)$ and its spectral parameter by $T=T(\chi)$, as defined in  (\ref{def. L-factor and conductor of GL1}) and (\ref{def. spectral parameters}), respectively. The induced representations $\pi=\mathcal{I}(\chi)$ consist of smooth functions $f$ on $G$ that satisfy
	\begin{equation}\label{step. translation law}
		f(a(y)n(x)g)=\chi(y)|y|_F^{1/2}f(g),\quad y\in F^\times, x\in F.
	\end{equation}
	These representations are precisely all the tempered principal series for $\mathrm{PGL}_2(F)$. The archimedean Atkin-Lehner operator $w_T$ is defined as
	\begin{equation}\label{def. Atkin-Lehner}
		w_T:=\begin{pmatrix} 0 & 1 \\ -\frac{1}{T^2} & 0\end{pmatrix}.
	\end{equation}
	\vspace{11pt}

	We now introduce the notion of \textbf{a family of analytic newvectors associated to $h(t)$}.  Whenever we refer to a family of analytic newvectors, it is understood that the function $h(t)$ (together with its support parameter $a$) has been fixed once and for all.  
	Accordingly, any implicit constant in an estimate involving analytic newvectors is understood to depend only on this fixed choice of $h(t)$.

\begin{definition}\label{def. analytic new}
	
	Let $h(t)$ be a fixed function on $F^\times$ satisfying the following conditions:
	\begin{enumerate}[label=\textnormal{(\roman*)}]
		\item $h(t)$ is a smooth compactly supported non-negative function on $F^\times$.
		\item $h(t)$ is a radial and central symmetric function, i.e. $h(t)=h(\frac{1}{t})=h(|t|)$ for any $t\in F^\times$. 
		\item $h(t)$ is supported on $\frac{1}{1+a}\leq |t|\leq 1+a$ for some $a>0$. 
		\item $h(t)$ has normalized $L^2$-norm corresponding to the measure $d_F^\times t$, i.e. $\int_F |h(t)|^2 \frac{d_F t}{|t|_F}=1.$
	\end{enumerate}
	For each tempered principal series representation $\pi=\mathcal{I}(\chi)$, we define a vector $f_\chi\in \mathcal{I}(\chi)$ by 
	\begin{equation}\label{def. analytic newvector}
		f_\chi(wn(t))=C_{F,\chi}\cdot h\left(\frac{t}{T}\right)\chi^{-1}\left(\frac{t}{T}\right)\left|t\right|^{-\frac{1}{2}}_F.
	\end{equation}
	where $C_{F,\chi}=\chi(-1)e^{\left(-\frac{\pi}{4}\right)i}\psi_F(-T)$ if $F=\mathbb{R}$ and $C_{F,\chi}=\chi(-1)\psi_F(-T)$ if $F=\mathbb{C}$. 
	
\end{definition}

\begin{proposition}[First properties of analytic newvectors]
	
	Let $f_\chi\in \mathcal{I}(\chi)$ be a family of analytic newvectors defined by (\ref{def. analytic newvector}). Then 
	\begin{enumerate}[label=\textnormal{(\arabic*)}]
		\item Each $f_\chi$ is a unit smooth vector in $\mathcal{I}(\chi)$. 
		\item The vectors $f_\chi$ satisfy the equivariance relation 
		\begin{equation}\label{prop. Symmetric relation f}
			\pi(w_T)f_{\chi}=\chi(-1)f_{\chi}.
		\end{equation}
		\item The value of $f_\chi$ on lower unipotent elements is given by
		\begin{equation}\label{prop. f(n'(x))}
			f_\chi(n'(x))=\chi(-T^2)C_{F,\chi}\cdot h\left(Tx\right)\chi^{-1}\left(Tx\right)|x|_F^{-\frac{1}{2}}.
		\end{equation}
	\end{enumerate}
	
\end{proposition}

\begin{proof}
	
	Smoothness follows from Lemma \ref{lem. smooth vectors}, since $f_\chi(wn(t))$ is supported on a compact set and depends smoothly on $t$. Its $L^2$-norm of $f_\chi$ is computed by $\int_{F}|f_\chi(wn(t))|^2d_Ft=\int_{F}h(\frac{t}{T})^2|t|^{-1}d_Ft=1$ by the normalization of $h(t)$. 
	
	For (\ref{prop. Symmetric relation f}), we use the matrix identity
	$$wn(t)w_T=\begin{pmatrix} \frac{1}{T^2} & 0 \\ -\frac{t}{T^2} & 1\end{pmatrix}=\begin{pmatrix} -\frac{1}{t} & \frac{1}{T^2} \\ 0 & -\frac{t}{T^2}\end{pmatrix}\begin{pmatrix} 0 & -1 \\ 1 & 0\end{pmatrix}\begin{pmatrix} 1 & -\frac{T^2}{t} \\ 0 & 1\end{pmatrix}.$$
	From the transformation law (\ref{step. translation law}) in $\mathcal{I}(\chi)$, this yields
	$$\begin{aligned}	\pi(w_T)f_\chi(wn(t))&=\chi\left(\frac{T^2}{t^2}\right)\left|\frac{T}{t}\right|_Ff\left(wn\left(-\frac{T^2}{t}\right)\right)\\
		&=C_{F,\chi}\chi\left(-\frac{T}{t}\right)h\left(-\frac{T}{t}\right)|t|_F^{-\frac{1}{2}}=\chi(-1)f_\chi(wn(t)).
	\end{aligned}$$
	Here we use that $h(-\frac{T}{t})=h(\frac{t}{T})$.  As $\pi(w_T)f_\chi$ and $\chi(-1) f_\chi$ agree on the big Bruhat cell
	$BwN$, we conclude that the identity holds everywhere.
	
	Finally using $n'(x)=wn(-x)w^{-1}$, we compute
	$$f_\chi(n'(x))=f_\chi(wn(-x)w^{-1})=f_\chi\left(wn(-x)a(T^{-2})w_T\right)=f_\chi(a(T^2)wn(-T^2x)w_T)$$
	and the formula (\ref{prop. f(n'(x))}) follows from the definition	(\ref{def. analytic newvector}).
\end{proof}

\begin{lemma}\label{lem. smooth vectors}
	
	For any $f\in \mathcal{I}(\chi)$, we define $F_f(t)=f(wn(t))$. If $f$ is a smooth vector in $\mathcal{I}(\chi)$, then $F_f(t)$ is a smooth function on $F$. Conversly, given any smooth compactly supported function $F(t)$ on $F$, there exists a smooth vector
	$f\in \mathcal{I}(\chi)$ such that $F=F_f$.
	
\end{lemma}

\begin{proof}
	
	The first statement is clear from the definition of $F_f$, since $f$ and the map $t\mapsto wn(t)$ are both smooth. For the converse, we note that the big Bruhat cell $BwN$ is a dense open subset of $G$. A function $f$ can be defined on this cell via the transformation property using $F$. The compact support of $F(t)$ ensures that $f$ vanishes in a neighborhood of $e$ and then the Borel subgroup $B=G-BwN$. Thus, $f$ is smooth on all of $G$.
	
\end{proof}

\begin{remark}
	
	When $F$ is non-archimedean, $\chi$ is a ramified character of $F^\times$, and
	$\pi=\mathcal{I}(\chi)$ is the induced representation, the theory of $p$-adic newforms
	is well understood. In this setting, the newform $f_\pi$ satisfies the
	Atkin-Lehner equivariance
	$$\pi(w_\pi) f_\pi = \chi(-1)\, f_\pi,
	\qquad
	w_\pi :=
	\begin{pmatrix}
		0 & 1 \\ -\varpi^{c(\pi)} & 0
	\end{pmatrix}.$$
	Its values on lower unipotent elements are given by
	$$f_\chi(n'(x))=\begin{cases}
		\chi^{-1}(x) & \text{if }v(x)=c(\chi);\\
		0 & \text{otherwise}.
	\end{cases}$$
	as in \cite[(21)]{Sch02}. Furthermore, the Whittaker function of the newform satisfies $$W_{f_\chi}(a(y))=\mathrm{1}_{\mathfrak{o}^\times}(y),\quad y\in F^\times.$$ 
	In our construction, the bump function $h(t)$ plays the role of the characteristic
	function $\mathbf{1}_{\mathfrak{o}^\times}$ in the non-archimedean theory.
	Thus, the formulas (\ref{prop. Symmetric relation f}),
	(\ref{prop. f(n'(x))}), and the later
	(\ref{cor. Whittaker function for analytic newvectors})
	are directly comparable to the corresponding statements above.

\end{remark}

We now turn to the study of the Whittaker functions associated with our analytic newvectors.  
For $f \in \mathcal{I}(\chi)$, the Whittaker function of $f$ is given by the Jacquet integral
\begin{equation}\label{def. Jacquet integral}
	W_{f}(a(y))
	=
	\chi^{-1}(y)|y|_F^{\frac{1}{2}}
	\int_{F}^{\mathrm{reg}} 
	f(wn(t))\, \overline{\psi_F(ty)}\, d_F t .
\end{equation}
This integral is absolutely convergent when $\Re(\chi)>0$, and in general is defined by analytic continuation; it may also be interpreted as a regularized integral in the sense of \cite{Nel19}.  
In particular, it is absolutely convergent whenever $f(wn(t))$ has compact support. We introduce the \textbf{translated Whittaker functions} by
\begin{equation}\label{def. translated Whittaker functions}
	W_f^{(z)}(a(y))
	:=
	W_f\left(
	\begin{pmatrix} y & 0 \\ \frac{z}{T} & 1 \end{pmatrix}
	\right)
	=
	W_f\left(
	\begin{pmatrix} y & 0 \\ 0 & 1 \end{pmatrix}
	\begin{pmatrix} 1 & 0 \\ \frac{z}{T} & 1 \end{pmatrix}
	\right).
\end{equation}

	In the following, we study the translated Whittaker functions of analytic newvectors.  
	Recall that $h(t)$ has compact support in the region $\frac{1}{1+a} \leq |t| \leq 1+a$, and we fix a constant $0 < b < \frac{1}{1+a}$.
	
	\begin{proposition}[Further properties of analytic newvectors]
		Let $f_\chi\in \mathcal{I}(\chi)$ be a family of analytic newvectors associated to $h(t)$. Then 
		\begin{equation}\label{def. Phi for n'(z)f}
			f_\chi\left(wn(t)n'\left(\frac{z}{T}\right)\right)=C_{F,\chi}\cdot\chi\left(\frac{1}{\frac{t}{T}(1+\frac{zt}{T})}\right)\left|\frac{1}{\frac{t}{T}(1+\frac{zt}{T})}\right|_F^{1/2}h\left(\frac{\frac{t}{T}}{1+\frac{zt}{T}}\right)|T|_F^{-\frac{1}{2}}.
		\end{equation}
		Furthermore, 
		\begin{enumerate}[label=\textnormal{(\arabic*)}]
			\item If $|z|\leq b$, then $f_\chi\left(wn(t)n'\left(\frac{z}{T}\right)\right)$ is a smooth compactly supported function on $F^\times$. In particular, the Jacquet integral
			\begin{equation}\label{def. Jacquet intergal n'(z)}
				W_{f_\chi}^{(z)}(a(y))=\chi^{-1}(y)|y|_F^{\frac{1}{2}}\int_F f_\chi\left(wn(t)n'\left(\frac{z}{T}\right)\right)\overline{\psi_F(ty)}d_Ft	,
			\end{equation}
			is absolutely convergent. 
			
			\item For any unit $u\in F^\times$ with $|u|=1$, we have 
			\begin{equation}\label{prop. invariance under unit}
				W_{f_\chi}^{(zu)}(a(y))=W_{f_\chi}^{(z)}(a(yu^{-1})).
			\end{equation}
			
			\item If $z\neq 0$, the Whittaker function satisfies the symmetric relation
			\begin{equation}\label{prop. symmetric of Whittaker function}
				W^{(z)}_{f_\chi}(a(y))=\chi(-1)\psi_F\left(\frac{yT}{z}\right)W^{(-\frac{1}{z})}_{f_\chi}\left(a\left(\frac{y}{z^2}\right)\right)
			\end{equation}
			
		\end{enumerate}
	\end{proposition}
	
	\begin{proof}
		
		Using the matrix identity
		$$w\begin{pmatrix} 1 & t \\ 0 & 1\end{pmatrix}\begin{pmatrix} 1 & 0 \\ x & 1\end{pmatrix}=\begin{pmatrix} \frac{1}{1+xt} & -x \\ 0 & 1+xt\end{pmatrix}w\begin{pmatrix} 1 & \frac{t}{1+xt} \\ 0 & 1\end{pmatrix}$$
		we obtain 
		$$f(wn(t)n'(x))=\frac{\chi^{-2}(1+tx)}{|1+tx|_F}f\left(wn\left(\frac{t}{1+xt}\right)\right).$$
		Equation (\ref{def. Phi for n'(z)f}) then follows by inserting the definition of the analytic newvector (\ref{def. analytic newvector}).
		
		For (1), the support of $f_\chi\left(wn(t)n'\left(\frac{z}{T}\right)\right)$ is the same as the support of  $h(\frac{\frac{t}{T}}{1+\frac{zt}{T}})$. Thus $t$ lies in its support exactly when  
		$$\frac{1}{(1+a)|z|}\leq\left|1+\frac{1}{\frac{zt}{T}}\right|\leq\frac{1+a}{|z|}.$$
		Since $|z|\leq b$, we have $|z(1+a)|<1$, so $t$ remains in a compact subset of $F^\times$. Therefore the support is compact and (\ref{def. Jacquet intergal n'(z)}) is absolutely convergent.
		
		For (2), we note that $f_\chi$ is invariant under these $a(u)$, and use the matrix identity
		$$\begin{pmatrix}yu^{-1} & 0 \\ \frac{z}{T} & 1\end{pmatrix}\begin{pmatrix}u & 0 \\ 0 & 1\end{pmatrix}=\begin{pmatrix}y & 0 \\ \frac{zu}{T} & 1\end{pmatrix}.$$

		For (3), we apply the symmetry relation (\ref{prop. Symmetric relation f}) together with the matrix identity
		$$\begin{pmatrix} y& 0 \\ \frac{z}{T} & 1\end{pmatrix}\begin{pmatrix} 0& 1 \\ -\frac{1}{T^2} & 0\end{pmatrix}=\begin{pmatrix} 0 & y \\ -\frac{1}{T^2} & \frac{z}{T}\end{pmatrix}=\begin{pmatrix} \frac{z}{T}& 0 \\ 0 & \frac{z}{T}\end{pmatrix}\begin{pmatrix} 1& \frac{yT}{z} \\ 0 & 1\end{pmatrix}\begin{pmatrix} \frac{y}{z^2}& 0 \\ -\frac{1}{T z} & 1\end{pmatrix},$$
		then (\ref{prop. symmetric of Whittaker function}) follows. 
		
	\end{proof}

\begin{theorem}[Explicit formula for translated Whittaker function for analytic newvectors]\label{thm. Explicit formula for translated Whittaker function}
	Let $F=\mathbb{R}$ or $\mathbb{C}$, and let $\chi$ be a unitary character of $F^\times$. Let $f_\chi\in \mathcal{I}(\chi)$ be a family of analytic newvectors associated to $h(t)$. Then, for any given $0<b<\frac{1}{1+a}$, there exist constants $0<Y_1<1<Y_2$ such that, for any fixed $N>0$, the bounds below hold uniformly for all $|z|\leq b$ and $\chi$:
	\begin{itemize}
		\item If $|y|>Y_2$, 
		\begin{equation}\label{step. Whittaker function, y>Y_2}
			W_{f_\chi}^{(z)}(a(y))=O\left(\frac{|y|_F^{-N}}{C(\chi)^N}\right);
		\end{equation}
		\item If $0<|y|<Y_1$,
		\begin{equation}\label{step. Whittaker function, y<Y_1}
			W_{f_\chi}^{(z)}(a(y))=O\left(\frac{|y|_F^{\frac{1}{2}}}{C(\chi)^N}\right);
		\end{equation}
		\item If $Y_1\leq |y|\leq Y_2$, we have an asymptotic expansion 
		\begin{equation}\label{step. Whittaker function, Y_1<y<Y_2}
			W_{f_\chi}^{(z)}(a(y))=e^{i\Phi_\chi(zt_0)}h(t_0)\left|\frac{1-z^2t_0^2}{1+z^2t_0^2}\right|_F^{\frac{1}{2}}+O\left(\frac{1}{C(\chi)^{1/\deg(F)}}\right),
		\end{equation}

	\end{itemize}
	where $t_0=t_0(y,z)$ is the unique solution of $1-z^2t^2+ty=0$ satisfying $|zt_0|<1$, and the phase factor is given by 
	\begin{equation}\label{step. Whittaker function, Phase}
		e^{i\Phi_\chi(zt_0)}=\psi_F(T zt_0)\chi\left(\frac{1-zt_0}{1+zt_0}\right).
	\end{equation}
	The error term is bounded uniformly for all $|z|\leq b$ and $y\in F^\times$. 
\end{theorem}

\begin{corollary}[Whittaker function for analytic newvectors] 
	
	Let $F=\mathbb{R}$ or $\mathbb{C}$. Let $f_\chi\in \mathcal{I}(\chi)$ be a family of analytic newvectors associated to $h(t)$. Then
	\begin{equation}\label{cor. Whittaker function for analytic newvectors}
		W_{f_\chi}(a(y))=h(y)+O\left(\frac{1}{C(\chi)^{1/\deg(F)}}\right).
	\end{equation}
	The error term is bounded uniformly for all  $y\in F^\times$.
	
\end{corollary}

\begin{proof}
	
	This follows directly from Theorem\ref{thm. Explicit formula for translated Whittaker function} by taking $z=0$.  
	In this case, $t_0 = -\frac{1}{y}$, and since $h(t) = h(-1/t)$, the claim follows.
	
\end{proof}

Before we prove the theorem, we can do some comparasion with our known results about p-adic newform and Jana-Nelson's analytic newvector.

\begin{remark}
	
	As outlined in the introduction, this theorem provides the asymptotic behavior of the translated Whittaker function. Specifically, the function is shown to approximate a smooth bump function up to a controlled error estimate. This result is comparable to similar findings in the literature concerning automorphic forms, such as those established in \cite{MV10, JN19, BJN23}. For context, the $p$-adic analogues of this result have been well-studied in a series of works, including \cite{Sch02, Hu17, Hu18, HMN22}.
	
\end{remark} 

\vspace{11pt}

\begin{proof}[Proof of Theorem \ref{thm. Explicit formula for translated Whittaker function}]
	
	By substituting the expression (\ref{def. Phi for n'(z)f}) for $f\left(wn(t)n'\left(\frac{z}{T}\right)\right)$ into the Jacquet integral (\ref{def. Jacquet intergal n'(z)}) and applying the change of variable
	$$t'=\frac{\frac{t}{T}}{1+\frac{zt}{T}},\quad t=\frac{Tt'}{1-zt'},\quad d_Ft=\left|\frac{T}{(1-zt')^2}\right|_Fd_Ft',$$
	we obtain the following integral representation for the Whittaker function:
	$$\begin{aligned}
		W_{f_\chi}^{(z)}(a(y))&=C_{F,\chi}\cdot\chi^{-1}(y)|y|_F^{\frac{1}{2}}\int_{F}\chi\left(\frac{1}{\frac{t}{T}(1+\frac{zt}{T})}\right)\left|\frac{1}{\frac{t}{T}(1+\frac{zt}{T})}\right|_F^{\frac{1}{2}}h\left(\frac{\frac{t}{T}}{1+\frac{zt}{T}}\right)|T|_F^{-\frac{1}{2}}\psi_F(-ty)d_Ft\\
		&=C_{F,\chi}\cdot \chi^{-1}(y)|Ty|_F^{\frac{1}{2}}\int_F\chi\left(\frac{(1-zt)^2}{t}\right)\psi_F\left(-\frac{Tty}{1-zt}\right)\frac{h(t)|t|_F^{-\frac{1}{2}}}{|1-zt|_F}d_Ft.
		\end{aligned}$$
	Note that $h(t)$ has support $\frac{1}{1+a}\leq|t|\leq 1+a$ and $|z|\leq b<\frac{1}{1+a}$. Hence these $\frac{h(t)|t|_F^{-1/2}}{|1-zt|_F}$  have uniform compact support and their value and derivatives are uniformly bounded. As for the phase function, we denote by 
	$$e^{i\Phi(t)}=\chi\left(\frac{(1-zt)^2}{t}\right)\psi_F\left(-\frac{Tty}{1-zt}\right).$$
	Then in the support of $h(t)$, we have $\Phi(t)$ are flat and $\mathcal{D}^\alpha \Phi(t)\ll C(\chi)$ for any $|\alpha|\geq 0$. In particular, $\Phi(t)$ and $\frac{h(t)|t|_F^{-1/2}}{|1-zt|_F}$ satisfy the flat condition (\ref{step. derivative bounds}).
	
	Furthermore, we write $\chi=\chi_m|\cdot|_F^{i\lambda}$. If $F=\mathbb{R}$, we write $\Phi(t)=2\lambda\ln|1-zt|-\lambda\ln|t|-\frac{2\pi Tty}{1-zt}+m\pi\mathrm{sgn}(t)$. If $F=\mathbb{C}$, we write $\Phi(t)=\frac{m}{2i}\left(2\ln(1-zt)-2\ln(\overline{1-zt})+\ln(\bar{t})-\ln(t)\right)+\lambda(2\ln|1-zt|^2-\ln|t|^2)-2\pi T(\frac{ty}{1-zt}+\overline{\frac{ty}{1-zt}})$. In both case, we have 
	$$\|\nabla\Phi(t)\|=\deg(F)\cdot 2\pi \left|\frac{T(yt+1-z^2t^2)}{(1-zt)^2t}\right|.$$
	Note that for any $t\in \mathrm{Supp}(h)$, $|tz|<1$. The equation $yt+1-z^2t^2=0$ has at most one solution satisfy that $|zt|<1$, which we denote by $t_0$. 
	
	For any $a'>a$, let $Y_2=(1+b^2(1+a)^2)(1+a')$ and $Y_1=\frac{1-b^2(1+a)^2}{1+a'}$. We note again $|zt|<b(1+a)$ for $t\in \mathrm{Supp}(h)$. If $|y|\geq Y_2$, we have 
	$$\|\nabla \Phi(t)\|\geq \deg(F)\cdot 2\pi T\frac{|y|-\frac{|1-z^2t^2|}{|t|}}{|1-zt|^2}\gg T|y|.$$
	If $|y|\leq Y_1$, we have 
	$$\|\nabla \Phi(t)\|\geq \deg(F)\cdot 2\pi T\frac{\frac{|1-z^2t^2|}{|t|}-|y|}{|1-zt|^2}\gg T.$$
	Apply Theorem \ref{lem. stationary phase lemma, no stationary point}, then we conclude (\ref{step. Whittaker function, y>Y_2}) and (\ref{step. Whittaker function, y<Y_1}).
	
	As for $Y_1\leq |y|\leq Y_2$, we apply the asymptotic formula in Theorem \ref{lem. stationary phase general, stationary point}, Corollary \ref{lem. stationary phase, C}. The determinant of Hessal matrix is given by (when $F=\mathbb{R}$, we move out the absolute value directly)
	$$\det(H_\Phi(t)/2\pi)=-\deg(F)\cdot\left| T\left(\frac{2z^2}{(1-zt)^2}-\frac{1}{t^2}+\frac{2yz}{(1-zt)^3}\right)\right|_F$$
	and in particular, 
	$$\det(H_\Phi(t_0)/2\pi i)=i^{-2+\deg(F)}\deg(F)\left|\frac{T(1+z^2t_0^2)}{t_0^2(1-zt_0)^2}\right|_F.$$
	Hence the leading amplitude term is given by 
	$$|Ty|_F^{\frac{1}{2}}\frac{h(t_0)|t_0|_F^{-\frac{1}{2}}}{|1-zt_0|_F}\deg(F)\left|\frac{t_0^2(1-zt_0)^2}{T(1+z^2t_0^2)}\right|_F^{\frac{1}{2}}=h(t_0)\deg(F)\cdot\left| \frac{1-z^2t_0^2}{1+z^2t_0^2}\right|_F^{\frac{1}{2}}.$$ 
	Since $d_\mathbb{C}$ is twice the usual Lebesgue measure, the term $|\deg F|_F$ is canceled again. The phase function is given by 
	$$e^{i\Phi(zt_0)}=\psi_F(T+Tzt_0)\chi\left(-\frac{(1-zt_0)y}{1+zt_0}\right).$$
	Combine it with $C_{F,\chi}=\chi(-1)\psi_F(-T)i^{-2+\deg{F}}$, we conclude (\ref{step. Whittaker function, Y_1<y<Y_2}) holds. 
	
\end{proof}

\subsection{Lower bound for local Rankin-Selberg integral}

Let $\pi_1,\pi_2$ be unitary representations of $G$, and let $\mathcal{W}(\pi_1,\psi_F)$ and $\mathcal{W}(\pi_2,\bar{\psi}_F)$ denote their respective Whittaker models.. Let $\pi_3=\mathcal{I}(\chi_3)$ be a representation of $G$ induced from $\chi_3=|\cdot|^s$. In this section, we show that our analytic newvectors can serve as test vectors for (\ref{Intro. lower bound}).  

For $W_1\in \mathcal{W}(\pi_1,\psi_F)$, $\bar{W}_2\in \mathcal{W}(\pi_2,\bar{\psi}_F)$, and $f_3\in \mathcal{I}(\chi_3)$, recall that the local Rankin-Selberg integral is given by
\begin{equation}\label{step. local Rankin-Selberg}
	\begin{aligned}
		\Psi_s(W_1,\bar{W}_2,f_3)&=\int_{N\backslash G}W_1(g)\bar{W}_2(g)f_3(g)dg\\
		&=\int_{x\in F}\int_{y\in F^\times}W_1(a(y)n'(x))\bar{W}_2(a(y)n'(x))|y|_F^{-\frac{1}{2}+s}f_3(n'(x))\ d_F^\times yd_Fx
	\end{aligned}
\end{equation}
We introduce a bilinear functional on the space $\mathcal{W}(\pi_1,\psi_F)\otimes\mathcal{W}(\pi_2,\bar{\psi}_F)$ as
\begin{equation}\label{def. ell linear form}
	\ell_s(W_1,\bar{W}_2,x)=\int_{y\in F^\times}W_1(a(y)n'(x))\bar{W}_2(a(y)n'(x))|y|_F^{-\frac{1}{2}+s}\ d_F^\times y.
\end{equation}
This defines a Borel-equivariant functional, as in \cite[(3.45)]{MV10}.

\begin{proposition}\label{prop. lower bound for ell_s}
	Let $F=\mathbb{R}$ or $\mathbb{C}$. For $j=1,2$, let $\chi_j$ be unitary characters of $F^\times$, and $\pi_j=\mathcal{I}(\chi_j)$. Then, for any $s\in \mathbb{C}$ with $\Re(s)>-\frac{1}{2}$, there exist constants $c,c',R>0$ and a family of analytic newvectors $f_{\chi_i}\in \mathcal{I}(\chi_i)$ such that 
	\begin{equation}\label{step. lower bound for ell_s}
		\mathrm{Re}(\ell_s(W_{f_{\chi_1}},\bar{W}_{f_{\chi_2}},x))\geq c'.
	\end{equation}
	whenever $C(\chi_1),C(\chi_2)\geq R$ and $0\leq |x|_F\leq cC(\pi_1\otimes\pi_2)^{-1/2}$. 
\end{proposition}

\begin{proof}
	
	Given $s \in \mathbb{C}$ with $\Re(s) > -\frac{1}{2}$, we fix constants $a,b,R$ to be chosen later. Let $h(t)$ be a fixed non-negative smooth function, compactly supported on $\frac{1}{1+a} \leq |t| \leq 1+a$, and attaining its maximum value $M = \sup |h(t)|$ on the smaller region $\frac{1}{1+0.9a} \leq |t| \leq 1+0.9a$. Let $f_{\chi}$ be a family of analytic newvectors in $\mathcal{I}(\chi)$ associated to $h(t)$ as in (\ref{def. analytic newvector}). For $j=1,2$, we denote $T_j=T(\chi_j)$ as in (\ref{def. spectral parameters}) and set $z_j=T_jx$.  We write $f_j = f_{\chi_j}$, $W_j = W_{f_{\chi_j}}$, and let $t_j=t_j(y,x)$ be the unique solution of $1-z_j^2t_j^2+t_jy=0$ with $|z_jt_j|<1$ and define the phase $\Phi_j(y,x)=\Phi_{\chi_j}(z_jt_j)$ by 
	$$e^{i\Phi_j(y,x)}=\psi_F(T_jz_jt_j)\chi\left(\frac{1-z_jt_j}{1+z_jt_j}\right)$$
	as in Theorem \ref{thm. Explicit formula for translated Whittaker function}, (\ref{step. Whittaker function, Y_1<y<Y_2})and (\ref{step. Whittaker function, Phase}). 
	
	Let $a' > a$ and $0 < c \le b$ be constants to be chosen later. Set $Y_2=(1+b^2(1+a)^2)(1+a')$ and $Y_1=\frac{1-b^2(1+a)^2}{1+a'}$. For any $0\leq |x|_F\leq cC(\pi_1\otimes\pi_2)^{-1/2}$, we have $|z_j|_F=|T_jx|_F\leq c\frac{C(\pi_j\otimes\pi_j)^{1/2}}{C(\pi_1\otimes\pi_2)^{1/2}}\leq b$, so that the asymptotic formula in Theorem \ref{thm. Explicit formula for translated Whittaker function} hold uniformly. We then split the integration over $F^\times$ into three regions: $0 < |y| < Y_1$, $|y| > Y_2$, and $Y_1 \leq |y| \leq Y_2$. The first and second integrals converge for $\Re(s) > -\frac{1}{2}$ and can be bounded by $C(\chi_1)^{-N} C(\chi_2)^{-N}$ using \ref{step. Whittaker function, y<Y_1} and \ref{step. Whittaker function, y>Y_2}. For the middle region $Y_1 \leq |y| \leq Y_2$, we insert \eqref{step. Whittaker function, Y_1<y<Y_2}. Then we  obtain:
	$$\ell_s(W_1,\bar{W}_2,x)=\int_{Y_1\leq |y|\leq Y_2}e^{i\Phi(y,x)}H(y,x)d_F^\times y+O\left(\frac{1}{C(\chi_1)^{1/\deg(F)}}\right)+O\left(\frac{1}{C(\chi_2)^{1/\deg(F)}}\right),$$
	where the amplitude function is 
	$$H(y,x)=h(t_1)h(t_2)\left|\frac{1-z_1^2t_1^2}{1+z_1^2t_1^2}\right|_F^{\frac{1}{2}}\left|\frac{1-z_2^2t_2^2}{1+z_2^2t_2^2}\right|_F^{\frac{1}{2}}|y|_F^{-\frac{1}{2}+\Re(s)},$$
	and the phase function is 
	$$e^{i\Phi(y,x)}=e^{i\Phi_1(y,x)-i\Phi_2(y,x)}|y|_F^{i\Im(s)}.$$

	Now we are going to decide $a,b,c,a',c'$ such that  for all $0\leq |x|_F\leq cC(\pi_1\otimes\pi_2)^{-1/2}$ and $Y_1\leq |y|\leq Y_2$, $|\Phi(y,x)|\leq \frac{\pi}{3}$ and $H(y,x)$ is large enough. 
	
	Consider the phase function, we have 
	$$\begin{aligned}
		&|\Phi_1(y,x)-\Phi_2(y,x)|\\
		\leq &2\pi \deg(F)\cdot \left|T_1z_1t_1+T_1\ln\left(\frac{1-z_1t_1}{1+z_1t_1}\right)-T_2z_2t_2-T_2\ln\left(\frac{1-z_2t_2}{1+z_2t_2}\right)\right|\\
		\leq& 2\pi \deg(F)\cdot\left(|T_1z_1t_1-T_2z_2t_2|+\sum_{j=1}^\infty\frac{2}{2j+1}\left|T_1z_1^{2j+1}t_1^{2j+1}-T_2z_2^{2j+1}t_2^{2j+1}\right|\right)
	\end{aligned}$$ 
	Since $1-z_j^2t_j^2+t_jy=0$, we have $|t_1-t_2|=|t_1t_2(z_1^2t_1-z_2^2t_2)|$. We also note that $|z_j|<b$, $|t_j|<1+a$. By applying $|z_1^2t_1-z_2^2t_2|\leq |z_1^2-z_2^2||t_1|+|z_2^2||t_1-t_2|$, we obtain 
	$$ |T_1z_1t_1-T_2z_2t_2|\leq \frac{(1+a)}{1-(b(1+a))^2}\left|x(T_1^2-T_2^2)\right|.$$
	Furthermore, we have
	$$\begin{aligned}
		\left|T_1z_1^{2j+1}t_1^{2j+1}-T_2z_2^{2j+1}t_2^{2j+1}\right|
		\leq&  \frac{1}{|x|}\left(|t_1^{j}||z_1^{2j+2}t_1^{j+1}-z_2^{2j+2}t_2^{j+1}|+|z_2^{2j+2}t_2^{j+1}||t_1^j-t_2^j|\right)\\
		\leq & \left((j+1)(b(1+a))^{2j}+j(b(1+a))^{2j+2}\right)|T_1z_1t_1-T_2z_2t_2|\\
		\leq &(2j+1)(b(1+a))^{2j}|T_1z_1t_1-T_2z_2t_2|.
	\end{aligned}
	$$
	Hence we conclude that for any $|x|_FC(\pi_1\otimes\pi_2)^{-1/2}\leq c$, 
	$$|\Phi_1(y,x)-\Phi_2(y,x)|\leq 2\pi\deg(F)\cdot\frac{(1+b^2(1+a)^2)(1+a)}{(1-b^2(1+a)^2)^2}\cdot  c.$$
	Take $c=\frac{(1-b^2(1+a)^2)}{12\deg(F)(1+b^2(1+a)^2)(1+a)}$, we get $|\Phi_1(y,x)-\Phi_2(y,x)|\leq \frac{\pi}{6}$. 
	
	Now we turn to $|y|_F^{i\Im(s)}=e^{i\deg(F)\Im(s)\ln|y|}$. Since $Y_1\leq |y|\leq Y_2$, we get 
	$$|\ln|y||\leq \max\{|\ln|Y_2||, |\ln|Y_1||\}\leq \left\{a'+b^2(1+a)^2(1+a'),\ \frac{a'+b^2(1+a)^2}{1-b^2(1+a)^2}\right\}.$$ 
	We choose $a'=2a$, $b$ such that $b(1+a)=a$, and $|a|<\min\{\frac{1}{3},\frac{\pi}{\deg(F)\cdot 18|\Im(s)|}\}$. Then we obtain  $|\deg(F)\Im(s)\ln|y||\leq \frac{\pi}{6}$ and furthermore, $|\Phi(y,x)|\leq \frac{\pi}{3}$
	for all $|x|_F\leq cC(\pi_1\otimes\pi_2)^{-1/2}$ and $Y_1\leq |y|\leq Y_2$.
	
	Finally, take $Y_1<Y_1'<Y_2'<Y_2$, such that $Y_2'-Y_1'\gg a$ and for any $Y_1'< |y|<Y_2'$, $\frac{1}{1+0.9a}\leq |t_j|\leq 1+0.9a$, and then $H(y,x)\geq \frac{1}{2}M^2$. In general, we always have $H(y,x)\geq 0$. Hence we get the integral is larger than some absolute constant $2c'$:  
	$$\mathrm{Re}\int_{Y_1\leq |y|\leq Y_2}e^{i\Phi(y,x)}H(y,x)d_F^\times y \geq \int_{Y_1'\leq |y|\leq Y_2'}\frac{1}{4}M^2d_F^\times y:=2c'>0.$$
	Finally, taking $R$ sufficiently large so that the error terms $O(\frac{1}{C(\chi_1)^{1/\deg(F)}})$, $O(\frac{1}{C(\chi_2)^{1/\deg(F)}})$ are bounded by $\frac{1}{2}c'$, we conclude (\ref{step. lower bound for ell_s}) as desired. 

\end{proof}

\begin{theorem}\label{Final proof. lower bound for local Rankin-Selberg integral}
	
	Let $F=\mathbb{R}$ or $\mathbb{C}$. Let $\pi_1$ and $\pi_2$ be two unitary principal series representations of $G$, and let $\pi_3 = \mathcal{I}(\chi_3)$ with $\chi_3 = |\cdot|^s$ fixed, where $|\Re(s)| < 1/2$. Then there exist unit vectors $W_j \in \mathcal{W}_j(\pi_j,\psi)$ and $f_3 \in \pi_3$ such that 
	\begin{equation}\label{thm. lower bound for local Rankin-Selbegr integral}
		\Psi_s(W_1,\bar{W}_2,f_3) \gg_{\pi_3} C(\pi_1 \otimes \pi_2)^{-\frac{1}{4}+\frac{1}{2}\Re(s)},
	\end{equation}
	and
	\begin{equation}\label{thm. lower bound for local trilinear form}
		I^T(W_1,\bar{W}_2,f_3) \gg_{\pi_3} C(\pi_1 \otimes \pi_2)^{-1/2}.
	\end{equation}
	
\end{theorem}

\begin{proof}
	
	When either $C(\pi_1)$ or $C(\pi_2)$ is bounded, the statement follows from \cite{MV10} or \cite{BJN23}, which cover the case where $\pi_2$ has bounded conductor, including the complementary series. 
	
	Now assume that $\pi_1$ and $\pi_2$ are tempered principal series and that $\min\{C(\pi_1),C(\pi_2)\} \to \infty$. Set $|Q|_F=C(\pi_1\otimes\pi_2)^{1/2}$. For a given $s$, let $f_{\chi_j}$ be the analytic newvector family as in Proposition \ref{prop. lower bound for ell_s}, and let $f_3^0$ be a unit, positive, smooth bump function supported on $|x|_F\leq c$. Define $f_3=\pi_3(a(Q))f_3^0$. Applying (\ref{step. local Rankin-Selberg}), we obtain
	$$\Psi_s(W_1,\bar{W}_2,f_3)=|Q|_F^{\frac{1}{2}+s}\int_{|x|\leq cQ^{-1}}\ell_s(W_1,\bar{W}_2,x)f_3^0(Qx)d_Fx\gg |Q|_F^{-\frac{1}{2}+\Re(s)}.$$
	Finally, (\ref{thm. lower bound for local trilinear form}) follows from (\ref{trilinear form}) together with $\gamma(\pi_1\otimes\pi_2\otimes\chi,\frac{1}{2})\asymp C(\pi_1\otimes\pi_2)^{-\Re(s)}$. 

\end{proof}

\subsection{Mellin component of analytic newvectors}

\subsubsection{Mellin components}

	Let $F$ be a local field. For unitary character $\mu$ of $F^\times$ and $\sigma\in \mathbb{R}$, we write $\mu_\sigma:=\mu|\cdot|_F^\sigma$. 
	
	\begin{definition}\label{def. Mellin component def}
		Let $V(y)$ be a fixed compactly supported function on $F^\times$. Suppose that $W$ is a smooth function on $F^\times$ and $A>0$ is a parameter. The \textbf{Mellin component} of $W$ is defined as
		\begin{equation}\label{def. Mellin component}
			M(W,A,\mu_\sigma):=\int_{y\in F^\times}W(a(y))V\left(\frac{y}{A}\right)\mu_\sigma(y)d^\times_F y
		\end{equation}
		Since $V$ has compact support, this integral is absolutely convergent for all $W,A,\mu,\sigma$. Furthermore, we denote by
		\begin{equation}
			L^2(W,A,\sigma):=\int_{y\in F^\times}\left|W(a(y))V\left(\frac{y}{A}\right)\right|^2|y|_F^{2\sigma}d_F^\times y
		\end{equation}
		which characterizes the $L^2$-mass of $W(a(y))$ on $y\asymp A$. 
	\end{definition}
	
	Let $\pi=\mathcal{I}(\chi)$ and $W_f\in \mathcal{W}(\pi,\psi)$. Recall that the intertwining map 
	$$f(wn(t))=\int_{F}W_f(a(y))|y|_F^{-\frac{1}{2}}\chi(y)\psi_F(ty)d_F y$$
	is absolutely convergent whenever $\Re(\chi)>-1$, as in \cite{Go70}. In fact, $W(a(y))|y|_F^{-\frac{1}{2}}\chi(y)\in L^1(F)$ when $\Re(\chi)>-1$. 

	\begin{lemma}
	 	Notation as above. When $\Re(\chi)>-1$, the following formula holds 
		\begin{equation}\label{def. Mellin component by Gauss sum}
			M(W_f,A,\mu_\sigma)=|A|_F^{\frac{1}{2}+\sigma}\int_{t\in F}f(wn(t))\mu^{-1}\chi(t)G\left(\mu\chi^{-1},At\right)dt
		\end{equation}
		where $G$ is the generalized Gauss sum defined in Lemma \ref{lem. Gauss sum. } corresponding to $V_1(y)=V(y)|y|_F^{\frac{1}{2}+\sigma}$.

	\end{lemma}

	\begin{proof}
	
	It is well known that, according to Fubini's theorem, for $g_1,g_2\in L^1(F)$, we have 
	$$\int_{y\in F}g_1(y)\check{g}_2(y)=\int_{y\in F}\int_{t\in F}g_1(y)g_2(t)\psi(ty)dtdy=\int_{t\in F}\check{g}_1(t)g_2(t)dt. $$
	In our case, let $g_1(y)=W_f(a(y))|y|_F^{-\frac{1}{2}}\chi(y)$ be one of the $L^1$-function with $\check{g}_1(t)=f(wn(t))$, and we define $\check{g}_2(y)=V\left(\frac{y}{A}\right)|y|_F^{-\frac{1}{2}}\mu_\sigma(y)\chi^{-1}(y)$. Since $\check{g}_2$ has compact support, its Fourier inversion $g_2(t)$ is a Schwartz function. We obtain that 
	\begin{equation}\label{lem. Mellin component}
		M(W_f,A,\mu_\sigma)=\int_{t\in F}f(wn(t))\left(\int_{y\in F^\times}V\left(\frac{y}{A}\right)|y|_F^{\frac{1}{2}+\sigma}\mu(y)\chi^{-1}(y)\psi_F(-ty)d^\times_F y\right) dt.
	\end{equation}
	(\ref{lem. Mellin component}) holds for any $W_f\in \mathcal{W}(\pi,\psi)$ with $\Re(\chi)>-1$. 
	
	Furthermore, by writing $V_1(y)=V(y)|y|_F^{\frac{1}{2}+\sigma}$ which is again a compactly support function, the inside integral is then given by the generalized Gauss sum defined in (\ref{def. generalized gauss sum})
	$$\int_{y\in F^\times}V\left(\frac{y}{A}\right)|y|_F^{\frac{1}{2}+\sigma}\mu(y)\chi^{-1}(y)\psi_F(-ty)d^\times_F y=\chi(t)\mu^{-1}(t)|A|^{\frac{1}{2}+\sigma}G(\mu\chi^{-1},At).$$
	Then we conclude. 
	
	\end{proof}

\subsubsection{Mellin component of analytic newvectors}

	Let $\sigma\in \mathbb{R}$ be fixed. Let $V(y)$ be a fixed compactly supported smooth bump function, usually coming from a dyadic partition, and set
	$$V_1(y)=V(y)|y|_F^{\frac{1}{2}+\sigma}.$$
	Let $G=G_{V_1}$ be the generalized Gauss sum associated to $V_1$ as in Lemma~\ref{lem. Gauss sum. }, with the parameters $A_1$, $B_1$ chosen so that (\ref{lem. Gauss sum. rapidly decay}) and (\ref{lem. Gauss sum. bounds for derivatives}) hold.

	Let $f_\chi\in \mathcal{I}(\chi)$ be a family of analytic newvectors associated with a fixed $h(t)$. Throughout this section we restrict to the range $|z|\leq 1$, and we assume that  
	$$V_2(t)=h(t)|t|_F^{\frac{1}{2}},\quad\text{with}\quad  \mathrm{Supp}(h)\subseteq \left\{\frac{1}{1+a}\leq |t|\leq 1+a\right\},$$
	and we fix a parameter $0<b<\frac{1}{1+a}$. 
	Let $W_{f_\chi}^{(z)}$ denote the corresponding Whittaker functions obtained via the Jacquet integral \eqref{def. Jacquet integral}, and $T=T(\chi)$ be the spectral parameter of $\chi$ as in (\ref{def. spectral parameters}). 
	
	For this fixed family of analytic newvectors $f_\chi$, and the fixed data $V$ and $\sigma$, our objective is to study the Mellin components
	$$M(W_{f_\chi}^{(z)},Y,\mu_\sigma)$$
	uniformly for all $\chi$, $\mu$ and all $|z|\leq 1$. All implicit constants will be uniform in $\chi$, $\mu$ and $Y$, once the parameters $h(t)$, $a$, and $b$ are fixed.  

	\begin{lemma}
		 Let $\mu$ be a unitary character of $F^\times$, and let $Y>0$ be a parameter. Then the Mellin component of $W_{f_\chi}^{(z)}(y)$ at scale $y\sim Y$ with respect to $\mu$ is given by
		\begin{equation}\label{step. Mellin component. main equation}
			M\left(W_{f_\chi}^{(z)},Y,\mu_\sigma\right)=C_{F,\chi}|Y|_F^{\frac{1}{2}+\sigma}|T|_F^{\frac{1}{2}}\int_{t\in F}\frac{\mu^{-1}(t)}{\chi(1+zt)}G(\mu\chi^{-1},YTt)V_2\left(\frac{t}{1+zt}\right)\frac{d_F t}{|t|_F}.
		\end{equation}

	\end{lemma}

	\begin{proof}
		Substituting the expression \eqref{def. Phi for n'(z)f} into
		\eqref{def. Mellin component by Gauss sum} yields the stated formula.
	\end{proof}

	\begin{theorem}\label{thm. Mellin component 1}
		
		Suppose that $|z|\leq b$. Then there exist $0<Y_1<1<Y_2$ and $0<C_1<C_2$, such that
		\begin{enumerate}[label=\textnormal{(\arabic*)}]
			\item If $Y\leq Y_1$ or $Y\geq Y_2$, then for 	any $\mu$, we have 
			$$M(W^{(z)},Y,\mu_\sigma)\ll_N |Y|_F^{\frac{1}{2}+\sigma}(1+Y)^{-N}C(\chi)^{-N}C(\mu)^{-N}.$$
			\item If $Y_1\leq Y\leq Y_2$, then for $\mu$ satisfies that  $C_1|zT(\chi)|\leq |T(\mu)|\leq C_2|zT(\chi)|$, we have
			$$M(W^{(z)},Y,\mu_\sigma)\ll |Y|_F^{\frac{1}{2}+\sigma}C(\mu)^{-1/2}.$$ 
			Otherwise, we have 
			$$M(W^{(z)},Y,\mu_\sigma)\ll_N |Y|_F^{\frac{1}{2}+\sigma}C(\mu)^{-N};$$
			
		\end{enumerate} 
	In particular, we have 
	$$L^2(W^{(z)},Y,\sigma)\ll \begin{cases}
		|Y|_F^{1+2\sigma}, & Y_1\leq Y\leq Y_2;\\
		|Y|_F^{1+2\sigma}(1+Y)^{-N}C(\chi)^{-N}, & Y\leq Y_1\text{ or }Y\geq Y_2.
	\end{cases}$$
	\end{theorem}

	\begin{theorem}\label{thm. Mellin component 2}
	
	Suppose that $b\leq |z|\leq 1$. Then there exist $1<Y_2$ and $0<C_1<C_2$, $0<D_1<D_2$, such that
	\begin{enumerate}[label=\textnormal{(\arabic*)}]
		\item If $Y\geq Y_2$, then for any $\mu$, we have 
		$$M(W^{(z)},Y,\mu_\sigma)\ll_N |Y|_F^{\frac{1}{2}+\sigma}(1+Y)^{-N}C(\chi)^{-N}C(\mu)^{-N}.$$
		\item If $0<Y\leq Y_2$, then we have 
		$$M(W^{(z)}_{f_\chi},Y,\mu_\sigma)\ll_N |Y|_F^{\frac{1}{2}+\sigma}\begin{cases}
			C(\chi)^{-N}, & \text{if }\left|\frac{T(\mu)}{zT(\chi)}\right|\leq C_1, \\
			C(\mu)^{-N}, &\text{if }\left|\frac{T(\mu)}{zT(\chi)}\right|\geq C_2,
		\end{cases}$$
		and for those $C_1\leq \left|\frac{T(\mu)}{zT(\chi)}\right|\leq C_2$, we have 
		$$M(W^{(z)}_{f_\chi},Y,\mu_\sigma)\ll_N |Y|_F^{\frac{1}{2}+\sigma}\begin{cases}
			(1+|YT|)^{-1/2}, & \text{if }D_1|YT|\leq \frac{|T(\mu)^2-T(\chi)^2|}{|T(\chi)|}\leq D_2|YT|,
			\\
			(1+|YT|)^{-N}. & \text{otherwise}. 
		\end{cases}$$

	\end{enumerate} 
	In particular, we have 
	$$L^2(W_{f_\chi}^{(z)},Y,\sigma)\ll \begin{cases}
		|Y|_F^{1+2\sigma}, &  Y\leq Y_2;\\
		|Y|_F^{1+2\sigma}(1+Y)^{-N}C(\chi)^{-N}, & Y\geq Y_2.
	\end{cases}$$
	\end{theorem}

	\begin{remark}
		
		This result is more clean in the p-adic case, see \cite[Proposition 2.12]{Hu18}, \cite{Ass19}, \cite[Proposition 3.20]{Hu20}, \cite[Corollary 4.22]{HMN22}, where the statement "rapidly decay w.r.t. $C(\chi)$ or $C(\mu)$" is exactly vanishing. These theorems give the support of Mellin components and size for single Mellin component. 
		
	\end{remark}

	Before we give the proof, we discuss a direct corollary. 

	\begin{proposition}[$L^2$-concentration of translated Whittaker functions]\label{prop. L^2 concentration of translated Whittaker functions}
		Let $f_\chi\in \mathcal{I}(\chi)$ be a family of analytic newvectors and $W_{f_\chi}$ be its Whittaker function.  Then we have
		\begin{equation}
			\int_{y\in F^\times}|W^{(z)}_{f_\chi}(a(y))|^2|y|_F^{-\frac{1}{2}+\sigma}d^\times_F y\ll 1 
		\end{equation}
		for any $\sigma>-\frac{1}{2}$ and $0\leq|z|\leq 1$.
	\end{proposition}
	
	\begin{proof}
		
	We introduce a smooth dyadic partition
		$$1=\sum_{k\in \mathbb{Z}}V(2^ky)^2,$$
		where $V(y)$ is a positive compactly supported smooth function. Then we have 
		$$\begin{aligned}
			\int_{y\in F^\times}|W^{(z)}_{f_\chi}(a(y))|^2|y|_F^{-\frac{1}{2}+\sigma}d^\times_F y&=\sum_{k\in \mathbb{Z}}\int_{y\in F^\times}\left|W^{(z)}_{f_\chi}(a(y))|y|_F^{-\frac{1}{4}+\frac{1}{2}\sigma}V(2^ky)\right|^2d^\times_F y
		\end{aligned}$$
		Applying the estimates, the $L^2$-mass
		$$L^2\left(W^{(z)},2^k,-\tfrac{1}{4}+\tfrac{1}{2}\sigma\right)\ll \begin{cases}|Y|_F^{1+2\sigma}, & 2^k\leq Y_2;\\
			 |Y|_F^{1+2\sigma}(1+Y)^{-N}C(\chi)^{-N}, & 2^k\geq Y_2.
		\end{cases}$$
		Since $\sigma>-\frac{1}{2}$, we conclude. 
		
	\end{proof}

	\begin{remark}
		This proposition is the archimedean analogy of \cite[Corollary 4.23]{HMN22}. In the p-adic case, when the parameter $|z|< q^{-\frac{c(\pi)}{2}}$, the translated Whittaker function $W_{f_\chi}^{(z)}$ is still supported on $\mathcal{O}^\times$, i.e. concentrating around $1$. When $|z|=q^{-\frac{c(\pi)}{2}}$, the situation becomes complicated. The above lemma shows that  its $L^2$-mass would not move to $y=0$ extremely. 
	\end{remark}

	\vspace{11pt}

\subsubsection{Proof of theorem \ref{thm. Mellin component 1} and \ref{thm. Mellin component 2}}
	
	To streamline notation, throughout the proof we fix the data $h(t)$, $V(y)$, and $\sigma$, and write 
	\begin{equation}\label{step. Mellin integral}
		I=I(z,Y,\chi,\mu)=\int_{t\in F}\frac{\mu^{-1}(t)}{\chi(1+zt)}G(\mu\chi^{-1},YTt)V_2\left(\frac{t}{1+zt}\right)\frac{d_F t}{|t|_F}.
	\end{equation}
	For convenience, we introduce the notations
	$$w(t)=G(\mu\chi^{-1},YTt)V_2\left(\frac{t}{1+zt}\right),\quad e^{i\Phi(t)}=\frac{\mu^{-1}(t)}{\chi(1+zt)}.$$
	We shall study this integral uniformly in the varying parameters $(z,Y,\chi,\mu)$, with $|z|\leq 1$. 
	
	The support of $w(t)$ is controlled by $V_2(\frac{t}{1+zt})$. Let $S_0=S_0(z)$ denote its support, then
	\begin{equation}\label{step. support of V2}
		t\in S_0\ \Longleftrightarrow\  \frac{1}{1+a}\leq \left|\frac{t}{1+zt}\right|\leq 1+a.
	\end{equation}
	Introduce $A_2=\frac{1}{2+a}$, $B_2=\frac{1+a}{1-(1+a)b}$. Then (\ref{step. support of V2}) implies that for $t\in S_0$, 
	\begin{equation}
		\begin{cases}
			A_2\leq |t|\leq B_2, & \textnormal{if }|z|\leq b;\\
			A_2\leq |t|, &  \textnormal{if }|z|\leq 1.
		\end{cases}
	\end{equation}
	
	\vspace{11pt}

	\begin{lemma}\label{lem. Mellin w function}
		
		We define $S_1=S_1(Y,\chi,\mu)$ to be the range
		$$t\in S_1\ \Longleftrightarrow\ \left|\frac{T(\mu\chi^{-1})}{YTB_1}\right|\leq |t|\leq \left|\frac{T(\mu\chi^{-1})}{YTA_1}\right|,$$
		and for given $N$, we denote by
		$$X(t)=X(Y,\chi,\mu,t)=\begin{cases}
			C(\mu\chi^{-1})^{-1/2}, & t\in S_1;\\
			C(\mu\chi^{-1})^{-N}(1+|YTt|)^{-N}, & t\notin S_1.
		\end{cases}$$
		Then for any $|\alpha|\geq 0$ and $t\in S_0\cap S_1$, we have 
		$$\mathcal{D}^\alpha w(t)\ll_{\alpha, N} \frac{X(t)}{|t|^{|\alpha|}}.$$
		Moreover, we have 
		$$\int_{t\in S_0}X(t)d_F^\times t\ll_N \begin{cases}
			C(\mu\chi^{-1})^{-1/2}, & S_0\cap S_1\neq \emptyset;\\
			C(\mu\chi^{-1})^{-N}(1+|YT|)^{-N}, & S_0\cap S_1=\emptyset.
		\end{cases}$$
		
	\end{lemma}

	\begin{proof}
		This is a direct corrolary of Lemma \ref{lem. Gauss sum. }. 
	\end{proof}

	\begin{lemma}\label{lem. Mellin Phi function}
		
		The derivative of $\Phi$ and solution of $\nabla\Phi=0$ (if exists) are given by
		\begin{equation}\label{step. Mellin Phi function derivative}
			\|\nabla\Phi(t)\|=2\pi \deg(F)\cdot\left|\frac{T(\chi)}{t}\left(\frac{T(\mu)}{T(\chi)}+\frac{zt}{1+zt}\right)\right|,\quad t_0=-\frac{T(\mu)}{zT(\chi\mu)}.
		\end{equation}
		If $t_0$ exists, then for any $t\in S_0$ and $C_1<\frac{1}{1+a}$ and $C_2>1+a$, we have 
		\begin{equation}\label{step. Mellin Phi function derivative lower bound 1}
			\|\nabla\Phi(t)\|\gg\begin{cases}
				\left|\frac{zT(\chi)}{t}\right|, & \textnormal{ if }\left|\frac{T(\mu)}{zT(\chi)}\right|\leq C_1;\\
				\left|\frac{T(\mu)}{t^2t_0}(t-t_0)\right|, & \textnormal{ if }C_1\leq \left|\frac{T(\mu)}{zT(\chi)}\right|\leq C_2;\\
				\left|\frac{T(\mu)}{t}\right|, & \textnormal{ if }\left|\frac{T(\mu)}{zT(\chi)}\right|\geq C_2.
			\end{cases}
		\end{equation}
		For any $t\in S_0$ and $j\geq 2$, we have 
		\begin{equation}\label{step. Mellin Phi function higher derivative}
			\left|\partial_{t}^{j}\Phi(t)\right|\ll_j\frac{1}{|t|^j}\max\left\{\frac{|T(\mu)|}{|t_0|},\frac{|T(\mu)|}{|t|}\right\},\quad 	\left|\partial_{t}^{j}\Phi(t)\right|\ll_j \frac{1}{|t|^j}\max\{|zT(\chi)|,|T(\mu)|\}.
		\end{equation}
		Moreover, if $t_0$ exists, then 
		\begin{equation}\label{step. Mellin Phi function second derivative}
			\left|\partial_t^2\Phi(t_0)\right|=2\pi\deg(F)\left|\frac{T(\mu)}{t_0^3}\left(\frac{t_0}{1+zt_0}\right)\right|.
		\end{equation}

	\end{lemma}

	\begin{proof}
		
		Write $\chi(t)=\left(\frac{t}{|t|}\right)^m|t|_F^{i\lambda}$, and $\mu(t)=\left(\frac{t}{|t|}\right)^{m'}|t|_F^{i\eta}$, then after picking a branch of $\ln$, we obtain 
		$$\Phi(t)=-\frac{m'}{4\pi i}\ln\frac{t}{\bar{t}}-\frac{m}{4\pi i}\ln\frac{1+zt}{1+\bar{z}\bar{t}}-\frac{\eta}{2\pi}\ln|t|_{F}-\frac{\lambda}{2\pi}\ln|1+zt|_F.$$
		(when $F=\mathbb{R}$, we set $m,m'=0$). We note that $\Phi(t)$ is real-valued and is harmonic when $F=\mathbb{C}$. By taking  derivatives, we get $\|\nabla\Phi\|$ is given by 
		$$\begin{aligned}
		\|\nabla\Phi(t)\|&=2\pi\deg(F)\cdot \left|\frac{T(\mu)}{t}+\frac{zT(\chi)}{1+zt}\right|=2\pi\deg(F)\cdot \left|\frac{zT(\chi)}{t}\left(\frac{T(\mu)}{zT(\chi)}+\frac{t}{1+zt}\right)\right|\\
		&=2\pi\deg(F)\cdot \left|\frac{T(\mu)}{t}\left(1+\frac{zT(\chi)}{T(\mu)}\frac{t}{1+zt}\right)\right|=2\pi\deg(F)\cdot \left|\frac{T(\mu)}{t^2t_0}\frac{t}{1+zt}(t-t_0)\right|
		\end{aligned}$$
		For any $t\in S_0$, since $\frac{t}{1+zt}$ is controled as in (\ref{step. support of V2}), we see (\ref{step. Mellin Phi function derivative}), (\ref{step. Mellin Phi function derivative lower bound 1}) hold. For higher derivatives, we note that 
		$$\begin{aligned}
			\partial_t^j\Phi(t)&=(-1)^{j-1}\frac{(j-1)!}{t^j}\left(T(\mu)+T(\chi)\left(\frac{zt}{1+zt}\right)^j\right)\ll \frac{1}{|t|^j}\max\{|T(\mu)|,|zT(\chi)|\},
		\end{aligned}
		$$
		and furthermore, 
		$$\partial_t^j\Phi(t)\ll \frac{1}{|t|^j}\left(T(\mu)\left(1-\left(\frac{zt}{1+zt}\right)^j\right)+T(\chi\mu)\left(\frac{zt}{1+zt}\right)^{j}\right)\ll \frac{1}{|t|^j}\max\left\{\frac{|T(\mu)|}{|t|},\frac{|T(\mu)|}{|t_0|}\right\}.$$
		Hence (\ref{step. Mellin Phi function second derivative}) and (\ref{step. Mellin Phi function higher derivative}) hold. 

	\end{proof}
 
 	Now we are going to prove Theorem \ref{thm. Mellin component 1}, for the case $|z|\leq b$. In this case, we have $S_0$ is compact, and $A_2\leq |t|\leq B_2$. 
 
	\begin{lemma}\label{lem. Mellin Phi function 2}
		
		If $|z|\leq b$, then for any  $C_1'<1-b(1+a)$, $C_2'>1+b(1+a)$, we have 
		$$
		\|\nabla \Phi(t)\|\gg \begin{cases}
			|T(\mu)|, & \text{when }C_2'\leq \left|\frac{T(\mu\chi^{-1})}{T(\chi)}\right|;\\
			|T(\mu)||t-t_0|, & \text{when }C_1'\leq \left|\frac{T(\mu\chi^{-1})}{T(\chi)}\right|\leq C_2';\\
			|T(\chi)|, & \text{when }C_1'\geq \left|\frac{T(\mu\chi^{-1})}{T(\chi)}\right|.
		\end{cases}$$
		Furthermore, if $C_1'\geq \left|\frac{T(\mu\chi^{-1})}{T(\chi)}\right|$, we have $|T(\mu)|\asymp |T(\chi)|$; if $C_1'\leq \left|\frac{T(\mu\chi^{-1})}{T(\chi)}\right|\leq C_2'$, we have $|T(\mu)|\ll |T(\chi)|$; if $C_1'\leq \left|\frac{T(\mu\chi^{-1})}{T(\chi)}\right|$, we have $T(\mu)\gg T(\chi)$. 

	\end{lemma}
	
	\begin{proof}
			
		We note that $|t|\asymp 1$ and 
		$$\begin{aligned}
			\|\nabla\Phi(t)\|&=2\pi \deg(F)\cdot \left|\frac{T(\chi)}{t}\left(\frac{T(\mu\chi^{-1})}{T(\chi)}+1+\frac{zt}{1+zt}\right)\right|\\
			&=2\pi\deg(F)\cdot \left|\frac{T(\mu)}{t}\left(1+\frac{1}{1+\frac{T(\mu\chi^{-1})}{T(\chi)}}\frac{zt}{1+zt}\right)\right|\\
			&=2\pi\deg(F)\cdot \left|\frac{T(\mu)(t-t_0)}{t^2t_0}\left(\frac{t}{1+zt}\right)\right|.
		\end{aligned}$$
		Hence this lemma holds if $C_1'<1-b(1+a)$ and $C_2'>1+b(1+a)$. 
	\end{proof}
	
	\begin{proof}[Proof of Theorem \ref{thm. Mellin component 1}]
			 
		We set $Y_1=\frac{C_1'}{B_1B_2} $, $Y_2=\frac{C_2'}{A_1A_2}$. The idea is that, for $Y\leq Y_1$ or $Y\geq Y_2$ and for $t\in S_0$, we want to show that either $\nabla\Phi(t)\neq 0$ or $w(t)$ is small. 

	
		\vspace{3pt}
		(1.1) Suppose that $Y\leq Y_1$. 
		\vspace{3pt}
		
		\begin{itemize}
			
			\item If  $0\leq \left|\frac{T(\mu\chi^{-1})}{T(\chi)}\right|\leq C_1'$, then for any $t\in S_0$, $\|\nabla\Phi(t)\|\gg |T(\chi)|$, $w(t)\ll 1$. Since $C_1'<1$, we have $C(\mu)\asymp C(\chi)$. Apply Theorem \ref{lem. stationary phase lemma, no stationary point} and obtain 
			$$I\ll_N C(\chi)^{-2N}\ll (1+Y)^{-N}C(\chi)^{-N}C(\mu)^{-N}$$
			\item If $C_1'\leq \left|\frac{T(\mu\chi^{-1})}{T(\chi)}\right|$, then for any $t\in S_0$, we have $\left|\frac{T(\mu\chi^{-1})}{YT(\chi)t}\right|\geq B_1$, hence $w(t)\ll C(\mu\chi^{-1})^{-N}(1+|YTt|)^{-N}$. Furthermore, we have $C(\mu\chi^{-1})\gg C(\chi)$, $C(\mu\chi^{-1})\gg C(\mu)$. Then by trivial bound  we obtain 
			$$I\ll_N C(\mu\chi^{-1})^{-2N}(1+|YT|)^{-2N}\ll_N(1+Y)^{-N}C(\chi)^{-N}C(\mu)^{-N}$$

		\end{itemize}
		
		\vspace{3pt}
		(1.2) Suppose that $Y\geq Y_2$. 
		\vspace{3pt}
		
		\begin{itemize}
			\item If $0\leq \left|\frac{T(\mu\chi^{-1})}{T(\chi)}\right|\leq YA_1A_2$, then for any $t\in S_0$, we have  $w(t)\ll C(\mu\chi^{-1})^{-N}(1+|YTt|)^{-N}$. Furthermore, we have $(1+|YT|)\gg C(\mu)$, $(1+|YT|)\gg (1+Y)(1+|T|)$. Then by trivial bound we obtain
			$$I\ll_N C(\mu\chi^{-1})^{-2N}(1+|YT|)^{-2N}\ll_N(1+Y)^{-N}C(\chi)^{-N}C(\mu)^{-N}.$$
			\item If $YA_1A_2\leq \left|\frac{T(\mu\chi^{-1})}{T(\chi)}\right|\leq YB_1B_2$, then for any $t\in S_0$, $\|\nabla\Phi(t)\|\gg |T(\mu)|$, $w(t)\ll 1$. We have $|T(\mu\chi^{-1})|\asymp |YT(\chi)|$, hence $C(\mu)\gg (1+Y)(1+|T(\chi)|)$. Apply Theorem \ref{lem. stationary phase lemma, no stationary point} and obtain 
			$$I\ll_N C(\mu)^{-2N}\ll (1+Y)^{-N}C(\chi)^{-N}C(\mu)^{-N}.$$
			\item If $YB_1B_2\leq \left|\frac{T(\mu\chi^{-1})}{T(\chi)}\right|$, then for any $t\in S_0$, $\|\nabla\Phi(t)\|\gg |T(\mu)|$ and $w(t)\ll C(\mu\chi^{-1})^{-N}(1+|YTt|)^{-N}$. Apply Theorem \ref{lem. stationary phase lemma, no stationary point} and obtain 
			$$I\ll_N C(\mu\chi^{-1})^{-N}(1+|YT|)^{-N}C(\mu)^{-N}\ll (1+Y)^{-N}C(\chi)^{-N}C(\mu)^{-N}.$$
		\end{itemize}
		
		\vspace{3pt}
		(2) Suppose that $Y_1\leq Y\leq Y_2$. 
		\vspace{3pt}
		
		\begin{itemize}
			
			\item If $|T(\mu)|\leq C_1|zT(\chi)|$, then $\|\nabla\Phi(t)\|\gg |zT(\chi)|$ and $w(t)\ll C(\mu\chi^{-1})^{-1/2}$. Apply Theorem \ref{lem. stationary phase lemma, no stationary point} and then we obtain 
			$$I\ll_N C(\mu\chi^{-1})^{-1/2}(1+|zT(\chi)|)^{-N}\ll C(\mu\chi^{-1})^{-1/2}C(\mu)^{-N}. $$
			\item If $|T(\mu)|\geq C_2|zT(\chi)|$, then $\|\nabla\Phi(t)\|\gg |T(\mu)|$ and $w(t)\ll C(\mu\chi^{-1})^{-1/2}$. 
			Apply Theorem \ref{lem. stationary phase lemma, no stationary point} and then we obtain 
			$$I\ll_N C(\mu\chi^{-1})^{-1/2}C(\mu)^{-N}. $$
			\item If $C_1|zT(\chi)|\leq |T(\mu)|\leq C_2|zT(\chi)|$, then $\|\nabla\Phi(t)\|\gg |T(\mu)||t-t_0|$ and $w(t)\ll C(\mu\chi^{-1})^{-1/2}$. Since we assume that $zC_2<1$, we have $C(\mu\chi^{-1})\asymp C(\chi)$. Apply Theorem \ref{lem. stationary phase general, stationary point}, and then we obtain
			$$I\ll C(\mu\chi^{-1})^{-1/2}C(\mu)^{-1/2}.\ll C(\chi)^{-1/2}C(\mu)^{-1/2}.$$
			
		\end{itemize}
	We conclude by using 
	$$|M(W_{f_\chi}^{(z)},Y,\mu_\sigma)|=|Y|_F^{\frac{1}{2}+\sigma}|T|_F^{\frac{1}{2}}\cdot|I(z,Y,\chi,\mu)|.$$	
	The last argument follows from
	$$\begin{aligned}
		L^2(W_{f_\chi}^{(z)},Y,\sigma)&=\int_{y\in F^\times}\left|W_{f_\chi}^{(z)}(y)V\left(\frac{y}{Y}\right)\right|^2|y|_F^{2\sigma}d^\times_Fy=\int_{\mu\in \widehat{F^\times}}|M(W^{(z)},Y,\mu_\sigma)|^2d\mu,
	\end{aligned}
	$$
	and the fact that $\int_{\mu}C(\mu)^{-N}d\mu\ll 1$. 
	
	\end{proof}

	\vspace{11pt}
	Now we turn to the case $b\leq |z|\leq 1$. In this case, $S_0$ is no longer compact. Recall $A_2=\frac{1}{2+a}$ and  any $t\in S_0$ satisfies that $A_2\leq |t|$. We will apply \S 2.3.4 to analyze the behavior of $I$. 

	\begin{lemma}\label{lem. Mellin Phi function 3}
		If $b\leq |z|\leq 1$, then for any $C_2'>2+a$, we have
		$$
		\|\nabla \Phi(t)\|\gg \begin{cases}
			\left|\frac{T(\mu)}{t}\right|, & \text{when }\left|\frac{T(\mu\chi^{-1})}{T(\chi)}\right|\geq C_2'\\
			\left|\frac{zT(\mu\chi)}{t^2}\right||t-t_0|, & \text{when }\left|\frac{T(\mu\chi^{-1})}{T(\chi)}\right|\leq C_2'
		\end{cases}$$
		Furthermore, if $\left|\frac{T(\mu\chi^{-1})}{T(\chi)}\right|\geq C_2'$, we have $T(\mu\chi^{-1})\asymp T(\mu)\gg T(\chi)$. If $\left|\frac{T(\mu\chi^{-1})}{T(\chi)}\right|\leq C_2'$, we have $T(\mu\chi^{-1})\ll T(\chi)$, $T(\mu)\ll T(\chi)$. 
		
	\end{lemma}
	
	\begin{proof}
	The result follows by the same reasoning as in Lemma \ref{lem. Mellin Phi function 2}.
	\end{proof}

	\begin{proof}[Proof of Theorem \ref{thm. Mellin component 2}]
	Again, we set $Y_2=\frac{C_2'}{A_1A_2}$. We note that $|t|\geq A_2$. 
	
	\vspace{3pt}
	(1) Suppose that $Y\geq Y_2$. In this case, we want to show the major range of $w$ would not meet the stationary region of $\Phi$. 
	\vspace{3pt}
	
	\begin{itemize}
		\item When  $\left|\frac{T(\mu\chi^{-1})}{T(\chi)}\right|\leq C_2'$, for any $|t|\geq A_2$, we have $S_0\cap S_1=\emptyset$. Hence by Lemma \ref{lem. Mellin w function},  $w(t)\ll C(\mu\chi^{-1})^{-N}(1+|YTt|)^{-N}$ for any $t\in S_0$. Then we can apply the trivial bound, so that
		$$I\ll_N C(\mu\chi^{-1})^{-2N}(1+|YT|)^{-2N}\ll_N (1+|Y|)^{-N}C(\chi)^{-N}C(\mu)^{-N}.$$
		\item When $C_2'\leq \left|\frac{T(\mu\chi^{-1})}{T(\chi)}\right|$, and $Y\geq \frac{|T(\mu\chi^{-1})|}{|TA_1A_2|}$, applying Lemma \ref{lem. Mellin Phi function 3}, and Lemma \ref{lem. Mellin w function}, we have $w(t)\ll C(\mu\chi^{-1})^{-N}(1+|YTt|)^{-N}$ and $\|\nabla\Phi(t)\|\gg \frac{|T(\mu)|}{|t|}$. Applying Proposition \ref{prop. non-compact support stationary phase, rapidly decay}, and then we obtain
		$$I\ll_N C(\mu)^{-N}(1+|YT|)^{-N}\ll_N (1+|Y|)^{-N}C(\chi)^{-N}C(\mu)^{-N}.$$ 
		\item When $C_2'\leq \left|\frac{T(\mu\chi^{-1})}{T(\chi)}\right|$, and $Y_2\leq Y\leq \frac{|T(\mu\chi^{-1})|}{|TA_1A_2|}$, Lemma \ref{lem. Mellin Phi function 3} gives $\|\nabla\Phi(t)\|\gg \frac{|T(\mu)|}{|t|}$. By our assumption, $|T(\mu)|\gg (1+Y)(1+|T|)$.
		Applying Proposition \ref{prop. non-compact support stationary phase, rapidly decay}, and then again we obtain
		$$I\ll_N C(\mu)^{-2N}\ll_N (1+|Y|)^{-N}C(\chi)^{-N}C(\mu)^{-N}.$$ 
	
	\end{itemize}

	\vspace{3pt}
	(2.1) Suppose that $Y\leq Y_2$, and $\left|\frac{T(\mu)}{zT(\chi)}\right|\leq C_1$ or $\left|\frac{T(\mu)}{zT(\chi)}\right|\geq C_2$. By Lemma \ref{lem. Mellin Phi function}, we have 
	$$\|\nabla\Phi(t)\|\gg \frac{\max\{|T(\chi)|,|T(\mu)|\}}{|t|},\quad \partial_t^j\Phi(t)\ll \frac{\max\{|T(\chi)|,|T(\mu)|\}}{|t|^j}.$$
	\qquad\qquad We note that $w(t)$ satisfies Lemma \ref{lem. Mellin w function}. Apply Proposition \ref{prop. non-compact support stationary phase, rapidly decay}, and then we get 
	$$I\ll_N\max\{C(\chi),C(\mu)\}^{-2N}\ll_N C(\chi)^{-N}C(\mu)^{-N}.$$
	\vspace{3pt}

	(2.2) Suppose that $Y\leq Y_2$, and $C_1\leq \left|\frac{T(\mu)}{zT(\chi)}\right|\leq C_2$. In this case, we have 
	$$\|\nabla\Phi(t)\|\gg \frac{|T(\mu)|}{|t^2t_0|}|t-t_0|,\quad \frac{|T(\mu)|}{|t_0|}\asymp|T(\mu\chi)|.$$
	Hence we can apply Proposition \ref{prop. non-compact support stationary phase, stationary point}. Let $D_1=\frac{1}{2}A_1C_1$, $D_2=2B_1C_2$. We will see that these constants describe whether the peak of $w(t)$ is nearby $t_0$. 
 
	\begin{itemize}
		
		\item If 
		$$D_1\cdot |YT(\chi)|\leq \frac{\left|T(\mu\chi^{-1})T(\mu\chi)\right|}{|T(\chi)|}\leq D_2\cdot |YT(\chi)|,$$
		then for $t\sim t_0$, we have $w(t)\ll C(\mu\chi)^{-1/2}$. Apply Proposition \ref{prop. non-compact support stationary phase, stationary point}, we obtain
		$$\begin{aligned}
			I&\ll C(\mu\chi^{-1})^{-1/2}C(\chi\mu)^{-1/2}\ll C(\chi)^{-1/2}(1+|YT(\chi)|)^{-1/2}.
		\end{aligned}
		$$
		
		\item If 
		$$\left|\frac{T(\mu\chi^{-1})T(\mu\chi)}{T(\chi)}\right|\geq D_2\cdot|YT(\chi)|,$$
		then for any $t\sim t_0$, we have $w(t)\ll C(\mu\chi^{-1})^{-N}(1+|YTt_0|)^{-N}$. Furthermore, since $\left|\frac{T(\mu\chi)}{T(\chi)}\right|,\left|\frac{T(\mu\chi^{-1})}{T(\chi)}\right|\leq 1+|z|C_2$, by our assumption we have $|T(\mu\chi)|, |T(\mu\chi^{-1})|\gg 1+|YT(\chi)|$. Apply Proposition \ref{prop. non-compact support stationary phase, stationary point}, we obtain
		$$\begin{aligned}
			I&\ll_N C(\mu\chi^{-1})^{-1/2}C(\chi\mu)^{-N}+C(\mu\chi)^{-1/2}C(\mu\chi^{-1})^{-N}\\
			&\ll_N C(\chi)^{-1/2}(1+|YT(\chi)|)^{-N}.
		\end{aligned}
		$$
		
		\item If 
		$$\left|\frac{T(\mu\chi^{-1})T(\mu\chi)}{T(\chi)}\right|\leq D_1|YT(\chi)|\quad \text{and }|T(\chi\mu)|\geq |T(\chi)|,$$
		then for any $t\sim t_0$, we have $w(t)\ll C(\mu\chi^{-1})^{-N}(1+|YTt_0|)^{-N}$.Apply Proposition \ref{prop. non-compact support stationary phase, stationary point}, we obtain
		$$\begin{aligned}
			I&\ll_N C(\mu\chi^{-1})^{-1/2}C(\chi\mu)^{-N}+C(\mu\chi)^{-1/2}C(\mu\chi^{-1})^{-N}(1+|YT(\chi)|)^{-N}\\
			&\ll_N C(\chi)^{-1/2}(1+|YT(\chi)|)^{-N}.
		\end{aligned}
		$$
		
		\item If 
		$$\left|\frac{T(\mu\chi^{-1})T(\mu\chi)}{T(\chi)}\right|\leq D_1|YT(\chi)|\quad \text{and }|T(\chi\mu^{-1})|\geq 1,$$
		then $\left|\frac{T(\mu\chi^{-1})}{AT(\chi)A_1}\right|\leq \frac{1}{2}A_1C_1\left|\frac{T(\chi)}{T(\mu\chi)}\right|\leq\frac{1}{2}C_1$. We introduce a trunctation function $\rho(t)$, supported on $[0,C_1]$ and equals to $1$ at $[0,\frac{1}{2}C_1]$. We divide $I$ into 
		$$I=\int_{|t|\leq C_1}\rho(t)w(t)e^{i\Phi(t)}d_F^\times t+\int_{|t|\geq \frac{1}{2}C_1}(1-\rho(t))w(t)e^{i\Phi(t)}d_F^\times t$$
		For the second integral, we have $(1-\rho(t))w(t)\ll C(\mu\chi^{-1})^{-N}(1+|YTt|)^{-N}$ and hence it is bounded by Proposition \ref{prop. non-compact support stationary phase, stationary point}
		$$I_2\ll C(\mu\chi^{-1})^{-N}C(\mu\chi)^{-1/2}\ll C(\chi)^{-N}$$
		For the first integral, when $|t|\leq C_1$, we have $|t|\leq \frac{1}{2}|t_0|$, and in this range, $\|\nabla\Phi(t)\|\gg T(\mu)$. By Theorem \ref{lem. stationary phase lemma, no stationary point}, we obtain 
		$$I_1\ll_N C(\mu\chi^{-1})^{-1/2}C(\mu)^{-N}\ll_N C(\chi)^{-N}$$
		Again we conclude 
		$$I\ll_N(1+|YT(\chi)|)^{-N}.$$

	\end{itemize}
	
	For the $L^2$-mass argument, if $Y\geq Y_2$, we obtain by $\int_{\mu}C(\mu)^{-N}\ll 1$. If $Y\leq Y_2$, we  have
	$$L^2(W^{(z)},Y,\sigma)= \int_{\left|\frac{T(\mu\chi)T(\mu\chi^{-1})}{T(\chi)}\right|\sim |YT(\chi)|}|M(W^{(z)},Y,\mu_\sigma)|^2d\mu+|Y|_F^{1+2\sigma}(1+|YT(\chi)|)^{-N}\ll |Y|_F^{1+2\sigma}$$

\end{proof}

\subsection{Upper bound for local trilinear form}

\subsubsection{Bounds of matrix coefficients}

Let $F$ be a local field, $G=\mathrm{PGL}_2(F)$, $K$ be its maximal compact subgroup. Let $\chi=|\cdot|_F^s$, and $\pi_s=\mathcal{I}(\chi)$ be the principal series representation unitarily induced from $\chi$.  We denote $f_s\in \pi_s$ as the unique $K$-invariant vector normalized by $f_s(1)=1$. The spherical functions are defined as
\begin{equation}\label{def. Xi function}
	\Xi_s(g):=(\pi(g)f_s,f_{-s})=\int_{K}f_s(kg)\overline{f_{-s}(k)}dk.
\end{equation}
These $\Xi_s(g)$ are bi-$K$-invariant. If $s=0$, we call
$$\Xi(g):=\Xi_0(g)$$
the Harish-Chandra spherical fucntion. When $s\geq 0$, we see the integrals are always positive and satisfy 
\begin{equation}
	|y|^{-\frac{1}{2}+s+\epsilon}\ll \Xi_s(a(y))\ll |y|^{-\frac{1}{2}+s+\epsilon}.
\end{equation}
With $\Xi(g)$, we have the following well-known estimate for matrix coefficients.

\begin{lemma}\label{lem. decay of matrix coefficients}
	
	Let $\pi$ be a unitary irreducible representation of $G$. Let $x_1, x_2$ be two $K$-finite vectors in $\pi$. Then: 
	\begin{enumerate}[label=\textnormal{(\arabic*)}]
		\item If $\pi$ is tempered, then
		\begin{equation}
			\left\langle\pi(g) x_1, x_2\right\rangle \leq \operatorname{dim}\left(K. x_1\right)^{1 / 2} \operatorname{dim}\left(K. x_2\right)^{1 / 2}\left\|x_1\right\| \left\|x_2\right\|\cdot \Xi(g).
		\end{equation}
		\item  If $\pi$ is $\theta$-tempered, then 
		\begin{equation}
			\left\langle\pi(g) x_1, x_2\right\rangle \ll_\epsilon  \operatorname{dim}\left(K. x_1\right)^{1 / 2} \operatorname{dim}\left(K. x_2\right)^{1 / 2}\left\|x_1\right\| \left\|x_2\right\|\cdot \Xi(g)^{1-\theta-\epsilon}.
		\end{equation}
		\item Furthermore, let $0<\vartheta<\alpha<\frac{1}{2}$ and $\epsilon>0$. Let $\pi$ be $\vartheta$-tempered. Then
		\begin{equation}
			\left\langle\pi(g) x_1, x_2\right\rangle \ll_\epsilon  \operatorname{dim}\left(K. x_1\right)^{\frac{1}{2}+\epsilon} \operatorname{dim}\left(K. x_2\right)^{\frac{1}{2}+\epsilon}\left\|x_1\right\| \left\|x_2\right\|\cdot \Xi_{\alpha}(g).
		\end{equation}
	\end{enumerate}
	Here $\operatorname{dim}\left(K. x\right)$ is the dimension of the span of $x$ by $K$-action.
	
\end{lemma}

\begin{proof}
	
	The first lemma is a special case in \cite{CHH88} for tempered representation for semisimple groups. For the non-tempered case, we refer to \cite{Ve10}, \cite[Prop 2.31]{Wu14}. The last lemma comes from \cite{Nel25}.
	
\end{proof}

\subsubsection{A solid upper bound}

\begin{proposition}[Upper bound for local triple product integral for analytic newvectors]\label{Final proof. upper bound for local triple integral}
	
	Let $\chi$ be a unitary character of $F^\times$, $T=T(\chi)$ be the spectral parameter defined in (\ref{def. spectral parameters}) and $\pi=\mathcal{I}(\chi)$ be the tempered principal series representation induced from $\chi$. Let $f=f_\chi\in \mathcal{I}(\chi)$ be a family of analytic newvectors, and $W=W_{f_\chi}$ be their Whittaker functions. 
	
	Let $\pi'$ be another $\vartheta$-tempered irreducible unitary representation for $G$, and $\psi\in \pi'$ be unit vectors. Then there exists an absolute constant $d_0\geq 0$, independent of $\pi$ and $\pi'$, such that
	\begin{equation}\label{step. upper bound}
		I^T(W,\bar{W},a(Q)\psi)\ll_\epsilon |Q|_F^{-1}\left|\frac{Q}{T}\right|_F^{2\vartheta+\epsilon}\mathcal{S}_{d_0}(\psi)^2, 
	\end{equation}
	for any $Q\geq|T|$. 
	
\end{proposition}

\begin{proof}
	
	According to \cite[Lemma 6.8]{Nel19}, we only need to prove (\ref{step. upper bound}) for those $\psi\in \pi'$ are $K$-isotypic vectors. We assume that $\pi'$ is $\theta$-tempered with $\theta<\alpha<\frac{1}{2}$, and $\psi_\alpha$ is the unit spherical vector in $\pi_\alpha$ as in \S3.5.1. As in \cite{Nel25}, Lemma \ref{lem. decay of matrix coefficients} shows
	$$\begin{aligned}
	I^T(W,\bar{W},\pi'(a(Q))\psi)&=I^T(\pi(a(Q)^{-1})W,\pi(a(Q)^{-1})\bar{W},\psi)\\
	&\leq \int_{g\in G}\left|\langle \pi(g)\pi(a(Q)^{-1})W,\pi(a(Q)^{-1})\bar{W}\rangle\right|^2\langle \pi'(g)\psi,\psi\rangle dg\\
	&\ll_\epsilon \mathrm{dim}(K.\psi)^{1+2\epsilon}\cdot \int_{g\in G}\left|\langle \pi(g)\pi(a(Q)^{-1})W,\pi(a(Q)^{-1})\bar{W}\rangle\right|^2\Xi_{\alpha}(g) dg\\
	&=\mathrm{dim}(K.\psi)^{1+2\epsilon}\cdot I^T(W,\bar{W},\pi_0(a(Q))\psi_\alpha)
	\end{aligned}.$$
	We note that these spherical vectors are $K$-invariant, then according to Iwasawa decomposition 
	$$\begin{aligned}
	\psi_\alpha\begin{pmatrix} y & 0 \\ x & 1\end{pmatrix}&=\psi_\alpha\left(\begin{pmatrix} \frac{y}{\sqrt{1+|x|^2}} & \frac{\bar{x}y}{\sqrt{1+|x|^2}} \\ 0 & \sqrt{1+|x|^2}\end{pmatrix}\begin{pmatrix} \frac{1}{\sqrt{1+|x|^2}} & -\frac{\bar{x}}{\sqrt{1+|x|^2}} \\ \frac{x}{\sqrt{1+|x|^2}} & \frac{1}{\sqrt{1+|x|^2}}\end{pmatrix}\right)
	\end{aligned}$$
	we get 
	$$\psi_\alpha\begin{pmatrix} y & 0 \\ x & 1\end{pmatrix}=\left|\frac{y}{1+|x|^2}\right|_F^{\frac{1}{2}+\alpha}\psi_\alpha(1).$$
	Furthermore, we have 
	$$\begin{aligned}
	M^*(|\cdot|^\alpha)\psi_\alpha(1)&=\gamma_F(|\cdot|_F^{2\alpha},\psi_F,0)\int_{t\in F}\psi_\alpha(wn(t))d_Ft\\
	&=\gamma_F(|\cdot|_F^{2\alpha},\psi_F,0)\int_{t\in F}\left|\frac{1}{1+|t|^2}\right|_F^{\frac{1}{2}+\alpha}d_Ft\cdot \psi_\alpha(1)=\frac{\zeta_F(1-2\alpha)}{\zeta_F(1+2\alpha)}\psi_\alpha(1)
	\end{aligned}$$
	(In particular, when $\alpha=0$, it equals to $\psi_{\alpha}$ by analytic continuation). That gives
	$$\|\psi_\alpha\|^2=\langle \psi_\alpha,M^*(|\cdot|^{2\alpha})\psi_\alpha\rangle=\frac{\zeta_F(1-2\alpha)}{\zeta_F(1+2\alpha)}\cdot|\psi_\alpha(1)|^2\cdot \mathrm{vol}(K).$$

	Now, we apply the linearlization of the local trilinear form from Lemma \ref{trilinear form} and get
	$$	I^T(W,W,a(Q)\psi_\alpha)=\overline{\gamma\left(\pi\otimes\pi\otimes|\cdot|_F^\alpha,\psi,\frac{1}{2}\right)}|\Psi(W,W,a(Q)\psi_\alpha)|^2.$$
	where the $\gamma$-factor is real-valued and satisfying that
	$$\gamma\left(\pi\otimes\pi\otimes|\cdot|_F^\alpha,\psi,\frac{1}{2}\right)\asymp C(\pi\otimes\pi)^{-\alpha}\asymp |T|^{-2\alpha}.$$
	The local Rankin-Selberg integrals are given by 
	$$\begin{aligned}
		\Psi(W,W,a(Q)\psi_\alpha)&=\int_{x\in F}\int_{y\in F^\times}\left|W\begin{pmatrix} y & 0 \\ x &1 \end{pmatrix}\right|^{2}\psi_\alpha\begin{pmatrix} Qy & 0 \\ Qx &1 \end{pmatrix}\frac{d_F^\times y}{|y|_F}d_Fx\\
		&=|Q|_F^{\frac{1}{2}+\alpha}\psi_\alpha(1)\cdot\int_{x\in F}\int_{y\in F^\times}\left|W\begin{pmatrix} y & 0 \\ x &1 \end{pmatrix}\right|^{2}\frac{|y|_F^{-\frac{1}{2}+\alpha}}{\left|1+|Qx|^2\right|_F^{\frac{1}{2}+\alpha}}d^\times_Fyd_F x.
	\end{aligned}
	$$
	For $|x|\leq \frac{1}{|T|}$, we apply proposition \ref{prop. L^2 concentration of translated Whittaker functions}, which gives
	$$\int_{y\in F^\times}\left|W\begin{pmatrix} y & 0 \\ x &1 \end{pmatrix}\right|^{2}|y|_F^{-\frac{1}{2}+\alpha}d^\times_Fy\ll 1,$$
	For $|x|\geq \frac{1}{|T|}$, we apply the invariance property \ref{prop. symmetric of Whittaker function} and again proposition \ref{prop. L^2 concentration of translated Whittaker functions}, and get
	$$\int_{y\in F^\times}\left|W\begin{pmatrix} y & 0 \\ x &1 \end{pmatrix}\right|^{2}|y|_F^{-\frac{1}{2}+\alpha}d^\times_Fy=\int_{y\in F^\times}\left|W\begin{pmatrix} \frac{y}{T^2x^2} & 0 \\ -\frac{1}{T^2x} &1 \end{pmatrix}\right|^{2}|y|_F^{-\frac{1}{2}+\alpha}d^\times_Fy\ll |Tx|_F^{-1+2\alpha}.$$
	Hence by the assumption $Q\geq |T|$, the integral is bounded by 
	$$\begin{aligned}
		\Psi(W,W,a(Q)\psi_\alpha)&\ll \psi_\alpha(1) |Q|_F^{\frac{1}{2}+\alpha}\cdot\left(\int_{|x|\leq \frac{1}{|T|}}\frac{1}{|1+|Qx|^2|_F^{\frac{1}{2}+\alpha}}d_Fx+\int_{|x|\geq \frac{1}{|T|}}\frac{|Tx|_F^{-1+2\alpha}}{|1+|Qx|^2|_F^{\frac{1}{2}+\alpha}}d_Fx\right)\\
		&=\psi_\alpha(1)|Q|_F^{-\frac{1}{2}+\alpha}\cdot\left(\int_{|x|\leq \frac{Q}{|T|}}\frac{1}{|1+|x|^2|_F^{\frac{1}{2}+\alpha}}d_Fx+\left|\frac{Q}{T}\right|_F^{1-2\alpha}\int_{|x|\geq \frac{Q}{|T|}}\frac{|x|_F^{-1+2\alpha}}{|1+|x|^2|_F^{\frac{1}{2}+\alpha}}d_Fx\right)\\
		&\ll \psi_{\alpha}(1)|Q|_F^{-\frac{1}{2}+\alpha}\cdot\left(\left|\frac{Q}{T}\right|_F^{\epsilon}+\left|\frac{Q}{T}\right|_F^{-2\alpha}\right)
	\end{aligned}
	$$
	Finally, recall that, when $\psi$ is a weight $l$ unit vector in $\pi_\theta$, then the $d$-th Sobolev norm of $\psi$ is given by $\mathcal{S}_d(\psi)=(1+\frac{1}{4}-\theta^2+\frac{l^2}{4})$. Moreover, we set $\alpha=\vartheta+\epsilon$ is fixed. Then we conclude that there exists $d_0'$ such that 
	$$I^{T}(W,W,a(Q)\psi)\ll_\epsilon \dim(K.\psi)^{1+2\epsilon}|\psi_\alpha(1)|^2|Q|_F^{-1}\left|\frac{Q}{T}\right|_F^{2\alpha+\epsilon}\ll_{\epsilon'} |Q|_F^{-1}\left|\frac{Q}{T}\right|_F^{2\vartheta+\epsilon'}\mathcal{S}_{d_0'}^2(\psi),$$
	for any $Q\geq |T|$ and then pass through to any $\psi\in \pi$ by replacing $d_0'$ to some $d_0$. 

\end{proof}

\subsubsection{Resolution of the local test vector conjecture}

Now we are going to prove Theorem \ref{Intro. local test vector problem}. Let $\pi_i=\mathcal{I}(\chi_i)$, with $\chi_3=|\cdot|_F^{s}$. Let $Q>0$ such that $|Q|_F=C(\pi_1\otimes\pi_2)^{\frac{1}{2}}$. According to Theorem \ref{Final proof. lower bound for local Rankin-Selberg integral}, for the given $\pi_3$, there exists a fixed $R>0$, a family of analytic newvectors $f_\chi$ and one $f_3^0\in \mathcal{I}(\chi_3)$, such that for any  $\min\{C(\pi_1),C(\pi_2)\}\geq R$, we have 
$$I^T(W_{f_{\chi_1}},W_{f_{\chi_2}},a(Q)f_3^0)\gg |Q|_F^{-1}.$$ 
Since for any $\pi_1,\pi_2$, we always have $C(\pi_1\otimes\pi_2)^{1/2}\geq C(\pi_2\otimes\pi_2)^{1/2}$. Then we can apply Proposition \ref{Final proof. upper bound for local triple integral} to $W_{f_{\chi_2}},$ and obtain (\ref{Intro. upper bound}) hold. Then $f_{\chi_1}$, $f_{\chi_2}$, $f_3^0$ provide the solution. When $\min\{C(\pi_1),C(\pi_2)\}\leq R$ is bounded, including the cases $\pi_1,\pi_2$ are complementary series, we refer to the results in \cite[\S 3.6]{MV10}, where they resolve the conjecture for fixed $\pi_2$ and hence for bounded spectral parameters.

\section{Application to subconvexity problems}

For the sake of self-containment, we briefly review the period approach to the subconvexity problem for triple product L-functions, following \cite{MV10} and \cite{HMN22}. In particular, we show how Theorem \ref{Intro. main theorem} can be deduced from the local test vector problem \ref{Intro. local test vector problem}.

\subsection{Triple product period formula}

\subsubsection{Automophic representations}

Let $\mathbb{F}$ be a number field, $\mathbb{A}$ be its adele ring. Let $\mathbf{G}=\mathrm{PGL}_2$, $\mathbf{X}=\mathrm{PGL}_2(\mathbb{F})\backslash \mathrm{PGL}_2(\mathbb{A})$ be its adelic quotient. $\mathbf{X}$ is non-compact with finite volume $V_\mathbf{X}$. The measure on $\mathbf{X}$ is normalized to be the product of those measures in \S \ref{sec. notation}.  We also fix an addictive additive character $\psi_\mathbb{F}=\otimes_v\psi_v$ as in \S \ref{sec. notation}.  

Let $L^2(\mathbf{X})$ be the space of functions on $\mathrm{PGL}_2(\mathbb{A})$ such that
$$\int_{\mathbf{X}}|\varphi(g)|^2dg<\infty.$$
$\mathrm{PGL}_2(\mathbb{A})$ acts on its by right regular representation. 

For $\varphi$ a function on $\mathbf{X}$, we denote its constant term $\varphi_N$ and Whittaker function $W_\varphi$ by
$$\varphi_N(g):=\int_{\mathbb{A}/\mathbb{F}}f(n(t)g)dt,\quad W_\varphi(g)=\int_{\mathbb{A}/\mathbb{F}}W(n(t)g)\overline{\psi_\mathbb{F}(t)}dt,$$
then one has Fourier expansion
$$\varphi(g)=\varphi_N(g)+\sum_{\alpha\in \mathbb{F}^\times}W_\varphi(a(\alpha)g).$$
We say $\varphi$ is a cusp form if $\varphi_N=0$.

The space of cusp forms form a subrepresentation of $\mathbf{G}$ in $L^2(\mathbf{X})$. We denote it by $L_{\mathrm{cusp}}^2(\mathbf{X})$. It is well known that (See \cite{GJ77}) we have an orthogonal decomposition 
\begin{equation}\label{spectral decomposition}
	L^2(\mathbf{X})=L^2_{\mathrm{cusp}}(\mathbf{X})\oplus L^2_{\mathrm{cont}}(\mathbf{X})\oplus L^2_{\mathrm{res}}(\mathbf{X}).
\end{equation}
The cuspidal part 
$L_{\text{cusp}}^2(\mathbf{X})$ decomposes into irreducible representations known as cuspidal automorphic representations. The continuous part $L_{\text {cont}}^2(\mathbf{X})$ is constructed from Eisenstein series, and its constituent representations are called Eisenstein space. Together, the irreducible components of the cuspidal and continuous spectrum are known as the generic standard representations for $\mathbf{G}$. Finally,  the residue part $L_{\text{res}}^2(\mathbf{X})$ is the sum of all one-dimensional subrepresentations of $L^2(\mathbf{X})$, each are of form $\chi\circ \det$, $\chi^2=1$.

\subsubsection{L-functions}

For general reductive group $\mathbf{G}$, $\Pi$ be an irreducible generic automorphic representation for $\mathbf{G}$. We also have a tensor product decomposition $\Pi=\otimes_v\Pi_v$. We may associate it with $L(\Pi_v,s)=L(\rho_v,s)$ in the local Langlands correspondence and we define
$$L(\Pi,s):=\prod_{v< \infty}L_v(\Pi_v,s),\quad \Lambda(\Pi,s):=\prod_vL_v(\Pi_v,s)$$
be its L-function and complete L-function. We define 
$$C(\Pi)=\prod_v C(\Pi_v)$$
to be its analytic conductor defined in \cite{IS00}. 

It is well-known that when $\pi$ is a generic automorphic representation of $\mathrm{GL}_2$, the adjoint L-fuction has no pole at $1$ when $\pi$ is cuspidal, or has a single pole when $\pi$ is a non-singular Eisenstein series, or has an order $3$ pole when $\pi$ is a singular Eisenstein series. We define $L^*(\pi,\mathrm{Ad},1)$ to be the first non-zero coefficient for $\pi$. Then we have 
\begin{equation}\label{Final proof. estimate for adjoint L functions}
	C(\pi)^{-\epsilon}\ll_\epsilon L^*(\pi,\mathrm{Ad},1)\ll_\epsilon C(\pi)^\epsilon.
\end{equation}
When $\pi$ comes from Maass form on $\mathbb{Q}$, it due to \cite{HL94}, and \cite{BH10} in general.

\subsubsection{Eisenstein series}

Let $\chi$ be a unitary character of $\mathbb{A}^\times/\mathbb{F}^\times$. Let  $\mathcal{I}(\chi)$ be the space of smooth functions on $\mathbf{G}(\mathbb{A})$ such that 
$$f(a(y)n(x)g)=|a|_\mathbb{A}^{1/2}\chi(a)f(g),\quad \langle f,f\rangle=\int_K|f(k)|^2dk.$$
Given $f$ in such a space, we denote by $\mathrm{Eis}(f)$ the corresponding Eisenstein series (defined by analytic continuation in general)
$$\mathrm{Eis}(f)(g):=\sum_{\gamma\in \mathrm{B}(\mathbb{F})\backslash \mathrm{PGL}_2(\mathbb{F})}f(\gamma g)$$
For $s$ a complex parameter and $\pi=\mathcal{I}(\chi)$, we set the deformation of $\pi$ 
$$\pi(s)=\mathcal{I}(\chi|\cdot|^s_\mathbb{A}).$$
For $f\in \pi$ we define $f_s\in \pi(s)$ to be the unique fuction in $\pi(s)$ whose restriction to $K$ coincide with $f$. 

\subsubsection{Canonical norms}

As in \cite[Lemma 2.2.3]{MV10}, we may define the canonical norms for $\varphi\in \pi$ by 
\begin{equation}\label{def. canonical norms}
	\|\varphi\|_{\mathrm{can}}^2:=\begin{cases}
		\|\varphi\|_{L^2(\mathbf{X})}^2& \text{if }\pi \text{ is cusipidal};\\
		2\Lambda_\mathbb{F}^*(\pi,\mathrm{Ad},1)\|\varphi\|_{\mathrm{Eis}}^2 & \text{if }\pi \text{ is Eisenstein}.
	\end{cases}
\end{equation}
Using \cite[\S 2.3]{Za19} we can compare the global and the local inner product: for $\varphi=\otimes_v\varphi_v\in \pi=\otimes_v\pi_v$, with $\pi$ either cuspidal or nonsingular Eisenstein series, we have 
\begin{equation}
	\|\varphi\|_\mathrm{can}^2=2\Delta_\mathbb{F}^{1/2}\Lambda^*(\pi,\mathrm{Ad},1)\prod_v\langle \varphi_v,\varphi_v\rangle_{\pi_v},
\end{equation}
where $\Lambda(\pi,\mathrm{Ad},s)$ is the complete adjoint L-functions.

\subsubsection{The regularized Plancherel formula}

In this section, we review the regularization process defining integrals for non-necessarily decaying functions on $\mathbf{X}$. It was firstly introduced by \cite{Zag82}, and then developed adelically by \cite[Section 4.3.1]{MV10}. We also metion the work of \cite{Wu14} and \cite{Za20}. The following definitions and theorems are quoted from \cite[Section 2.3]{Za20}.

	
	Let $\varphi$ be a smooth function on $\mathbf{X}$. We say that $\varphi$ is finitely regularizable if there exists characters $\chi_i: \mathbb{F}^{\times} \backslash \mathbb{A}^{\times} \rightarrow \mathbb{C}^{(1)}, \alpha_i \in \mathbb{C}, n_i \in \mathbb{N}$ and smooth functions $f_i \in \operatorname{Ind}_{\mathrm{B}\left(\mathbb{A}\right) \cap \mathbf{K}}^{\mathbf{K}}\left(\chi_i, \chi_i^{-1}\right)$, $i=1, \ldots, m$ such that for any $N \geqslant 1$
	\begin{equation}
		\varphi(n(x) a(y) k)=\varphi_{\mathrm{N}}^*(a(y) k)+O\left(|y|^{-N}\right), \quad \text { as }|y| \rightarrow \infty
	\end{equation}
	where the essential constant term is denoted by
	\begin{equation}
		\varphi_{\mathrm{N}}^*(n(x) a(y) k)=\varphi_{\mathrm{N}}^*(a(y) k)=\sum_{i=1}^m \chi_i(y)|y|^{1 / 2+\alpha_i}(\log |y|)^{n_i} f_i(k).
	\end{equation}
	The set of exponents of $\varphi$ is then defined by
	\begin{equation}\label{def. exponents}
		S_\varphi:=\left\{\chi_i|\cdot|^{1 / 2+\alpha_i}: 1 \leqslant i \leqslant m\right\}.
	\end{equation}
	We denote by $\mathcal{A}^{\mathrm{fr}}(\mathbf{G})$ to be the space of finitely regularizable functions on $\mathbf{X}$. 
	
	Let $S$ be a finite set of exponents and let $\mathcal{V}_S$ be the space generated by finitely regularizable functions whose exponents belong to $S$. Define
	\begin{equation}\label{def. regular V}
		\mathcal{V}:=\bigcup_{\substack{|S|<\infty \\ \chi \in S \Rightarrow \chi^2 \neq|\cdot|^2}} \mathcal{V}_S,
	\end{equation}

\begin{theorem}
	The integral on $L^1(\mathbf{X}) \cap \mathcal{V}$ extends to a $\mathrm{GL}_2\left(\mathbb{A}_{\mathbf{F}}\right)$-invariant linear functional on $\mathcal{V}$, denoted by $\int_{\mathbf{X}_{\mathrm{PGL}_2}}^{\mathrm{reg}}$.
\end{theorem}

\begin{proof}
	See \cite[\S 4.3.5]{MV10} and \cite[Theorem 2.5.]{Za20}.
\end{proof}

\begin{remark}
	The space $\mathcal{A}^{\mathrm{fr}}\left(\mathrm{GL}_2\right)$ contains the following functions
	\begin{itemize}
		\item Smooth cusp forms, whose exponent sets are $\emptyset$;
		\item 1-dimensional characters $g \mapsto$ $\chi(\operatorname{det} g)$ for $\chi$ a quasi-character satisfying $\chi^2=1$;
		\item Eisenstein series associated to the induced representation $\chi_1 \boxplus \chi_2$, whose exponent sets are $\left\{\chi_1|\cdot|^{1 / 2}, \chi_2|\cdot|^{1 / 2}\right\}$. 
	\end{itemize}
	Furthermore, if $\varphi_i\in \mathcal{A}^{\mathrm{fr}},i=1,2$, then $\varphi_1\varphi_2\in \mathcal{A}^{\mathrm{fr}}$ with exponent set  $S_{\varphi_1\varphi_2}=S_{\varphi_1}S_{\varphi_2}$.

\end{remark}
\vspace{11pt}

	We define the space $\mathcal{E}(\mathbf{G})$ the $\mathbb{C}$-vector space spanned by the derivatives of all Eisenstein series for $\mathbf{G}$. And we introduce the space
	\begin{equation}
		\mathcal{W}:=\left\{\varphi \in \mathcal{V} \left\lvert\, \Re (\chi) \neq \frac{1}{2}\right., \forall \chi \in S_\varphi\right\}.
	\end{equation}

	\begin{proposition}
		  
		 We have 
		  
		\begin{enumerate}[label=\textnormal{(\arabic*)}]
			\item For every $\varphi\in \mathcal{W}$, there exists a unique function $\mathcal{E}(\varphi)\in \mathcal{E}(\mathbf{G})$ such that
			$$\varphi-\mathcal{E}(\varphi)\in L^1(\mathbf{X}).$$ 
			When $S_\varphi$ have real part $>\frac{1}{2}$, then $\mathcal{E}(\varphi)$ could be taken as $\mathcal{E}(\varphi)=\mathrm{Eis}(\varphi_N^*)$. 
			\item For any $\mathcal{E}\in \mathcal{E}(\mathbf{G})$, we have
			$$\int_{\mathbf{X}}^{\mathrm{reg}}\mathcal{E}=0.$$
			\item Furthermore, for any $\mathcal{E}_i\in \mathcal{E}(\mathbf{G})$, $i=1,2$ with $S_{\mathcal{E}_1}\cap S_{\mathcal{E}_2}=\emptyset$, we have 
			$$\int_\mathbf{X}^{\mathrm{reg}}\mathcal{E}_1\overline{\mathcal{E}_2}=0.$$
		\end{enumerate}
		
	\end{proposition}
 	\vspace{11pt}
 
	Given $\varphi_1, \varphi_2 \in \mathcal{A}^{\mathrm{fr}}\left(\mathbf{G}\right)$ such that $\varphi_1 \bar{\varphi}_2 \in$ $\mathcal{V}$, we define the regularized inner product as
	\begin{equation}
		\left\langle\varphi_1, \varphi_2\right\rangle_{\mathrm{reg}}:=\int_{\mathbf{X}_{\mathrm{PGL}_2}}^{\mathrm{reg}} \varphi_1 \bar{\varphi}_2
	\end{equation}

\begin{theorem}[Regularized Plancherel Formula]
	Let  $\varphi_1, \varphi_2 \in \mathcal{W}$ such that the product $\varphi_1 \bar{\varphi}_2$ belong to $\mathcal{V}$. If $S_{\varphi_1}\cap S_{\varphi_2}=\emptyset$, then
	\begin{equation}\label{thm. Regularized Plancherel Formula}
		\left\langle\varphi_1, \varphi_2\right\rangle_{\mathrm{reg}}=\left\langle\varphi_1, \varphi_2\right\rangle_{\mathrm{gen}}+\left\langle\varphi_1, \varphi_2\right\rangle_{\mathrm{finite}}+\left\langle\varphi_1, \varphi_2\right\rangle_{\mathrm{deg}}
	\end{equation}
	where
	\begin{equation}\label{def. generic contribution}
		\begin{aligned}
		\left\langle\varphi_1, \varphi_2\right\rangle_{\mathrm{gen}}=&\sum_{\substack{\pi 
				\textnormal{ cuspidal }}} \sum_{\psi \in \mathscr{B}(\pi)}\left\langle\varphi_1, \psi\right\rangle\left\langle\psi, \varphi_2\right\rangle\\
		+&\sum_{\substack{
				\chi \in  \widehat{\mathbb{F}^{\times} \backslash \mathbb{A}^{(1)}}}} \int_{-\infty}^{\infty} \sum_{\varphi_{it} \in \mathscr{B}\left(\chi|\cdot|^{it}, \chi^{-1}|\cdot|^{-it}\right)}\left\langle\varphi_1, \mathrm{E}\left(, \varphi_{i t}\right)\right\rangle_{\mathrm{reg}}\left\langle\mathrm{E}\left(, \varphi_{i t}\right), \varphi_2\right\rangle_{\mathrm{reg}} \frac{d t}{4 \pi}
			\end{aligned}
	\end{equation}
	is the generic contribution in the spectral decomposition, while the finite contribution and degenerate contribution are given by 
	\begin{equation}\label{def. finite contribution}
		\left\langle\varphi_1, \varphi_2\right\rangle_{\mathrm{finite}}=V_{\mathbb{F}}^{-1} \sum_{\chi^2=1}\left\langle\varphi_1, \varphi_\chi\right\rangle_{\mathrm{reg}}\left\langle\varphi_\chi, \varphi_2\right\rangle_{\mathrm{reg}}
	\end{equation}
	
	\begin{equation}\label{def. degenerate contribution}
		\left\langle\varphi_1, \varphi_2\right\rangle_{\mathrm{deg}}=\left\langle\varphi_1, \mathcal{E}_2\right\rangle_{\mathrm{reg}}+\left\langle\mathcal{E}_1, \varphi_2\right\rangle_{\mathrm{reg}}
	\end{equation}
	where $\mathcal{E}_i=\mathcal{E}\left(\varphi_i\right)$.
\end{theorem}

\begin{proof}
	See \cite[Theorem 2.8]{Za20}, which generalize \cite[Proposition 4.3.8]{MV10}. 
\end{proof}

\subsubsection{The regularized triple product formula}

Let $\pi_1,\pi_2,\pi_3$ are three generic representations for $\mathrm{PGL}_2(\mathbb{A})$. The central value for the triple product L-functions $L(\pi_1\otimes\pi_2\otimes\pi_3,\frac{1}{2})$ has a long story, which was studied by \cite{Gar87}, \cite{PSS87}, \cite{Pra90},\cite{HK91}, \cite{Wat02}, and culmulated to the beautiful formula in \cite{Ich08} and \cite{MV10}. Nowadays, we study it from the point of view of triple product period formula. 

For $\varphi_i\in \pi_i$, $i=1,2,3$, we consider the linear functional on $\pi_1\otimes\pi_2\otimes\pi_3$ defined by
\begin{equation}\label{def. Trilinear period integral}
	I(\varphi_1\otimes\varphi_2\otimes\varphi_3):=\int_{\mathbf{X}}^{\mathrm{reg}}\varphi_1\varphi_2\varphi_3.
\end{equation}
Recall the local triple product integral is defined by 
$$I_v^T(\varphi_{1,v},\varphi_{2,v},\varphi_{3,v}):=\int_{\mathrm{PGL}_2(\mathbb{F}_v)}\prod_{i=1}^3\langle \pi_{i,v}(g)\varphi_{i,v},\varphi_{i,v}\rangle d_vg$$
and are normalized to be $I_v^0$ at finite places as in (\ref{def. normalized local triple product integral}). 

The following theorem connects the global trilinear form $I$ with the central value of triple product L-function $L\left(\pi_1 \otimes \pi_2 \otimes \pi_3, \frac{1}{2}\right)$ and the local triple product integrals $I_v^0$. 

\begin{theorem}
	
	Let $\pi_1, \pi_2, \pi_3$ be unitary automorphic representations of $\mathrm{PGL}_2\left(\mathbb{A}\right)$. Assume $\varphi_i=\otimes_{v}\varphi_{i,v}$ are factorizable vectors. Then there is a contant $C_\mathbb{F}$ depending on $\mathbb{F}$ and the type of the representations, such that the following formula holds
	\begin{equation}\label{thm. Triple product formula}
		\frac{|I(\varphi_1\otimes\varphi_2\otimes\varphi_3)|^2}{\prod_{i=1}^3\left\|\varphi_i\right\|_{\mathrm{can}}^2}=C_\mathbb{F}\cdot \frac{L\left(\pi_1 \otimes \pi_2 \otimes \pi_3, \frac{1}{2}\right)}{\prod_{i=1}^3 L^*\left(\pi_i, \mathrm{Ad}, 1\right)} \prod_{v\nmid\infty} \frac{I_v^0\left(\varphi_{1,v},\varphi_{2,v},\varphi_{3,v}\right)}{\prod_{i=1}^3\left\langle\varphi_{i, v},\varphi_{i,v}\right\rangle_v} \prod_{v|\infty}\frac{I_v^T\left(\varphi_{1,v},\varphi_{2,v},\varphi_{3,v}\right)}{\prod_{i=1}^3\left\langle\varphi_{i, v}, \varphi_{i, v}\right\rangle_v}
	\end{equation}
	The canonical norms are introduced in (\ref{def. canonical norms}). The local factors  $\frac{I_v^0\left(\varphi_v\right)}{\prod_{i=1}^3\left\langle\varphi_{i, v}, \varphi_{i,v}\right\rangle_v}=1$ for almost all places.
	
\end{theorem}
 
\begin{proof}
	When all $\pi_i$ are cuspidal is due to Ichino \cite{Ich08} and when at least one of the $\pi_i$ 's is Eisenstein due to \cite[Lemma 4.4.3]{MV10}.
\end{proof}

\subsection{The main result}

In this section, we focus on proving the spectral aspect.

\begin{theorem}\label{Final proof. main theorem}
	
	Let $\pi_1,\pi_2,\pi_3$ be irreducible automorphic representations for $\mathrm{PGL}_2(\mathbb{A})$. Assume that $\pi_1,\pi_2$ are cuspidal representations, while $\pi_3$ is a fixed cuspidal or Eisenstein series representation. Suppose that for all archimedean places $v$, the local components $\pi_{i,v}$ are principal series representation, and that $\pi_{i,v}$ are unramified at all finite places, for $i=1,2,3$. Then there exists an absolute constant $\delta>0$, such that 
	\begin{equation}\label{Final proof. Main result}
		L(\pi_1\otimes\pi_2\otimes\pi_3,\frac{1}{2})\ll_{\epsilon,\pi_3} C(\pi_1\otimes\pi_2)^{\frac{1}{2}+\epsilon}\left(\frac{C(\pi_1\otimes\pi_2)}{C(\pi_2\otimes\pi_2)}\right)^{-\delta}.
	\end{equation}
	
\end{theorem}

\begin{remark}
	By combining Theorem \ref{Final proof. main theorem} with ingredients and proofs of \cite[Theorem 1.3]{HMN22}, we obtain Theorem \ref{Intro. main theorem}, which yields a hybrid bound involving both the spectral and level aspects. In particular, Corollary \ref{Intro. main corollary} and Theorem \ref{Intro. main theorem for Maass forms} follow as direct consequences of Theorem \ref{Intro. main theorem} in specific settings. The remainder of this section is devoted to the proof of Theorem \ref{Final proof. main theorem}.
\end{remark}

\subsubsection{Test vectors for period approach}

Recall that in the period approach to the subconvexity problem for triple product $L$-functions, we apply the triple product period formula (\ref{thm. Triple product formula}) and obtain
\begin{equation}\label{Final proof. triple product formula}
	\left|\int_{\mathbf{X}}\varphi_1(g)\varphi_2(g)\varphi_3(g)dg\right|^2\sim \frac{L(\pi_1\otimes\pi_2\otimes\pi_3,\frac{1}{2})}{\prod_{i=1}^3L(\pi_i,\mathrm{Ad},1)}\prod_v I^T_v(\varphi_{1,v},\varphi_{2,v},\varphi_{3,v}).
\end{equation}
for test vectors $\varphi_i \in \pi_i$. Following \cite{MV10}, to obtain an upper bound for these $L$-functions, one needs to choose the test vectors $\varphi_i$ so that the local period integrals $I_v^T$ admit a suitable lower bound, while the global period integral admits a non-trivial saving in the upper bound.

In our setting, $\pi_3$ is fixed and $\pi_1,\pi_2$ vary. Let $Q_v=C(\pi_{1,v}\otimes\pi_{2,v})^{1/2}.$ We choose the test vectors to satisfy Theorem \ref{Intro. local test vector problem}:

\begin{enumerate}
	\item $\varphi_i=\otimes\varphi_{i,v}$ are pure tensors in $\pi_i=\otimes \pi_{i,v}$. Each $\varphi_{i,v}$ is of $L^2$-norm $1$.
	\item For non-archimedean places $v$, the representations $\pi_{i,v}$ are unramified, and we take $\varphi_{i,v}$ to be the normalized spherical vectors.
	\item For archimedean places $v$, we choose $\varphi_{1,v}$, $\varphi_{2,v}$, and $\varphi_{3,v}=a(Q_v)\varphi_{3,v}^0$ as in Thm \ref{Intro. local test vector problem}.

\end{enumerate}

\subsubsection{The first step}

Applying the triple product period formula \eqref{Final proof. triple product formula}, together with the bounds for adjoint $L$-functions \eqref{Final proof. estimate for adjoint L functions} and the local lower bounds \eqref{Intro. lower bound}, we obtain 
\begin{equation}\label{Final proof. lower bound}
	\begin{aligned}
		\left|\int_{\mathbf{X}}\varphi_1(g)\varphi_2(g)\varphi_3(g)\right|^2
		&\gg_{\pi_3,\epsilon} (C(\pi_1)C(\pi_2))^{-\epsilon}L(\pi_1\otimes\pi_2\otimes\pi_3,\frac{1}{2})\prod_v I^T_v\\
		&\gg_{\pi_3,\epsilon'} C(\pi_1\otimes\pi_2)^{-\frac{1}{2}-\epsilon'}L(\pi_1\otimes\pi_2\otimes\pi_3,\frac{1}{2}).
	\end{aligned}
\end{equation}

\subsection{Upper bounds for period integrals}

\subsubsection{Amplification and spectral decomposition}

To give an upper bound for the global period integrals, we apply the amplification method in \cite[\S 5.2.3]{MV10}, see also \cite[\S 4.1]{Ve10}, \cite[\S 3.3]{HMN22}. We  begin with the construction of the amplifier $\sigma$, quoting \cite[Lemma 3.7]{HMN22}.

\begin{lemma}\label{amplifier}
	Let $L\geq 1$ be a large parameter to be chosen. There exists a complex-valued measure $\sigma$ on $\mathrm{G}(\mathbb{A}_{\mathrm{fin}})$, which satisfies that 
	\begin{enumerate}[label=\textnormal{(\arabic*)}]
		\item $\sigma$ is compactly supported on $\mathrm{G}(\mathbb{A}_{\mathrm{fin}})$. For any $u\in \mathrm{Supp}(\sigma)$, its operator norm $\|u\|\leq L$.
		\item $\mathrm{Supp}(\sigma)$ commutes with $\mathrm{GL}_2(F_v)$ for any $v$ archimedean or the chosen $\psi_v$ is ramified. 
		\item Let $|\sigma|$ denote the total variation measure. Then the total mass of $|\sigma|$ is bounded above by $L^B$ for some absolutely constant $B$. Moreover, with $|\sigma|^{(2)}=|\sigma|\star|\check{\sigma}|$, one has for any $\gamma>1/2$ 
		$$\int \|u\|^{-\gamma}d|\sigma|^{(2)}(u)\leq L^{-\kappa},\quad \text{for some }\kappa=\kappa(\gamma)>0$$
		\item $\varphi_1\star \sigma=\lambda_1 \varphi_1$ with $\lambda_1 \gg_{F,\epsilon} C(\pi_1)^{-\epsilon}$ for any $\epsilon>0$.
	\end{enumerate}
\end{lemma}

Denote by $\Phi_{2,u}=\varphi_2^u\overline{\varphi_2}$, $\Phi_{3,u}=\varphi_3\overline{\varphi_3^u} $ are elements in $\pi_2\otimes\tilde{\pi}_2$, $\pi_3\otimes\tilde{\pi}_3$. Apply the amplifier as in \cite[(3.9)]{HMN22}, we obtain that 
\begin{equation}\label{Amplification method}
	|\lambda_1|^2\left|\int_{\mathbf{X}}\varphi_1\varphi_2\varphi_3dg\right|^2\leq \int_{u\in \mathrm{Supp}(\sigma^2)}\langle \Phi_{2,u},\Phi_{3,u}\rangle d(\bar{\sigma}\star \check{\sigma})(u)
\end{equation}
which is essentially a finite sum in $u$. 
 
\subsubsection{Spectral decomposition}

We define the generic contribution in $\langle\Phi_{2,u},\Phi_{3,u}\rangle$ by
$$\langle\Phi_{2,u},\Phi_{3,u}\rangle_{\mathrm{gen}}=\int_{\pi\text{ gen}}\sum_{\psi\in \mathcal{B}(\pi)}\langle \Phi_{2,u},\psi\rangle\langle \bar{\psi},\Phi_{3,u}\rangle d\mu_\mathrm{pl}$$
where $\pi$ runs over generic representations of $\mathbf{G}(\mathbb{A})$,
$d\mu_{\mathrm{pl}}$ denotes the Plancherel measure on $\widehat{G}$,
and $\mathcal{B}(\pi)$ is an orthonormal basis of $\pi$.

When $\pi_3$ is cuspidal, we apply the usual spectral decomposition
$$\langle \Phi_{2,u},\Phi_{3,u}\rangle=\langle \Phi_{2,u},\Phi_{3,u}\rangle_{\mathrm{gen}}+\langle \Phi_{2,u},\Phi_{3,u}\rangle_{\mathrm{fin}}$$
where the finite contribution is given by
\begin{equation}\label{Final proof. finite contribution}
	\langle\Phi_{2,u},\Phi_{3,u}\rangle_\mathrm{fin}
	=\sum_{\chi^2=1}
	\langle\Phi_{2,u},\varphi_\chi\rangle
	\langle\varphi_\chi,\Phi_{3,u}\rangle
	\ll (C(\pi_2)C(\pi_3))^\epsilon |u|^{2\theta-1},
\end{equation}
as in \cite[\S3.4]{HMN22}.

When $\pi_3 = \chi_3 \boxplus \chi_3^{-1}$ is an Eisenstein series representation,
we employ the deformation method of \cite[(4.4)]{MV10} or \cite[\S3.7]{HMN22} and use
the regularized spectral decomposition from Theorem \ref{thm. Regularized Plancherel Formula}.
Set $\pi_3(t)=\chi_3|\cdot|^{it}\boxplus\chi_3^{-1}|\cdot|^{-it}$ to be the deformation,
and define $\varphi_3(t)=\mathrm{Eis}(f_3(t))$, where
$f_3(t)\in \pi_3(t)$ is the section whose restriction to the maximal compact subgroup
coincides with $f_3$.
Let $\Pi_{3,u}(t)=\pi_3(t)\otimes\tilde{\pi}_3(t/2)$ and define the vector
$$\Phi_{3,u}(t):(g_1,g_2)\mapsto \varphi_3(t)(g_1)\times \overline{u.\varphi_3(t/2)(g_2)}$$
In this case, there is no finite contribution, as noted in \cite[\S5.2.7]{MV10}. The regularized spectral decomposition then gives
$$\langle\Phi_{2,u},\Phi_{3,u}(t)\rangle_{\mathrm{reg}}=\langle\Phi_{2,u},\Phi_{3,u}(t)\rangle_\mathrm{gen}+\langle\Phi_{2,u},\mathcal{E}_{3,u}(t)\rangle_\mathrm{deg}$$
where $\mathcal{E}_{3,u}(t)=\mathrm{Eis}(\Phi_{3,u}(t)^*_N)$.
For sufficiently small $0<|t|<\delta_0$, the degenerate contribution satisfies
\begin{equation}\label{Final proof. degeneric contribution}
	\langle\Phi_{2,u},\mathcal{E}_{3,u}(t)\rangle\mathrm{deg}
	\ll (C(\pi_2)C(\pi_3))^\epsilon \|u\|^{2\theta-1},
\end{equation}
again by \cite[\S5.2.9]{MV10}.
The map $t\mapsto \langle\Phi_{2,u},\Phi_{3,u}(t)\rangle_{\mathrm{reg}}$ defines a continuous function on the neighborhood of $0\in i\mathbb{R}$. 

In both cases, we refer to these parts as the main term in the spectral decomposition.
Combining \eqref{Final proof. finite contribution} and
\eqref{Final proof. degeneric contribution} with Lemma \ref{amplifier} and the continuity at $t=0$,
we deduce that the main term contribution after amplification satisfies
\begin{equation}\label{main term contribution}
	(|\bar{\sigma}|\star|\check{\sigma}|)\big(u\mapsto
	|\langle \Phi_{2,u},\Phi_{3,u}\rangle_{\mathrm{main}}|\big)
	\ll (C(\pi_2)C(\pi_3))^{\epsilon}L^{-\kappa},
\end{equation}
where $\kappa$ is determined by $2\theta-1$.

\subsubsection{The generic contribution}

The remaining task is to give an upper bound for these generic contribution. For $\pi$ a generic representation of $\mathbf{G}$, we define the projection to $\pi$ is
$$\mathrm{Pr}_\pi(\Phi_{2,u},\Phi_{3,u})=\sum_{\psi\in \mathcal{B}(\pi)}\langle \Phi_{2,u},\psi\rangle\langle \bar{\psi},\Phi_{3,u}\rangle, $$
where $\psi$ runs over an orthogonal basis of $\pi$. 

\begin{proposition}
	
	With our choice of test vector, there exist some $A>0$ such that 
	\begin{equation}\label{Final proof. generic contribution}
		\int_{u\in \mathrm{Supp}(\sigma)} \langle\Phi_{2,u},\Phi_{3,u}\rangle_{\mathrm{gen}}\ d|\sigma|^{(2)}(u)\ll_\epsilon L^{A}\left(\frac{C(\pi_1\otimes\pi_2)}{C(\pi_2\otimes\pi_2)}\right)^{-\frac{1}{4}+\frac{1}{2}\vartheta+\epsilon}.
	\end{equation}
	
\end{proposition}

\begin{proof}
	
	It suffices to prove that
	$$\langle\Phi_{2,u},\Phi_{3,u}\rangle_{\mathrm{gen}}=\int_{\pi \text{ gen}}\mathrm{Pr}_\pi(\Phi_{2,u},\Phi_{3,u})d\mu_\mathrm{pl}(\pi)\ll \|u\|^{A}\left(\frac{C(\pi_1\otimes\pi_2)}{C(\pi_2\otimes\pi_2)}\right)^{-\frac{1}{4}+\frac{1}{2}\vartheta+\epsilon}.$$
	Recall that $\Phi_{3,u}=a(Q)\Phi_{3,u}^0$. Then
	$$\begin{aligned}
		|\mathrm{Pr}_\pi(\Phi_{2,u},\Phi_{3,u})|&\leq \sum_{\psi\in \mathcal{B}(\pi)}|\mathrm{Pr}_\pi(\Phi_{2,u})(\psi)||\mathrm{Pr}_\pi(\Phi_{3,u})(\psi)|\\
		&=\sum_{\psi\in \mathcal{B}(\pi)}|\mathrm{Pr}_\pi(\Phi_{2,u})(a(Q)\psi)||\mathrm{Pr}_\pi(\Phi^0_{3,u})(\psi)|
	\end{aligned}
	$$
 	where $\mathcal{B}(\pi)$ is an orthonormal basis of $\pi$. By Lemma \ref{lem. Sobolev norm. Integration by parts} and \cite[\S 2.4, S1b]{MV10}, for any integer $N\geq 0$, there exists an $N'\geq 0$ and $A\geq 0$, such that 
	$$\mathrm{Pr}_\pi(\Phi_{3,u}^0)(\psi)\ll_{N,\pi_3,\varphi_3^0} \|u\|^{A/2}\mathcal{S}_{N'}(\varphi^0_3\otimes\varphi_3^0)\mathcal{S}_{-N}(\psi)C(\pi)^{-N}.$$
	Applying the triple product formula \eqref{thm. Triple product formula}, we obtain
	$$|\mathrm{Pr}_\pi(\Phi_{2,u})(a(Q)\psi)|\ll \frac{L(\pi_2\otimes\pi_2\otimes\pi)^{1/2}}{L^*(\pi_2,\mathrm{Ad},1)L^*(\pi,\mathrm{Ad},1)^{1/2}}\prod_{v<\infty} (I_v^u)^{1/2} \prod_{v|\infty} (I_v^T)^{1/2}$$
	By the convexity bound and estimates for adjoint $L$-functions as in (\ref{Final proof. estimate for adjoint L functions}),
	$$\frac{L(\pi_2\otimes\pi_2\otimes\pi)^{1/2}}{L^*(\pi_2,\mathrm{Ad},1)L^*(\pi,\mathrm{Ad},1)^{1/2}}\ll C(\pi_2\otimes\pi_2)^{1/4+\epsilon}C(\pi)^{2+\epsilon}.$$
	For the local period integrals, the $u$ action contributes a factor $I_v^u\ll \|u\|^A$ as in \cite[(5.17)]{MV10}, while for archimedean $v$,
	Theorem \ref{Intro. lower bound} gives
	$$(I^T_v)^{1/2}\ll_{\pi_3,\epsilon}\|u\|^{A/2} C(\pi_1\otimes\pi_2)^{-\frac{1}{4}}\left(\frac{C(\pi_1\otimes\pi_2)}{C(\pi_2\otimes\pi_2)}\right)^{\frac{1}{2}\vartheta+\epsilon}\mathcal{S}_{d_0}(\psi)$$
	for some constant $d_0$. Collecting the bounds above, we obtain
	$$|\mathrm{Pr}_\pi(\Phi_{2,u})(\psi)||\mathrm{Pr}(\Phi_{3,u})(\psi)|\ll_{N,\pi_3,\varphi^0_3}\|u\|^A \left(\frac{C(\pi_1\otimes\pi_2)}{C(\pi_2\otimes\pi_2)}\right)^{-\frac{1}{4}+\frac{1}{2}\vartheta+\epsilon}\mathcal{S}_{d_0-N}(\psi)C(\pi)^{2-N+\epsilon}.$$
	Taking $N$ sufficiently large and applying the trace property in Lemma\ref{lem. Sobolev norm. trace property} together with the Weyl law in Lemma \ref{lem. Sobolev norm. Weyl law} for $\hat{G}_{\mathrm{gen}}$, we conclude
	$$\int_{\pi \text{ gen}}\sum_{\psi\in \mathcal{B}(\pi)}|\mathrm{Pr}_\pi(\Phi_{2,u})(\psi)||\mathrm{Pr}_\pi(\Phi_{3,u})(\psi)|d\mu_{\mathrm{pl}}\ll \|u\|^A\left(\frac{C(\pi_1\otimes\pi_2)}{C(\pi_2\otimes\pi_2)}\right)^{-\frac{1}{4}+\frac{1}{2}\vartheta+\epsilon},$$
	and the proposition follows. 

\end{proof}

Now we combine (\ref{Amplification method}) with (\ref{main term contribution}) and (\ref{Final proof. generic contribution}), the original period integral is then bouned by
$$\left|\int_{\mathbf{X}}\varphi_1(g)\varphi_2(g)\varphi_3(g)dg\right|^2\ll_\epsilon  L^{-\kappa}C(\pi_1\otimes\pi_2)^\epsilon+L^{A}\left(\frac{C(\pi_1\otimes\pi_2)}{C(\pi_2\otimes\pi_2)}\right)^{-\frac{1}{4}+\frac{1}{2}\vartheta+\epsilon}.$$
We can take $L$ to be a small power of $\frac{C(\pi_1\otimes\pi_2)}{C(\pi_2\otimes\pi_2)}$ and optimize our choice. Applying the upper bound to (\ref{Final proof. lower bound}), we conclude (\ref{Final proof. Main result}) for some $\delta>0$.

\bibliographystyle{alpha} 

\end{document}